\documentclass[12pt,a4paper]{amsproc}
\pdfoutput=1
\usepackage[T1]{fontenc}

\usepackage{dsfont}

\usepackage[margin=0.83in]{geometry}  %{0.5}
\usepackage{wrapfig}

\usepackage{amsmath,amssymb}
 \usepackage{amsthm}
\usepackage[mathscr]{euscript}
\usepackage{subfig}
\newtheorem{theorem}{Theorem}

\newtheorem{problem}{Design problem}
\newtheorem{lemma}{Lemma}

\newtheorem{example}{Example}

\newtheorem*{hypothesis}{\it{Induction hypothesis~for time~$i$}}

\usepackage[usenames,dvipsnames,svgnames,table]{xcolor}

\usepackage[utf8x]{inputenc} % utf8 encoding
\usepackage[T1]{fontenc} % use T1 fonts
\usepackage{bm} % bold math
\usepackage{color} % change text color

\newcommand*\xbar[1]{%
  \hbox{%
    \vbox{%
      \hrule height 0.5pt % The actual bar
      \kern0.35ex%         % Distance between bar and symbol
      \hbox{%
        \kern-0.1em%      % Shortening on the left side
        \ensuremath{#1}%
        \kern-0.1em%      % Shortening on the right side
      }%
    }%
  }%
}

\usepackage{tikz}
\usepackage{pgfplots}
\usepackage{tikz-3dplot}
\usetikzlibrary{decorations.pathmorphing} % for snake lines
\usetikzlibrary{matrix} % for block alignment
\usetikzlibrary{arrows} % for arrow heads
\usetikzlibrary{shapes} % for arrow heads
\usetikzlibrary{calc} % for manimulation of coordinates
\usetikzlibrary{positioning} % for manimulation of coordinates
\usetikzlibrary{patterns}

\tikzstyle{block} = [draw,rectangle,thick,minimum height=2em,minimum width=2em]
\tikzstyle{sum} = [draw,circle,inner sep=0mm,minimum size=4mm]
\tikzstyle{connector} = [->,thick]
\tikzstyle{line} = [thick]
\tikzstyle{branch} = [circle,inner sep=0pt,minimum size=1mm,fill=black,draw=black]
\tikzstyle{guide} = []
\tikzstyle{snakeline} = [connector, decorate, decoration={pre length=0.2cm,
                         post length=0.2cm, snake, amplitude=.4mm,
                         segment length=2mm},thick, magenta, ->]

\renewcommand{\vec}[1]{\ensuremath{\boldsymbol{#1}}} % bold vectors
\newtheorem{definition}{Definition}

\makeatletter    %% a hack
\pgfplotsset{
    /tikz/max node/.style={
        anchor=south,
    },
    /tikz/min node/.style={
        anchor=south,
        name=minimum
    },
    mark min/.style={
        point meta rel=per plot,
        visualization depends on={x \as \xvalue},
        scatter/@pre marker code/.code={%
            \ifx\pgfplotspointmeta\pgfplots@metamin
                \def\markopts{}%
                \coordinate (minimum);
                \node [min node] {
                };
            \else
                \def\markopts{mark=none}
            \fi
            \expandafter\scope\expandafter[\markopts,every node near
coord/.style=olive]
        },%
        scatter/@post marker code/.code={%
            \endscope
        },
        scatter,
    },
    mark max/.style={
        point meta rel=per plot,
        visualization depends on={x \as \xvalue},
        scatter/@pre marker code/.code={%
        \ifx\pgfplotspointmeta\pgfplots@metamax
            \def\markopts{}%
            \coordinate (maximum);
            \node [max node] {
                \pgfmathprintnumber[fixed]{\xvalue},%
                \pgfmathprintnumber[fixed]{\pgfplotspointmeta}
            };
        \else
            \def\markopts{mark=none}
        \fi
            \expandafter\scope\expandafter[\markopts]
        },%
        scatter/@post marker code/.code={%
            \endscope
        },
        scatter
    }
}
\makeatother    %% a hack

\usepackage{epstopdf}

\begin{document}

\title[Separated design for networked LQ control]{Separated design of encoder and controller for
\\ networked linear quadratic optimal control}
%\dedicatory{To the memory of Jan Willems whose works inspire us}
\author[M.~Rabi]{
  Maben Rabi}
\address[Maben~Rabi]{School of Computing and Electrical Engineering,
  Indian Institute of Technology,
  Mandi, India.}
\email{firstname@iitmandi.ac.in}
\author[C.~Ramesh]{
  Chithrupa Ramesh}
\address[Chithrupa~Ramesh]{School of Electrical Engineering,
  Royal Institute of Technology, Stockholm,
  Sweden.}
\email{firstname.lastname@ee.kth.se}
\author[K.H.~Johansson]{
  Karl Henrik Johansson}
\address[Karl~Johansson]{School of Electrical Engineering,
  Royal Institute of Technology,
  Stockholm, Sweden.}
\email{firstname.lastname@ee.kth.se}
\date{\today}

%\thanks{M.R. acknowledges financial support from the Chalmers Foundation.}%
%\thanks{C.R. and K.H.J. are supported by the Swedish Research Council and
% \markboth{Draft. Please do not circulate.}{Draft. Please do not circulate.}
%%%%%%%%%%%%%%%%%%%%%%%%%%%%%%%%%%%%%%%%%%%%%%%%%%%%%%%%%%%%
%%%%%
%%%%%

\begin{abstract}
For a networked control system, we consider the problem of encoder and
controller design. We study a discrete-time linear plant with a finite
horizon performance cost, comprising of a quadratic function of the
states and controls, and an additive communication cost. We study
separation in design of the encoder and controller, along with related
closed-loop properties such as the dual effect and certainty
equivalence. We consider three basic formats for encoder outputs: quantized samples, real-valued samples at event-triggered times, and real-valued samples over additive noise channels. If the controller and encoder are dynamic, then we show that the performance cost is minimized by a separated design: the controls are updated at each time instant as per a certainty equivalence law, and the encoder is chosen to minimize an aggregate quadratic distortion of the estimation error. This separation is shown to hold even though a dual effect is present in the closed-loop system. We also show that this separated design need not be optimal when the controller or encoder are to be chosen from within restricted classes.
\end{abstract}
%%%%%%%%%%%%%%%%%%%%%%%%%%%%%%%%%%%%%%%%%%%%%%%%%%%%%%%%%%%
\maketitle
%\tableofcontents
%%%%%%%%%%%%%%%%%%%%%%%%%%%%%%%%%%%%%%%%%%%%%%%%%%%%%%%%%%%%
%%%%%
%%%%%
%%%%%%%%%%%%%%%%%%%%%%%%%%%%%%%%%%%%%%%%%%%%%%%%%%%%%
%%%%%% Section 1
%%%%%%%%%%%%%%%%%%%%%%%%%%%%%%%%%%%%%%%%
%\section{Control over a rate-limited channel}
\section{Introduction}
%\subsection{Motivation}
We consider discrete-time sequential decision problems for a control
loop that has a communication bottleneck between the sensor and the
controller~(Figure~\ref{simpleLoop}). The design problem is to choose in
concert an encoder and a controller. The encoder maps the sensor's raw
data into a causal sequence of channel inputs. Depending on the channel
model adopted in this paper, the encoder performs either sequential
quantization, sampling, or analog companding. The controller maps channel outputs into a causal sequence of control inputs to the plant. Such two-agent problems are generally hard because the information pattern
is non-classical, as the controller has less information than the sensor~\cite{Witsenhausen1971}. This gives scope for the controller to exploit any dual~effect present in the loop, even when the plant is linear~\cite{feldbaum1960dualEffectPaperOne}. These two-agent problems are at the simpler end of a range of design problems arising in networked control systems~\cite{borkarMitterTatikonda2001markovControl,baillieulAntsaklis2007,goodwinSilvaQuevedo2008cdcSurvey,andrievskyMatveevFradkov2010survey}. Naturally, one seeks formulations of these design problems as stochastic optimization problems whose solutions are tractable in some suitable sense.

The classical partially observed linear quadratic Gaussian (LQG) optimal
control problem is a one-agent decision
problem~\cite{wonham1968onTheSeparationTheorem}. Given a linear,
Gauss-Markov plant, one is asked for a causal controller, as a function
of noisy linear measurements of the state, to minimize a quadratic cost
function of states and controls. This problem has a simple and explicit
solution, where the optimal controller `separates' into two policies;
one to generate a minimum mean-squared error estimate of the state from
the noisy measurements, and the other to control the fully observed
Gauss-Markov process corresponding to the estimate. A networked version
of this problem is the following two-agent LQG optimal control
problem~\cite{borkarMitter1997}. Given a linear Gauss-Markov plant and a
channel model, one is asked for an encoder and controller to minimize a
performance cost which is a sum of a communication cost and a quadratic cost on states and controls. The communication cost is charged on decisions at the encoder, which are chosen to satisfy constraints imposed by the channel model. No causal encoding or control policies are, in general, excluded from consideration.
 As in the one-agent version, a certain `separated' design is optimal,
as has been suggested
 in various settings since the
sixties~\cite{larson1967mistaken,sauerMelsa1974continuouslyVariableObservationCosts,fischer1982mistaken,bansalBasar1989automatica,mitter2001ejc,tatikondaSahaiMitter2004,matveevSavkinSystemsAndControlLetters2004,wuArapostathis2005acc,nairEvans2007,leiBaoSkoglundKalle2011,molinHirche2013,yuksel2014tac}.
Precisely, the following combination is optimal: certainty equivalence
controls with a minimum mean-squared estimator of the state, and an
encoder that minimizes a distortion for state estimation at the
controller. The distortion is the average of a sum of squared estimation
errors with time-varying coefficients depending on the coefficients of
the performance cost.
This separation is different from that obtained in
the classical LQG problem, but it is still due to a
linear evolution of the state, and the statistical independence of
noises from all other current and past variables. As in the classical
one-agent version~\cite{shaw1965optimumStochasticControl,jamesRoot1969nonGaussianLQ}, the random variables need not be Gaussian.
 \begin{figure}
 \centering
\includegraphics[width=0.9\textwidth]{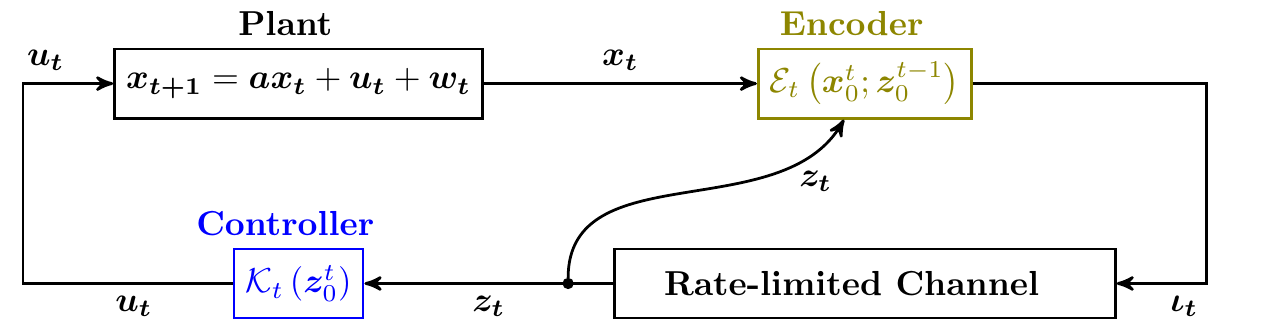}
 \caption{Control over a rate-limited channel with perfect feedback}
 \label{simpleLoop}
\vspace{-8mm}
 \end{figure}
%%%%%%%%%%%%%%%%%%%%%%%%%%%%%%%%%%%%%%%%%%%%%%%%%%%%%%%%%%%%%%%%%%%%%%%%%%%%%%%%%%%%%%%%%%%%%%%%%%%%%%%%%%%%%%%%%%%%%%%%%%%%%
\subsection{Previous works}
In the long history of the two-agent networked LQG problem,
different channel models have been treated, leading to different types of
encoders. We find in these works that the encoder is either a
quantizer, an analog time-dependent compander, or an event-based sampler.
%But there is a common theme in this literature.
%For our design problem~\ref{fullyDynamicControls},
% what we call the dynamic encoder-controller design problem,
%several authors suggest that it is optimal to apply a combimation of:
%separated design of encoder, and certainty equivalence control.

%When one treats a discrete alphabet channel, one has to treat the
%encoder as a time-dependent quantizer.
%Quantized control has been explored since the sixties, and structural results for this problem have seen spirited discussions over the years~\cite{larson1967mistaken,marleau1972,fischer1982mistaken}. This problem was revisited by Borkar and Mitter in recent years~\cite{borkarMitter1997}, setting off a new wave of interest. A survey can be found in~\cite{nairEvans2007,minyueFu2012}.
%
%When one treats an additive noise channel, one has to treat the encoder as a
%time-dependent, possibly non-linear, compander. The corresponding
%networked LQG problem has been studied
%in~\cite{bansalBasar1989automatica}, and more recently
%in~\cite{freudenbergMiddletonBraslavsky2011}.
%When one treats analog channels with channel use restrictions, one has
%to treat the encoder as an event-triggered sampler~\cite{astrom-bernhardsson-cdc}.
%The networked LQG problem for event-triggered sampling is studied
%in~\cite{molinHirche2013}.

When a discrete alphabet channel is treated, the encoder is a time-dependent quantizer. Quantized control has been explored since the sixties, and structural results for this problem have seen spirited discussions over the years~\cite{larson1967mistaken,marleau1972,fischer1982mistaken}. This problem was revisited by Borkar~et~al.~\cite{borkarMitter1997} in recent years, setting off a new wave of interest. Surveys can be found in~\cite{nairEvans2007,minyueFu2012}. For an additive noise channel, the encoder is a time-dependent, possibly non-linear, compander. The corresponding networked LQG problem has been studied in~\cite{bansalBasar1989automatica}, and more recently in~\cite{freudenbergMiddletonBraslavsky2011,gatsis2014}. Analog channels with channel use restrictions lead to an encoder being an event-triggered sampler~\cite{astrom-bernhardsson-cdc}. The networked LQG problem for event-triggered sampling is studied in~\cite{molinHirche2013}.

The above papers suggested separated designs for the two-agent LQG problem with dynamic encoder and controller, and certainty equivalence controls. This is despite other results~\cite{curry1970book,fengLoparo1997activeProbing},
confirming the dual effect in the two-agent networked control problem.
Thus, there can be an incentive to the controller to influence the
estimation error, and yet the optimal controller chooses to ignore this
incentive. Furthermore, for the two-agent LQG problem with
event-triggered sampling, and with zero order hold control between
samples, Rabi~et~al.~\cite{iwsm09} showed through numerical computations
that it is suboptimal to apply controls affine in the minimum mean
square error~(MMSE) estimate.
The optimal controls are nonlinear functions of the received samples.
Thus, the literature does not tell us when separation
holds, and when it does not, for the general class of two-agent
problems.
\subsection{Our contributions}
We make three main contributions. Firstly, we show that for the combination of a linear plant and nonlinear
encoder,
% ~(sensor),
the dual effect is present. This confirms
the results of Curry and
others~\cite{curry1970book,fengLoparo1997activeProbing}, by establishing
through a counter example that there is a dual effect in
the closed-loop system. In fact, each of the three models we allow for
the channel endow the loop with the dual effect. The dual role of the
controller lies in reducing the estimation error in the future, using
the predicted statistics of the future state and knowledge of the
encoding policy. Due to this dual role, we show that, in general,
separated designs need not be optimal for linear plants with non-linear
measurements, even with independent and identically distributed~(IID) Gaussian noise and quadratic costs.
Examples~\ref{Ex:ConstrainedEncoder_NoSep} and~\ref{Ex:ConstrainedControl_NoSep} show instances where the dual effect matters.
Example~\ref{example:predictionErrorDependence} shows how the dual effect in the two-agent networked LQ problem renders useless the techniques that work for the classical, single-agent, partially observed LQ problem.
These examples illustrate the insufficiency of arguments offered
in~\cite{larson1967mistaken,sauerMelsa1974continuouslyVariableObservationCosts,fischer1982mistaken,bansalBasar1989automatica,mitter2001ejc,tatikondaSahaiMitter2004,matveevSavkinSystemsAndControlLetters2004,wuArapostathis2005acc,nairEvans2007,leiBaoSkoglundKalle2011,molinHirche2013,yuksel2014tac}
 for the optimality of separation and certainty equivalent controls.

Our second contribution is a proof for separation in one specific design problem.
We prove that for the dynamic encoder-controller design problem,
it is optimal to apply separation and certainty equivalence. A key instrument in our proof is the class of `controls-forgetting encoders' (introduced in section~\ref{controlsForgettingSection}) which we show to be optimal despite it being a strict subset of the general class of state-based encoders.  We also
notice that the result holds under a variety of schemes for charging
communication costs. For example, it holds even when the encoder is an
analog compander with hard amplitude limits. Our proof does not require
the dual effect to be absent. Hence there is no contradiction with the
fact separation and certainty equivalence are not optimal for other
design problems concerning the same plant-sensor combination. Our work
also provides a direct insight to explain separation or the lack of it,
in the form of a property of the optimal cost-to-go function~{  (Example~\ref{Ex:CertaintyEquivalence_EncDes} in Section~\ref{Examples})}.
Furthermore, we show that when this property does not hold separation is no longer optimal.

Our third contribution points out some subtleties
that arise when dynamic policies are involved.
We explicitly demonstrate that with dynamic encoders for LQ optimal control,
one cannot extend and apply a result of Bar-Shalom and
Tse~\cite{barShalomtse1974dualEffectPaper} which mandates absence of dual effect for certainty equivalence
to be optimal.
The classical notion of a dual effect was introduced for static measurement
policies, and the dual role of the controls has been motivated through
the notion of a probing incentive~\cite{feldbaum1960dualEffectPaperOne}.
We ask if the concept of probing applies unchanged for dynamic measurement
policies and point out some subtleties in answering this question.

In recent years, there has been a resurgence of interest in problems
related to dynamic and decentralized decision making in stochastic
control. Old problems and results have been reexamined and reinterpreted
to find new insights and develop new methods, such as the common
information approach~\cite{Mahajan2009,Nayyar2011}. Others, such
as~\cite{Kulkarni2012}, have sought to reinterpret the proof techniques
used in~\cite{bansal1987}. Following in the path
of~\cite{Witsenhausen1968}, many new counterexamples have been
identified that show optimality of nonlinear strategies for control
problems under non-classical information
patterns~\cite{Lipsa2011,Zaidi2013}. Similarly, drawing from the many
works on two-agent networked LQG
problems~\cite{curry1970book,fengLoparo1997activeProbing,borkarMitter1997,nairEvans2007,minyueFu2012},
we have sought to understand why a structural simplification can be
found in some dynamic decision problems, despite the non-classical
information pattern and the consequent presence of a dual effect.

\subsection{Outline}
The remainder of the paper is organized as follows. In
Section~\ref{twoAgentProblemStatement}, we present a basic problem
formulation, pertaining to encoder and controller design for data-rate
limited channels. In Section~\ref{definitionOfDEandCE}, we discuss the
notion of a dual effect and certainty equivalence, and present a
counterexample to establish that there is a dual effect in the
considered networked control system. In Section~\ref{proofSection}, we
present a proof for separation in the two-agent networked LQG problem.
In Section~\ref{extensions}, we extend our results to other channel
models, including event-triggered samples and additive noise channels.
In Section~\ref{Examples}, we present a number of examples to
illustrate that in general, separation does not hold for constrained design
problems, followed by the conclusions in Section~\ref{discussionSection}.

%%%%%%%%%%%%%%%%%%%%%%%%%%%%
%      Problem formulation
%%%%%%%%%%%%%%%%%%%%%%%%%%%%
\section{Problem formulation%
\label{twoAgentProblemStatement}
}
In this section, we describe a version of the two-agent networked LQG problem, corresponding to a rate-limited channel model. We consider an instantaneous, error-free, discrete-alphabet channel and the logarithm of the size of the alphabet is the bit rate. A control system that uses such a channel to communicate between its sensor and controller is depicted in Figure~\ref{simpleLoop}, and comprises of four blocks. Each of these blocks, along with the performance cost, are described below, followed by a description of the design problems under consideration.

\subsection{Plant}
The plant state process $\left\{x_t\right\}$ is scalar, and its evolution law is linear:
\begin{align}
    x_{t+1}
&
    =
     ax_t + u_t + w_t,
\label{plant}
\end{align}
for $ 0 \leq t \le T.$
Here $\left\{u_t\right\}$ is the controls process, and
$\left\{w_t\right\}$ is the plant noise process, which is a sequence of independent random variables with constant variance~$\sigma_w^2$, and zero means.
%
%  but possibly
%  time-varying means~$\left\{\mu_t\right\}$.
%
The initial state~$x_0$ has a distribution with mean~$\overline{x}_0$ and variance~$\sigma_0^2$. At any time $t$, the noise~$w_t$ is independent of all state, control, channel input, and channel output data up to and including time $t$. We assume that the state process is perfectly observed by the sensor.
\subsection{Performance cost}
The performance cost is a sum of the quadratic cost charged on states and controls, and a communication cost charged on encoder decisions:
\begin{align}
    J %\left(\mathcal{E}^{T-1}_0, \mathcal{K}_0^{T-1} \right)
&
   =
   \mathbb{E}
   \left[  x^2_{T+1} +  p \sum^{T}_{i=1} x^2_i + q \sum^{T}_{i=0} u^2_i
   \right]
   +
%   \mathbb{E}
%   \left[
    J^{\textrm{Comm}}
%   \right] ,
\label{objectiveFunction}
\end{align}
where $p > 0$ and $q > 0$ are suitably chosen scalar weights for the
squares of the states and controls, respectively. The communication cost
$J^{\textrm{Comm}}$ is an average quantity that depends on the encoding and
control policies, and the channel model adopted.
%%%%%%%%%%%%%%%%%%%%%%%%%%%%%%%%%%%%%%%%%%%%%%%%%%%%%%%%%%%%%%%%
%%%%%%%%%          Channel model
%%%%%%%%%%%%%%
\subsection{Channel model%
%and communication cost%
\label{discreteChannelFixed}
}
The channel model refers to an input-output description of the
communication link from the sensor to the controller. We denote the
channel input at time $t$ by $\iota_t$, the corresponding output
by~$z_t$, and the encoding map generating~$\iota_t$ by $\mathcal{E}_t$.
In Figure~\ref{simpleLoop}, we consider an ideal, discrete alphabet
channel that faithfully reproduces inputs, and thus, $\iota_t = z_t \
\forall t$. The encoder's job is to pick at every time~$t$, the encoding map~$\mathcal{E}_t$ producing a channel output letter from the pre-assigned finite alphabet
$
   z_t
   \in
   \{1, \ldots, N \}, \forall \ t ,
$
where the non-negative integer $N$ is the pre-assigned size of the channel alphabet. Since the alphabet is fixed, we have a hard data-rate constraint at every time. Hence there is no explicit cost attached to communication, so~$J^{\textrm{Comm}} \equiv 0$ in this case. In Section~\ref{extensions}, we consider other channel models that permit the data-rate or energy needed for each transmission to be chosen causally by the encoder.
%%%%%%%%%
%%%%%%  Admissible policies
%%%%%%%%%
\subsection{Controller\label{dataBasisAtController}}
The control signal $u_t$ is real valued and is to be computed by a causal policy based on the sequence of channel outputs. The controller has perfect memory, and thus remembers all of its past actions, and the causal sequence of channel outputs. Thus, in general, at every time~$t$ the controller's map takes the form:
\begin{align*}
   \mathcal{K}_t
&
   \ : \
   \Bigl\{
       t,
       \left\{ z_i\right\}^t_0, \left\{
      u_i\right\}_0^{t-1}
   \Bigr\}
  \mapsto
   u_t.
%   \in {\mathbb{R}}.
\end{align*}
%%%%%%%%%%%%%%%%%%%%%%%%%%%%%%%
\subsection{Encoder}
At all times, the encoder knows the entire set of control policies
employed by the controller and the statistical parameters of the plant.
With this prestored knowledge, the encoder works as a causal quantizer
mapping the sequence of plant outputs. Thus, % in general, at every time~$t$,
the encoder's map takes the form:
\begin{align*}
    \mathcal{E}_t
&
   \ : \
   \left\{
         t, \left\{ x_i \right\}^t_0, \left\{ z_i\right\}^{t-1}_0,
            \left\{ {\mathcal{K}}_i \left( \cdot \right)
            \right\}^{t-1}_0
   \right\}
   \mapsto
   \ z_t . % \in  \left\{ 1, \ldots, N \right\} .
\end{align*}
Notice that we do not allow the encoder to directly view the sequence of
inputs to the plant. This subtle point plays an important role in the
examples we present in Section~\ref{discussionSection}.
%%%%%%%%%%%%%%%%%%%%%%%%%%%%%%%%%%%%%%
\subsection{Design problems}
For a given information pattern, different design spaces may arise due to engineering heuristics, hardware or software limitations, etc. Any such design space is a subset of the set of all admissible encoder and controller pairs. We identify four design problems, each associated with its own design space. For these design problems, an adopted channel model can be either the one described in Section~\ref{discreteChannelFixed}, or any of the models from Section~\ref{extensions}. First, we pose a single-agent design problem which has a classical information pattern. %But the loop might have the dual effect.
\begin{problem}[Controller-only Design]
\label{preassignedEncoders}
For the linear plant~(\ref{plant}), the adopted channel model, and a
given admissible set of encoding policies:
\begin{gather*}
 \left\{
    \mathcal{E}_t^{\dagger}
        \left( \, \vec{\cdot}  \, \, ;  \,
               \left\{ z_i \right\}_0^t,
%               \left\{ \xi_i  \left( \cdot \right)  \right\}_0^{t-1},
               \left\{ u_i \right\}_0^{t-1} %  = {0}_0^{t-1}
        \right)
 \right\}_0^T \; ,
\end{gather*}
the controller-only design problem requires one to pick a causal sequence of control policies $\{\mathcal{K}_t\}_0^T$ to minimize the performance cost~(\ref{objectiveFunction}).
\end{problem}
Next we pose a design problem where the design space is the largest possible non-randomized set of admissible encoder-controller pairs. We consider every causally time-dependent encoder and controller. %Thus the controller may freely vary the control signal every time tick, respecting of course causality.
In other words, for this type of design problem, regardless of the choices one makes for channel and communication cost, at any time, the controller can update the control signal using all of the channel outputs up till then. %Thus there is no direct penalty for, or ban against using the current as well as earlier  channel outputs to compute the latest controls.
%%%%%%%%%%%%%%%%%%%%%%%%%%%%%%%%%%%%%%%%%%%%%%%%%%%%%%%%%%%%%%%%%%%%%%%%%%
\begin{problem}[Dynamic Encoder-Controller Design]
\label{fullyDynamicControls}
For the linear plant~(\ref{plant}) and the adopted channel model, the dynamic encoder-controller design problem requires one to pick causal sequences of encoding and control policies $\{\mathcal{E}_t\}_0^T,\{\mathcal{K}_t\}_0^T$ to minimize the performance cost~(\ref{objectiveFunction}).
\end{problem}
%%%%%%%%%%%%%%%%%%%%%%%%%%%%%%%%%%%%%%%%%%%%%%%%%%%%%%%%%%%%%%%%%%%%%%%
Next we pose a design problem where the controller and encoder must respect a restriction on selecting the control signals or encoding maps. At every time, the control values must be chosen from a restricted set $\mathcal{U}$, such as the interval $(-1,1)$ or the finite set $\{-1,0,1\}$. Likewise, the encoding maps have to be chosen from within restricted sets. For example, the encoding maps may be constrained to consist of two quantization cells $(-\infty,\theta), (\theta,\infty)$, where the encoder threshold $\theta$ must be chosen from a restricted set $\Theta$, say the interval $(-5,5)$. Subject to these constraints, the controller and encoder policies are still to be dynamically chosen.

\begin{problem}[Constrained Encoder-Controller Design]
\label{RestrictedControlsEncoders}
For the linear plant~(\ref{plant}), and the adopted channel model,
%  channel~(Section~\ref{discreteChannelFixed}),
the constrained encoder-controller design problem requires one to pick causal sequences of encoding and control policies $\{\mathcal{E}_t\}_0^T, \{\mathcal{K}_t\}_0^T$, subject to the constraints represented by $\theta \in \Theta$ and $u_k \in \mathcal{U}$, to minimize the performance cost~(\ref{objectiveFunction}).
\end{problem}
%%%%%%%%%%%%%%%%%%%%%%%%%%%%%%%%%%%%%%%%%%%%%%%%%%%%%%%%%%%%%%%%%%%%%%%%

Next we pose a design problem where the controller must respect not only the information pattern in the dynamic encoder-controller design problem~(Design problem~\ref{fullyDynamicControls}), but must also respect a restriction on updating controls. Basically, the control waveform is generated in a piece-wise `open-loop' way, while epochs and encoding maps are picked using dynamic policies. Let $\epsilon_0, \epsilon_1 \ge 1$, be two random integers
% dividing the horizon $\left\{ 0, \ldots , T \right\} $.
such that $\epsilon_0 + \epsilon_1 = T+1$. Then the two epochs are
$\left\{ 0, \ldots , \epsilon_0 -1 \right\}$ and $\left\{ \epsilon_0 ,
\ldots , T \right\}$. These epochs are chosen by the controller
respecting the inequalities: $1 \le \epsilon_0 < T+1 $ and $ \epsilon_1 =  T+1 - \epsilon_0$, and hence have to be adapted to all the data available at the controller. %This means that for  $0 \leq t \le T$, the event:
Within an epoch, the controller must pick controls depending only on
data at the start of the epoch. Precisely, given the condition that $t < \epsilon_0$, and given the initial observation~$z_0$, the controls $u_t$ must be a fixed function of $\left( t, z_0 \right)$ regardless of the data $\left\{ z_1, \ldots , z_t \right\}$.
\begin{problem}[Hold-Waveform-Controller and Encoder Design]
\label{ZOHcontrols}
For the linear plant~(\ref{plant}), and the adopted channel model, the
hold-waveform-controller and encoder design problem is to pick a causal sequence of encoding polices $\{\mathcal{E}_t\}_0^T$ in concert with a causal sequence of policies for epochs and controls to minimize the performance cost~(\ref{objectiveFunction}). The controls are restricted to depend on the controller's data in the specific form:
\begin{align*}
u_t & =
       \begin{cases}
               {\mathcal{K}}^0_t   \Bigl(
                               z_0
                                \Bigr)
                & \text{for}~ 0 \le t \le \epsilon_0 -1,~\text{and},
               \\
               {\mathcal{K}}^1_t   \Bigl(  \left\{
                               z_i\right\}^{\epsilon_0 - 1}_0, \left\{
                               u_i\right\}_0^{\epsilon_0 - 1}
                                \Bigr)
                & \text{for}~\epsilon_0 \le t \le T .
       \end{cases}
\end{align*}
\end{problem}
\noindent
A special case of a hold-waveform controller is that of zero order
hold~(ZOH) control where an additional restriction forces
the control waveform be held constant over each epoch.

For all four design problems presented above, we assume the existence of measurable policies minimizing the associated costs. We avoid investigating the necessary technical qualifications except to say that if need be, one may allow randomized polices, or even reject the class of merely measurable policies in favour of the class of universally measurable policies~\cite{bertsekasShreve1978book}.
% In the case of design problem~\ref{fullyDynamicControls}, we succeed
% in explicitly describing the minimizing sets of policies.
%%%%%%%%%%%%%%%%%%%%%%%%%%%%%%%%%%%%%%%%%%%%%%%%%%%%%%%%%%%%%%%
%%%%%%%%%%%
%%%%%%%%%%      Definitions of DUAL EFFECT and CE
%%%%%%%%%%
%%%%%%%%%%%%%%%%%%%%%%%%%%%%%%%%%%%%%%%%%%%%%%%%%%%%%%%%%%%%%%%
%\section{Definitions of dual effect, certainty equivalence%
\section{Dual effect and certainty equivalence%
\label{definitionOfDEandCE}
}
We begin by presenting a definition of dual
effect~\cite{feldbaum1960dualEffectPaperOne} and certainty
equivalence~\cite{josephTou1961separation}. We then present an example to establish that there is a dual effect of the controls in the networked control system introduced in Section~\ref{twoAgentProblemStatement}.
\subsection{Dual effect}
%%%%%%%%%%%%%%%%%%%%%%%%%%%%%%%%%%%%%%%%%%%%%%%%%%%%%
\begin{figure}
\centering
\includegraphics[width=0.7\textwidth]{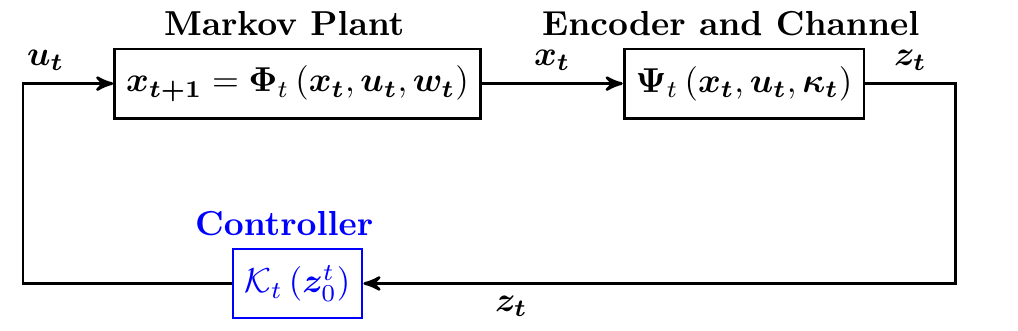}
\caption{Setup for definitions in Section~\ref{definitionOfDEandCE}}
\label{blockDiagramForDualEffectDefinition}
\end{figure}
%%%%%%%%%%%%%%%%%%%%%%%%%%%%%%%%%%%%%%%%%%%%%%%%%%%%%
In a feedback control loop,
the dual effect is an effect that the
controller may see in the rest of the loop. When it is present, the
control laws affect not just the first moment, but also second, third and
higher central moments of the controller's nonlinear filter for the
state. Below, we state this formally for a controlled Markov
process with partial
observations available to the controller:
\begin{gather}
x_{t+1} \  = \Phi_t \Bigl( x_t, u_t , w_t\Bigr),
\ %  \text{and,}
 \ \ \ \ \ \
z_t  =  \Psi_t \Bigl( x_t, u_t , \kappa_t\Bigr),
\label{generalMarkov}
\end{gather}
where the sequences $\left\{ x_t \right\}$ and  $\left\{ u_t \right\}$
are the real-valued plant state and control processes, respectively, see Figure~\ref{blockDiagramForDualEffectDefinition}. The
sequence~$\left\{ z_t \right\}$ is the observation process and the
sequences $\bigl\{ w_t\bigr\}$ and $\bigl\{ \kappa_t \bigr\}$ are the plant
noise and observation noise processes, respectively. Assume that all
the primitive random variables are defined on a suitable probability
triple, $\left[\Omega, \mathcal{F}, \mathcal{P} \right]$. Now, consider
two arbitrary admissible sets of control policies:
$
\Bigl\{ {\mathcal{K}} (t, \cdot ) \Bigr\},
\Bigl\{ \widetilde{\mathcal{K}} (t, \cdot ) \Bigr\}$.
Once we pick one such set of control policies, they together with the
measure~$\mathcal{P}$ define the states, observations and controls as
random processes. The choice of policies fixes their statistics.
We can advertise this relationship by (1)~specifying random variables, $x_t$
for example,  in the form $x_t\left( \vec{\omega} ; \mathcal{K} \right)$,
(2)~specifying a filtration, for example, the one generated by the
$z$-process as~${\mathcal{F}}^{{\mathcal{K}} ,z}$, or
(3)~specifying an  expected value of a functional,
$
     \mathbb{E} \left[  F_t \right]
$
for example, in the form
\begin{gather*}
\mathbb{E}_{\mathcal{P}, \mathcal{K}} \left[
F_t \Bigl(  t,
   \left\{ x_i \left( \vec{\omega} ; \mathcal{K} \right) \right\}_0^t,
   \left\{ z_i \left( \vec{\omega} ; \mathcal{K} \right) \right\}_0^t ,
   \left\{ u_i \left( \vec{\omega} ; \mathcal{K} \right) \right\}_0^t
 \Bigr) \right],
\end{gather*}
where~$\vec{\omega}$ stands for any element of the sample space of the
primitive random variables. To minimize the notational burden, we
advertise the dependence on the set of control policies only as needed.
We now define the dual effect by defining its absence.
%%%%%%%%%%%%%%%%%%%%%%%%%%%%%%%%%%%%%%%%%%%%%%%%%%%%%%%
\begin{definition}[Dual effect]
The networked control system in Figure~\ref{blockDiagramForDualEffectDefinition} is
said to have {\it{no dual effect of second-order}} if
\begin{enumerate}
    \setlength{\itemsep}{0pt}%
    \setlength{\topsep}{0pt}
    \setlength{\partopsep}{0pt}
    \setlength{\parsep}{0pt}
    \setlength{\parskip}{0pt}%
\item{for any two sets $\mathcal{K}, \widetilde{\mathcal{K}}$ of
      admissible control policies, and }
\item{for any two time instants~$t, s$,}
\end{enumerate}
we have
$\mathcal{F}_t^{\mathcal{K}, z} = \mathcal{F}_t^{\widetilde{\mathcal{K}},
z} $ for every~$t$, %be the same; denote both of them by $ \mathcal{F}_t^{z}$. Secondly,
and that for any given event $ X \in \mathcal{F}_t^{{\mathcal{K}} ,z}$,
\begin{multline*}
\setlength{\multlinegap}{1pt}
\mathbb{E}_{\mathcal{P}, \mathcal{K}} \left[
{
\Bigl(x_t ( \vec{\omega} ; \mathcal{K} )
 - \mathbb{E}_{\mathcal{P}, \mathcal{K}}
                \bigl[ x_t ( \vec{\omega} ; \mathcal{K} )
         \bigl|
        \bigl\{ z_i( \vec{\omega} ; \mathcal{K} )\bigr\}_0^s
        , \vec{\omega} \in X
         \bigr.
                \bigr]
\Bigr)
}^2
\Bigl|
        \bigl\{ z_i ( \vec{\omega} ; \mathcal{K} )
   \bigr\}_0^s
        , \vec{\omega} \in X
\Bigr.
\right]
=
\\
\mathbb{E}_{\mathcal{P}, \widetilde{\mathcal{K}}} \left[
{
\Bigl(x_t ( \vec{\omega} ;
                               \widetilde{\mathcal{K}} )
 - \mathbb{E}_{\mathcal{P}, \widetilde{\mathcal{K}}}
                \bigl[ x_t ( \vec{\omega} ;
                              \widetilde{\mathcal{K}} )
         \left|
        \bigl\{ z_i ( \vec{\omega} ; \widetilde{\mathcal{K}} )
        \bigr\}_0^s
        , \vec{\omega} \in X
         \right.
                \bigr]
\Bigr)
}^2
\left|
        \bigl\{ z_i ( \vec{\omega} ;
                              \widetilde{\mathcal{K}}
) \bigr\}_0^s
        , \vec{\omega} \in X
\right.
\right].
\end{multline*}
\end{definition}
Thus, we require equality of the two sets of covariances of
filtering/prediction/smoothing errors, corresponding to any two choices
of control strategies. In the definition above, by choosing one set of
control policies, say $\widetilde{\mathcal{K}}$ as resulting in $u_t =
0$, for all $t$, we obtain the definition of Bar-Shalom and Tse~\cite{barShalomtse1974dualEffectPaper}.
%%%%%%%%%%%%%%%%%%%%%%%%%%%%%%%%%%%%%%%%%%%%%%%%%%%%%%%%%%%%%%%

%%%%%%%%%%%%%%%%%%%%%%%%%%%%%%%%%%%%%%%%%%%%%%%%%%%%%%%%%%%%%%%%%%%%%%%%%
\subsection{Certainty equivalence}
For the controlled Markov process~(\ref{generalMarkov}), consider the general cost
\begin{align*}
J^{\textrm{general}} & =
\mathbb{E}\left[
         L \Bigl( \left\{ x_i \right\}_1^{T-1},  \left\{ u_i
              \right\}_0^T  \Bigr)
          \right],
\end{align*}
where $L$ is a given non-negative cost function.
Imagine that a muse could at
time~$t$ supply to the controller the exact values of all
primitive random variables by informing the controller the exact
element~$\vec{\omega}$ of the sample space~$\Omega$.
With such complete and acausal information, the controller
could, in principle, solve the deterministic optimization problem
\begin{align*}
\inf_u J_t
 \left( u ; \vec{\omega} \right)
 & =
\inf_u
         L \Bigl( \left\{ x_i \left( \vec{\omega} \right)
                    \right\}_0^T,
       \left\{ u_i  \left( \vec{\omega} \right)
              \right\}_0^{t-1},
        \,  u \, ,
       \left\{ u_i  \left( \vec{\omega} \right)
              \right\}_{t+1}^{T}
  \Bigr).
\end{align*}
Let $u_t^*\left( \vec{\omega} \right)$ be an optimal control law for
this deterministic optimization problem. We now state the definition
of certainty equivalence from van~der~Water and
Willems~\cite{vanDerWaterWillems1981}:
\begin{definition}%[certainty equivalence Control Law]
\label{ceDefinition}
A certainty equivalence control law for the plant~\eqref{plant} with the performance
cost~\eqref{objectiveFunction} has the form:
\begin{gather*}
\mathbb{E}\left[
     u_t^*\left( \vec{\omega} \right)
  \bigl|
         \left\{ z_i \left( \vec{\omega} \right)
                   \right\}_0^t,
         \left\{ u_i \left( \vec{\omega} \right)
                   \right\}_0^{t-1}
  \bigr.
         \right].
\end{gather*}
\end{definition}
\noindent
Clearly, this law is causal. Notice also that its form is tied to the performance cost, and to the statistics of the state and observation processes. It is possible for certainty equivalence control laws to be nonlinear, and such laws can be optimal even when separated designs may not be. For linear plants, they can sometimes be linear or affine, as indicated by the following proposition from \cite{vanDerWaterWillems1981} adapted to our problem.
%%%%%%%%%%%%%%%%%%%%%%%%%%%%%%%%%%%%%%%%%%%%%%%%%%
\begin{lemma}[Affine certainty equivalence laws for linear plants]
For the plant~\eqref{generalMarkov}, with $\Phi_t = ax_t + u_t + w_t$, and the quadratic
performance cost~\eqref{objectiveFunction} with $J^{\textrm{Comm}} = 0$,
the following are certainty equivalence laws:
\begin{align*}
  u_t^{CE}
&
  =
  - \,  k_t^{CE} \Bigl(
            \, a  \cdot
            \mathbb{E}\left[  x_t  \left|
                      {\left\{ z_i \right\}}_{0}^{t} ,
                      {\left\{ u_i \right\}}_{0}^{t-1}
                                  \right.
                     \right]
        +
            \mathbb{E}\left[  w_t  \left|
                      {\left\{ z_i \right\}}_{0}^{t} ,
                      {\left\{ u_i \right\}}_{0}^{t-1}
                                  \right.
                     \right]
              \Bigr) ,
\end{align*}
where the gains $ k_i^{CE} =
   {\frac{ { \beta_{i+1} } }
           { q + \beta_{i+1} }}
$,
$\alpha_i =  \beta_{i+1} + \alpha_{i+1}  $,
$\beta_i =  p +  {\frac{a^2 q { \beta_{i+1}} }
            { q  + \beta_{i+1} }}
$,
$\alpha_{T+1} = 0$, $\beta_{T+1} = 1$.
\label{lemma:affineCElaws}
\end{lemma}

\begin{definition}[Certainty equivalence property]
\label{ceDefinitionProperty}
The certainty equivalence property holds for a stochastic control
problem if it is
 optimal to apply the certainty equivalence control law.
\end{definition}
For the stochastic control problem described in
Lemma~\ref{lemma:affineCElaws}, with non-linear measurements that do not
result in a dual effect of the controls, Bar-Shalom and Tse~\cite{barShalomtse1974dualEffectPaper} showed that
the certainty equivalence property holds.

We now consider a simple example, and show that there is a dual effect of the control signal in the closed-loop system presented in Section~\ref{twoAgentProblemStatement}.

%%%%%%%%%%%%%%%%%%%%%%%%%%%%%%%%%%%%%%%%%%%%%%%%%%%%%%%%%%%%%%%
\begin{example}
For the plant~(\ref{plant}), let ~$a=1$, $x_0=2$, and $\sigma_0=0$.
Let this information be known to the encoder and the controller, which
simply means that~$z_0=x_0$.
Let the variance $\sigma_w^2=0.7^2$.
For the objective function, let the horizon end at~$T=1$, and let~$p=q=0.01$.
Let the channel alphabet be the discrete set~$\left\{1,2,3\right\}$.

For the given threshold~$\theta=1.6$, let the encoder at~$t=1$ be:
\begin{align}
     \xi_1  \left( x_1 \right)
&
     =
     \begin{cases}
        1  & \text{if} \ x_1 \in \left( -\infty,- \theta \right), \\
        2  & \text{if} \ x_1 \in \left( -\theta,\theta \right), \\
        3  & \text{if} \ x_1 \in \left( \theta, + \infty \right).
     \end{cases}
\label{quantizerExampleDualEffect}
\end{align}
The optimal control law at $t=1$ is
$ u_1 = - {\frac{a}{q+1}} {\widehat{x}_{ \left. 1 \right| 1}}$, where
${\widehat{x}_{\left. 1 \right| 1}} =
     \mathbb{E}\left[  x_1  \left|
           x_0 , u_0 , z_1  \right.
               \right]
$. %Then the only choice yet to be made is the policy for control~$u_0$. % \emph{Solution:}
Using the encoding policy $\xi_1$ and the optimal control signal $u_1$, the performance cost with $J^{\textrm{Comm}}=0$ can be written as a function of the control at $t=0$:
\begin{align*}
 J(u_0)
&
 =
     \sigma_w^2
 +
     q u_0^2
 +
     \left( p +  {\frac{qa^2}{q+1}} \right)
     \mathbb{E}\left[  x_1^2  \left|
           x_0 , u_0          \right.
               \right]
 +
  {\overset
    { \triangleq \ \Gamma }
    {\overbrace{
      {\frac{a^2}{q+1}}
       \mathbb{E}\left[   \left.
                         {\left(x_1 - \widehat{x}_{\left. m \right| 1}
                         \right)
                        }^2
                          \right|  x_0 , u_0 , z_1
                \right]
    }}
  }
\end{align*}
In the above expression, $\Gamma$ is the quantization distortion, which is thus proportional to the conditional variance of the controller's minimum mean-squared estimation error of~$x_1$. Notice that $\Gamma$ is a function of~$u_0$, thus resulting in a dual effect of the control signal in the plant-encoder-channel combination. Figure~\ref{figure:dualEffectPlots} shows how the quantization distortion $\Gamma$ depends on $u_0$. The total cost~$J$ is also plotted and the optimal value $u_0^*$ is shown to be different from the certainty equivalent control $u_0^{\mathrm{CE}}$.
\label{example:dualEffect}
\end{example}
%%%%%%%%%%%%%%%%%%%%%%%%%%%%%%%%%%%%%%%%%%%%%%%%%%%%%%%%%%%
\begin{figure}
\centering
\includegraphics[width=0.45\textwidth]{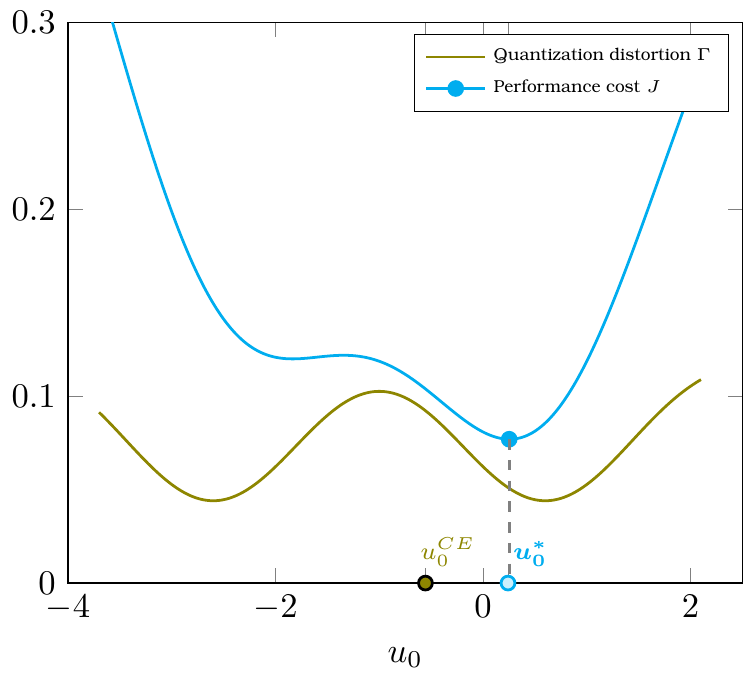}
\caption{
Plot of quantization distortion and performance cost for Example~\ref{example:dualEffect}
}
\label{figure:dualEffectPlots}
\end{figure}
%%%%%%%%%
%%%%%%%%%%%%%%%%%%%%%%%%%%%%%%%%%%%%%%%%%%%%%%%%%%%'
%%%%%%%% Proof of separation and certainty equivalence
%%%%%%%%%%%%%%%%%%%%%%%%%%%%%%%%%%%%%%%%%%
\section{Dynamic encoder-controller design
\label{proofSection}
}
In this section we solve the dynamic encoder-controller design problem~(Design
problem~\ref{fullyDynamicControls}) which
allows both controls and encoders to be dynamic. We work out the details for
the discrete alphabet channel with the fixed alphabet size~$N$.
% at every time tick.
We begin by examining a known structural property of
optimal encoders. This states that it is optimal for
the encoder to apply a quantizer on the state $x_t$, with the shape
of the quantizer depending only on past quantizer outputs.
Next, we present a structural property for
encoders called {\it{controls-forgetting}}, which leads to separation.
Finally, we show that one optimal encoder for Design
problem~\ref{fullyDynamicControls}
does indeed possess this property, which leads to separation and
certainty equivalence for this problem.

\subsection{Known structural properties of optimal encoders} \label{sufficientStatisticsDiscussion}

Let us now formulate the encoder's Markov decision problem.
Fix the control policies to be the arbitrary, but admissible laws:
\begin{align*}
    u_t
&
    =
    \mathcal{K}^{\dagger}_t  \Bigl(
            \left\{ z_i\right\}^t_0
%                 \left\{ u_i\right\}_0^{t-1}
                             \Bigr).
\end{align*}
Then the optimization problem reduces to one of picking encoding policies.
This is a single-agent, sequential decision problem, and hence one with a
classical information pattern. The action space for this decision problem
is the infinite dimensional function space of discrete-valued
encoders. At time~$t$, the encoder takes as input: the current and
previous states, all previous outputs, and all previous encoding maps.
For convenience, we can view this encoding
map as a function of only the current state but with the rest of the
inputs considered as parameters determining the form of this function.
Thus, without
loss of generality the encoder can be described as the function
\begin{gather*}
\xi_t\left( \cdot \right) : \mathbb{R}
\to \left\{1, \ldots, N \right\}
\end{gather*}
having $x_t$ as its argument with its shape determined by
$
%\begin{gather*}
    \Bigl(   \left\{  x_{i} \right\}^{t-1}_0
             \left\{  z_{i} \right\}^{t-1}_0
            \left\{  \xi_{i} \left( \cdot \right) \right\}^{t-1}_0
    \Bigr).
%\end{gather*}
$
Hence the action space at times~$t$ can be described as:
$
\Bigl\{   \xi  \left( \cdot \right)
                   : \mathbb{R}
   \to \left\{1, \ldots, N \right\}, \
   \text{Borel measurable}
\Bigr\}.
$
%%%%%%%%%%%%%%%%%  the Markov decision problem
Identifying encoders as decisions to be picked is not enough, as the
signal~$x_t$
need not be Markov. We utilize the following property.
%%%%%%%%%%%%%%%%%%%%%%%%%%%%%%%%%%%%%%%%%%%%%%%%%%%%%%%%%%%%%%%%%
\begin{lemma}[Striebel's sufficient statistics]
For every design problem we have set up, the signals
\begin{gather*}
     x_t ,
     \,
     \left\{
        z_i
     \right\}_0^t,
     \,
     \left\{
        \xi_{i} \left( \cdot \right)
     \right\}^{t-1}_0
\end{gather*}
form sufficient statistics for the encoding decision at time~$t$.
\end{lemma}
%%%%%%%%%%%%%%%%%%%%%%%%%%%%%%%%%%%%%%%%%%%%%%%%%%%%%%%%%%%%%%%%%
\begin{proof}
See Striebel~\cite{striebel1965}.
\end{proof}
Hence, at every time~$t$, performance is not degraded by the encoder
choosing to quantize just $x_t$ instead of quantizing the entire
waveform $\left\{ x_0, \ldots, x_t \right\}$. Of course the shape of the
quantizer is allowed to vary with past encoder shapes, past encoder
outputs, and on past control inputs. But given the sufficient statistics,
the encoder can forget the data: $\left\{ x_0, \ldots, x_{t-1} \right\}$.

%%%%%%%%%%%%%%%%%%%%%%%%%%%%%%%%%%%%%%%%%%%%%%%%%%%%%%%%%%%%%%%%%%%%%%%%%%
%\textbf{The common information approach: }
Denote by ${\mathscr{D}}_{ t^{-} }^{\rm{con}}$ the data at the controller
just after it has read the channel output~$z_t$ and just before it has
generated the control value~$u_t$. Similarly denote by
${\mathscr{D}}_{ t^{+} }^{\rm{con}}$ the data at the controller just after it
has  generated the control value~$u_t$.
Then
\begin{align*}
   {\mathscr{D}}_{ t^{-} }^{\rm{con}}
&
   =
\Bigl\{
     { \left\{ z_i  \right\} }_0^t ,
     { \left\{ \xi_i \left( \cdot \right)  \right\} }_0^t ,
     { \left\{ u_i  \right\} }_0^{t-1}
\Bigr\},
\\
  {\mathscr{D}}_{ t^{+} }^{\rm{con}}
&
  =
\Bigl\{
       {\mathscr{D}}_{ t^{-} }^{\rm{con}} ,
       u_t
\Bigr\}
\ = \
\Bigl\{
     { \left\{ z_i  \right\} }_0^t ,
     { \left\{ \xi_i \left( \cdot \right)  \right\} }_0^t ,
     { \left\{ u_i  \right\} }_0^{t-1} ,
     u_t
\Bigr\}.
\end{align*}
Also let
% \begin{align*}
$
         {\widehat{x}}_{ \left. t \right| t }
 =
          {\mathbb{E}} \left[  x_t
                \left|  {\mathscr{D}}_{ t^{-} }^{\rm{con}}  \right.
                      \right] .
$

The problem we consider has two decision makers that jointly
minimize a given cost function. The information available to these decision makers
is not the same, and neither is the information available to each agent a subset
of the information available to the agent downstream in the loop. Thus, the
information pattern here is neither classical nor nested. We apply the \emph{common information approach}\footnote{This approach was first proposed by Witsenhausen, as a
conjecture in \cite{Witsenhausen1971}, to deal with multiple decision
makers and non-classical information patterns in a general setting. This
conjecture was shown to be true by Varaiya and Walrand in
\cite{Varaiya1978} for a special case. Our terminology is derived from
\cite{Nayyar2011}, where the conjecture has been studied in detail.} to our problem. This approach allows a designer to treat a problem with multiple decision makers as a classical control problem with a single decision maker that has access to partial state information. When applied to our setup, this approach leads to the following structural result at the encoder. The encoding policy~$\xi_t\left( \cdot \right)$ is selected based on the information available to the controller at the previous time instant namely~${\mathscr{D}}_{ {\left( t -1 \right)}^{+}}^{\rm{con}}$. At times~$t^- , t^+$ respectively, the data~${\mathscr{D}}_{ {\left( t -1 \right)}^{-}}^{\rm{con}},
{\mathscr{D}}_{ {\left( t -1 \right)}^{+}}^{\rm{con}} $ comprise the common information in this problem. The encoding map~$\xi_t\left( \cdot \right)$ is applied to the state~$x_t$, which is private information available to the encoder. A similar approach has been used by others for problems of quantized control~\cite{borkar-mitter-tatikonda,walrandVaraiya1983,yukselIT2013}.

% \end{align*}
%%%%%%%%%%%%%%%%%%%%%%%%%%%%%%%%%%%%%%%%%%%%%%%%%%%%%%%%%%%%%%%%%
\subsection{Controls-forgetting encoders and separation\label{controlsForgettingSection}}
We now present a structural property of encoders which ensures separation
in design.
Recall the plant~(\ref{plant}) and cost~(\ref{objectiveFunction}), and
define the following control free part of the state:
\begin{align*}
   \zeta_0
&
  =
  x_0,
\\
  \zeta_{i+1}
&
  = x_{i+1} - \sum_{j=0}^i a^{i-1} u_j \ \text{for} \ i \ge 0.
\end{align*}
At the encoder, the change of variables
\begin{align}
    \Bigl(
          x_t, \left\{z_i\right\}_0^{t-1} ;
               \left\{ \mathcal{K}_i \left( \cdot \right) \right\}_0^T
    \Bigr)
&
    \longmapsto
    \Bigl(
          \zeta_t , \left\{z_i\right\}_0^{t-1} ;
               \left\{ \mathcal{K}_i \left( \cdot \right) \right\}_0^T
    \Bigr)
\label{changeOfVariables}
\end{align}
is causal and causally invertible. Hence the statistics
$
\Bigl(
          \zeta_t , \{z_i\}_0^{t-1} ; \{ \mathcal{K}_i (\cdot) \}_0^T
\Bigr)
$
are also sufficient statistics at the encoder.
We now introduce the innovation encoding of Borkar and Mitter~\cite{borkarMitter1997}.
%%%%%%%%%%%%%%%%%%%%%%%%%%%%%%%%%%%%%%%%%%%%%%%%%%%%%%%%%%%%
\begin{definition}[Innovation encoder~\cite{borkarMitter1997}]
An encoder with the inputs and outputs:
\begin{align*}
\Bigl(
          \zeta_t , \{z_i\}_0^{t-1} ; \{ \mathcal{K}_i (\cdot) \}_0^T
\Bigr)
\longmapsto
\iota_t
\end{align*}
is admissible and is called an `innovation' encoder.
\end{definition}
%%%%%%%%%%%%%%%%%%%%%%%%%%%%%%%%%%%%%%%%%%%%%%%%%%%%%%%%%%%%%
\noindent
The networked control system in Figure~\ref{simpleLoop} redrawn with an innovation encoder is shown in Figure~\ref{blockDiagramEquivalentEncoder}. Note that with innovation encoding, the control free part of the state is not affected by the control policies, but obeys the recursion
\begin{align*}
\zeta_{t+1} &= a \zeta_t + w_t.
\end{align*}
For any sequence of causal encoders, one can find an equivalent sequence
of innovation encoders such that when these two sets operate on the same
sequence of plant outputs, they produce two sequences of channel inputs
that are equal with probability one.
Hence, if for a
plant and channel, the dual effect is present in a certain class of causal
encoders, then the dual effect is also present in the equivalent class of
innovation encoders~\cite{fengLoparo1997activeProbing}. This is what the
following example illustrates:
%%%%%%%%%%%%%%%%%%%%%%%%%%%%%%%%%%%%%%%%%%%%%%%%%%%%%%%%%%%
\begin{figure}  % [ht]
\centering
\includegraphics[width=0.99\textwidth]{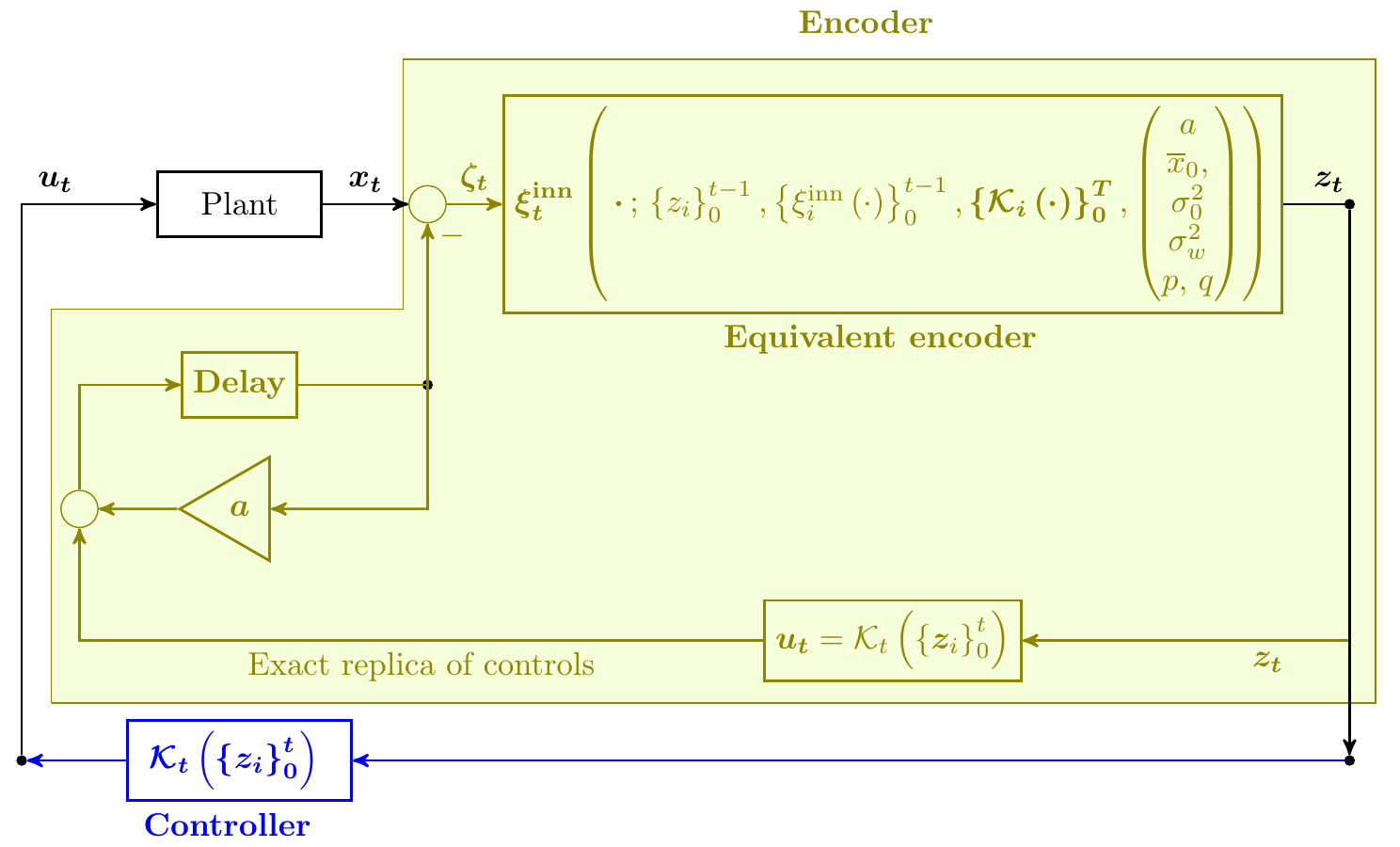}
\caption{The block diagram of Figure~\ref{simpleLoop}
with innovation encoding}
\label{blockDiagramEquivalentEncoder}
\end{figure}
%%%%%%%%%%%%%%%%%%%%%%%%%%%%%%%%%%%%%%%%%%%%%%%%%%%%%%%%%%%%%%%%%%%%%%%%%
\begin{example}[Dual effect in a loop with fixed innovation encoder] \label{example:dualEffectInnovationEncoding}
We use the same setup as in Example~$1$ with the encoder replaced by an innovation encoder.
For the given threshold~$\theta=1.6$, let the encoder at time~$t=1$
be the following innovation encoder:
\begin{align}
     \xi_1^{\rm{inn}}  \left( \zeta_1 \right)
&
     =
     \begin{cases}
        1  & \text{if} \ a\zeta_1+ {\mathcal{K}}_0
                  \left( z_0 \right)
             \in \left( -\infty,- \theta \right), \\
        2  & \text{if} \ a\zeta_1+ {\mathcal{K}}_0
                  \left( z_0 \right)
             \in \left( -\theta,\theta  \right), \\
        3  & \text{if} \ a\zeta_1+ {\mathcal{K}}_0
                  \left( z_0 \right)
             \in \left( \theta, + \infty  \right).
     \end{cases}
\label{quantizerExampleDualEffectInnovation}
\end{align}
The optimal control law at $t=1$ is still
$ u_1 = - {\frac{a}{q+1}} {\widehat{x}_{ \left. 1 \right| 1}}$, where
${\widehat{x}_{\left. 1 \right| 1}} =
     \mathbb{E}\left[  x_1  \left|
           x_0 , u_0 , z_1  \right.
               \right]
$. For the control $u_0$, notice that
\eqref{quantizerExampleDualEffect} and \eqref{quantizerExampleDualEffectInnovation}
tell us that this innovation encoder~$\xi^{\rm{inn}}_t$ is equivalent to the
causal encoder~$\xi_t$  of Example~\ref{example:dualEffect}.
% With probability one, they produce the same outputs,
For the same applied control policy~$\mathcal{K}_0$, and for the same
realizations of primitive random variables, we get
$
    \xi^{\rm{inn}}_1  \left( \zeta_1  \left( \vec{\omega} \right)   \right)
  =
    \xi_1  \left( x_1  \left( \vec{\omega} \right)   \right)
$.
Hence, with probability one the two nonlinear filters for the state given
$x_0, z_1$ are the same. Thus for an event $X \in
{\mathcal{F}}^{\left( x_0, z_1 \right)}$, we have:
\begin{align*}
    {\mathbb{P}}
    \left[
         x_1 \in  X \left| x_0 , \, z_1 =
                           \xi^{\rm{inn}}_1  \left( \zeta_1  \right)
                   \right.
    \right]
&
    =
    {\mathbb{P}}
    \left[
         x_1 \in  X \left| x_0 , \, z_1 =
                           \xi_1  \left( x_1  \right)
                   \right.
    \right].
\end{align*}
Hence the results in Figure~\ref{figure:dualEffectPlots} apply also to
this example.
\end{example}
%%%%%%%%%%%%%%%%%%%%%%%%%%%%%%%%%%%%%%%%%%%%%%%%%%%%%%%%%%%%%%%%%
The encoder~(quantizer) in the loop causes the dual effect. Furthermore,
the encoder's presence renders useless the techniques that worked in the
case of the classical, single-agent, partially observed LQ control problem.
The next  example illustrates this.
%%%%%%%%%%%%%%%%%%%%%%%%%%%%%%%%%%%%%%%%%%%%%%%%%%%%%%%%%%%%%%%%%%%%%%%%%
\begin{example}%[Another difference from classical LQ control]
\label{example:predictionErrorDependence}
We examine a scalar system as it evolves from time step 0 to time step
1. We have: $
x_0  \sim {\mathcal{N}} \left( \mu_0 , \sigma_0 \right),$
\begin{align*}
x_1 & = x_0 + u_0 + w_0,
\end{align*}
where $w_0$ is the process noise variable  which is independent of $x_0, u_0$,
and $w_0 \sim  {\mathcal{N}} \left( 0 , \sigma_w \right).$
We adopt the specific quantizing strategy given
below~(on the left in the form of a encoder for $x_t$, and on the
right, in the equivalent, innovation form):
$$
\begin{aligned}
  \xi_0 (x_0) & = \begin{cases}
             -1, & {\text{if}} \ \  x_0 \le 0, \\
             +1, & {\text{if}} \ \  x_0 > 0,
            \end{cases}
\quad \quad \quad &
  \xi_0^{\text{inn}} ({\zeta}_0) & = \begin{cases}
             -1, & {\text{if}} \ \  {\zeta}_0 \le 0, \\
             +1, & {\text{if}} \ \ {\zeta}_0 > 0,
            \end{cases}  \\
  \xi_1 (x_1) & = \begin{cases}
             -1, & {\text{if}} \ \  x_1 \le 0, \\
             +1, & {\text{if}} \ \  x_1 > 0,
            \end{cases}
% \xi_1 (x_1 ) & = \xi_0 (x_1) \ \forall x \in {\mathbb{R}}.
\quad \quad \quad &
 \xi_1^{\text{inn}} ({\zeta}_1) & = \begin{cases}
             -1, & {\text{if}} \ \ {\zeta}_1  \le -u_0\left(z_0
\right), \\
              +1, & {\text{if}} \ \  {\zeta}_1 > -u_0\left(z_0
\right).
             \end{cases}
\end{aligned}
$$
Since the encoder at time~$0$ is binary, the general control law at time~$0$ has the form:
\begin{align*}
  u_0 (z_0) & = \begin{cases}
             \alpha, & {\text{if}} \ z_0 = -1, \\
             \beta, & {\text{if}} \  z_0 = +1,
            \end{cases}
\end{align*}
where $\alpha, \beta$ are arbitrary real numbers.
The process~${\widehat{x}_{\left. t \right| t}} $ is fully observed at
the controller. We have
${\widehat{x}_{\left. 0 \right| 0}}
=  {\mathbb{E}} \left[ x_{0} \left| z_0 \right. \right]
 $, and as noted in~\cite{yuksel2014tac}, one can write:
\begin{align}
{\widehat{x}_{\left. 1 \right| 1}}
&
 = {\widehat{x}_{\left. 0 \right| 0}}  + u_0 + \xbar{w}_0,
\label{yukselsUpdateEquation}
\end{align}
where the noise-like random variable~$\xbar{w}_0$ is given by:
%\begin{align*}
$
\xbar{w}_0  \triangleq  {\mathbb{E}} \left[ x_{1} \left|z_0, z_1 \right. \right]
                - {\mathbb{E}} \left[ x_{1} \left|z_0 \right.
                  \right].
$
%\end{align*}
Then one can treat the problem as the control of the
fully observed process~${\widehat{x}_{\left. t \right| t}}$
to minimize the given cost, which can be rewritten as the following
sum of two terms:
\begin{align}
J  & =  {\mathbb{E}} \left[ {\widehat{x}}_{1\left|1\right.}^2
             +  p \cdot {\widehat{x}}_{0\left|0\right.}^2
             +  p \cdot  {\left( x_0 -
{\widehat{x}}_{0\left|0\right. }
\right)}^2 %_{\triangleq Var_{t\left|t\right.}^{\text{err}}}
             +     q u_0^2  \right]
 +
 {\mathbb{E}} \left[
                 {\left( x_1 -
{\widehat{x}}_{1\left|1\right. }
\right)}^2 %}_{\triangleq Var_{t\left|t\right.}^{\text{err}}}
\right].
\label{yukselCorrectCost}
\end{align}

%  {In~\cite{yuksel2014tac}, the authors claim that the statistics of $\xbar{w}_{t}$ do not depend on the past controls $u_{k}$, for $k < t$, when innovation encoding is used. They use this claim to establish that the second term on the RHS of~\eqref{yukselCorrectCost} does not vary with~$\{{u}_t\}$. Consequently, they claim that there is no dual effect and that separation holds for this problem. This is not correct, as we show below.}

Such a treatment actually works for the case of the classical,
single-agent partially observed LQ control problem.
There two special things happen: (1)~the random process~$\{\xbar{w}_t\}$
is statistically independent of the control process~$\{{u}_t\}$ and of
the `state' process~$\{{\widehat{x}}_{t\left|t\right.} \}$,
and (2)~because the dual effect is absent, the second term on the RHS
of~\ref{yukselCorrectCost} does not vary with~$\{{u}_t\}$.
Therefore, by considering~$\{{\widehat{x}}_{t\left|t\right.} \}$
as the process to be controlled, we get a single-agent,
fully observed LQ control problem.

In the two-agent problems considered in this paper,
neither of the above-mentioned special things may happen.
For this specific example, we have calculated, and then
plotted in Figure~\ref{graphsForExamplePredictionError} how
the second moments of~$\xbar{w}_0$ and $x_1 - {\widehat{x}}_{1\left|1\right.} $ vary with~$u_0$. The calculations are presented in Appendix~\ref{App:ExplicitCalcEx3}.
 \begin{figure}
 \centering
\includegraphics[width=0.48\textwidth]{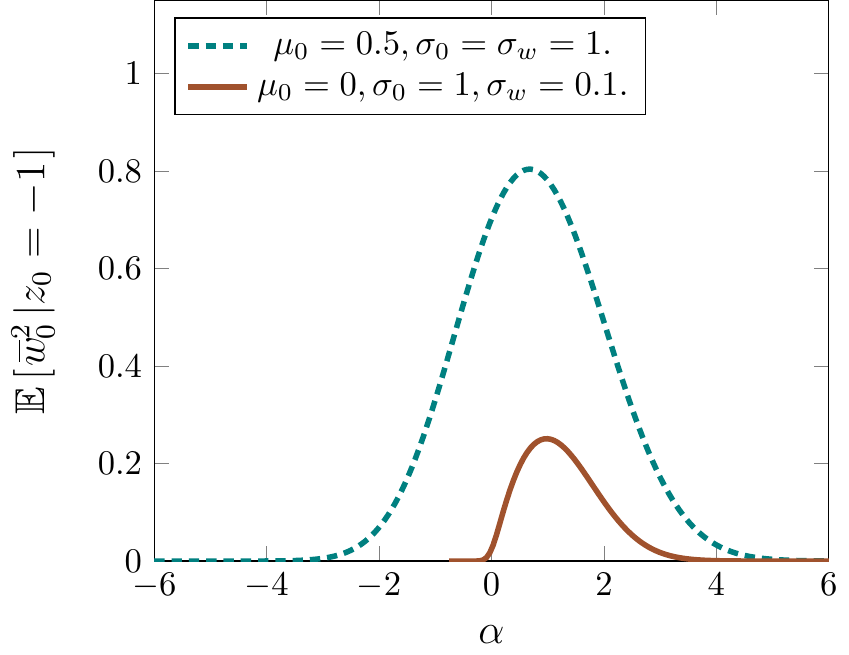}
\includegraphics[width=0.48\textwidth]{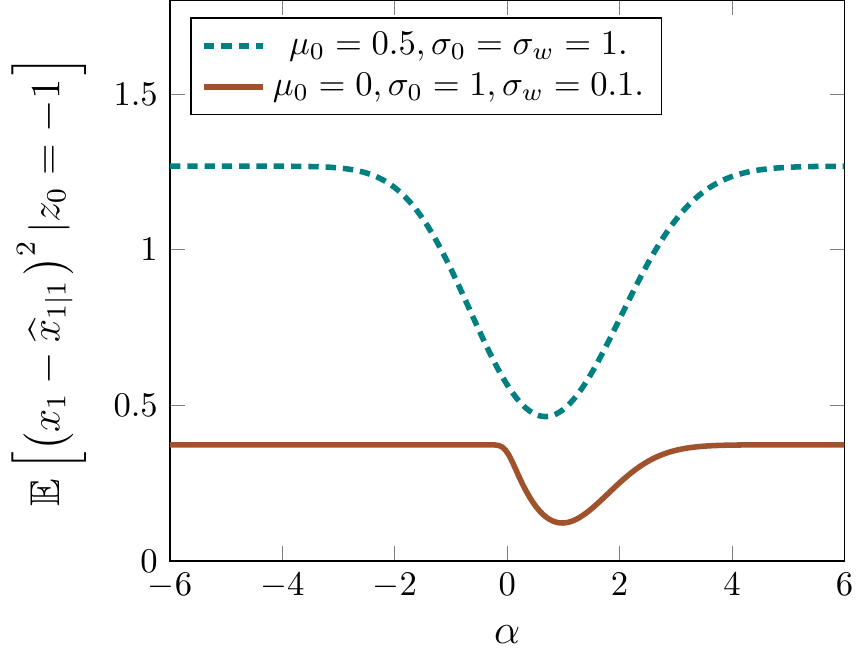}
\caption{Plots for
example~\ref{example:predictionErrorDependence}.}
\label{graphsForExamplePredictionError}
\vspace{-4mm}
 \end{figure}
\hfill $\square$\end{example}
%%%%%%%%%%%%%%%%%%%%%%%%%%%%%%%%%%%%%%%%%%%%%%%%%%%%%%%%%%%%%%%%%
%\noindent
%
%We now define special classes of controllers and encoders.

Next we define a class of encoders for which at prescribed times~$t,$ the statistics of~$\xbar{w}_t$, $x_{t+1} - {\widehat{x}}_{t+1\left|t+1\right.}$ are independent of the control~$u_t.$
%%%%%%%%%%%%%%%%%%%%%%%%%%%%%%%%%%%%%%%%%%%%%%%%%%%%%%%%%%%%%%%%%%%%%%%%%
\begin{definition}[Controls-forgetting encoder]
Denote by
$
                \rho_{ \tau \left| \tau - 1  \right. }^{\zeta}
                        \left( \cdot  \right)
$
the conditional density of~$\zeta_{\tau}$ given the
data~${\mathscr{D}}_{{\left( \tau -1 \right)}^-}^{\rm{con}}$.
An admissible encoding strategy is
{{controls-forgetting from time~$\tau$}} if it takes
the form:
\begin{align*}
     \xi_t^{CF,\, \tau} \left( x_t ;
             {\mathscr{D}}_{{\left( t -1 \right)}^-}^{\rm{con}}  \right)
&
    =
    \begin{cases}
            \xi_t^{\dagger} \left( x_t ;
             {\mathscr{D}}_{{\left( t -1 \right)}^-}^{\rm{con}}  \right),
          &
             \text{if} \ t \le \tau,
          \\
            \epsilon_t \left( \zeta_t ; \,
                \rho_{ \tau \left| \tau - 1  \right. }^{\zeta}
                        \left( \cdot  \right) , \,
                \left\{ z_i \right\}_{\tau}^{t-1} ,
                {\left\{ \epsilon_i  \left( \cdot \right) \right\}
                }_{\tau}^{t-1}
                       \right),
          &
             \text{if} \ t \ge  \tau + 1,
    \end{cases}
\end{align*}
where (1)~$
\xi_t^{\dagger} \bigl( \cdot ;
                       {\mathscr{D}}_{{\left( t -1 \right)}^-}^{\rm{con}}
               \bigr)
$
is any admissible policy for encoding at time~$t$, (2)~for~$t
\ge \tau + 1 $ the policies $
            \epsilon_t \bigl( \cdot ;
                \rho_{ \tau \left| \tau - 1  \right. }^{\zeta}
                        \left( \cdot  \right) , \,
                \left\{ z_i \right\}_{\tau}^{t-1} ,
        {\left\{ \epsilon_i  \left( \cdot \right) \right\}}_{\tau}^{t-1}
                       \bigr)
$
are adapted to the data
\begin{align*}
    {\mathscr{D}}_{{\left( t -1 \right)}^+}^{CF, \, \tau}
&
    =
    \left(
                \rho_{ \tau \left| \tau - 1  \right. }^{\zeta}
                        \left( \cdot  \right) , \,
                \left\{ z_i \right\}_{\tau}^{t-1} ,
        {\left\{ \epsilon_i  \left( \cdot \right) \right\} }_{\tau}^{t-1}
    \right) \ \subset \ {\mathscr{D}}_{{\left( t -1 \right)}^+}^{\rm{con}} ,
 \ \
       \text{for} \ t \ge \tau ,
\end{align*}
and (3)~for fixed values of the data~$
{\mathscr{D}}_{{\left( t -1 \right)}^+}^{CF, \, \tau}
$,
the map~$
 \epsilon_t \left( \cdot \right)
$
produces the same output regardless of both the controls~$
\left\{ u_i \right\}_{\tau}^{t}
$
and the control policies~$
  { \left\{  {\mathcal{K}}_i \left( \cdot \right)\right\} }_{t+1}^{T}.
$
\end{definition}
%%%%%%%%%%%%%%%%%%%%%%%%%%%%%%%%%%%%%%%%%%%%%%%%%%%%%%%%%%%%%%%%%%%%%%%%%
%\noindent

Clearly such controls-forgetting encoders exist. For example,
consider a set of encoders that quantize in sequence $
\zeta_{\tau + 1}, \ldots , \zeta_{T}
$
to minimize the estimation distortion
$
  \sum_{i=\tau + 1}^{T}
         {\mathbb{E}}
      \big[
             {\big(
                \zeta_i - \widehat{\zeta}_{i | i}
             \big)
             }^2
      \big]
$, where
$
    \widehat{\zeta}_{\left. i \right| i}
=
         {\mathbb{E}}
      \bigl[    \left.
                \zeta_i
                \right|
                       {\mathscr{D}}_{{\left( i -1 \right)}^+}^{CF, \, \tau}
      \bigr]
$.
Let the non-negative function~$\psi\left( \cdot \right)$
represent some notion of cost.
For example, $\psi\left( x \right) :=  {x}^2$.
%%%%%%%%%%%%%%%%%%%%%%%%%%%%%%%%%%%%%%%%%%%%%%%%%%%%%%%%%
\begin{lemma}[Distortions incurred by controls-forgetting encoders also forget controls]
Fix the time~$t=\tau$ and the distortion measure~$\psi$. If the encoder
is controls-forgetting from time~$\tau$, then for times~$i \ge \tau +1$, the
% conditional expected
distortions
%\begin{gather*}
$         {\mathbb{E}}
      \left[
               \left.
             \psi\left( x_i - \widehat{x}_{i\left| i \right. } \right)
                \right|
                       {\mathscr{D}}_{{i}^-}^{\rm{con}}
      \right]
$
%\end{gather*}
are statistically independent of the partial set of controls
$
   { \left\{
       u_i
     \right\}
   }_{i=\tau}^{T}
$.
\end{lemma}
%%%%%%%%%%%%%%%%%%%%%%
\begin{proof}
The unconditional statistics of
$
   { \left\{
       \zeta_t
     \right\}
   }
$
are independent of the entire control waveform, no matter what the
encoder is. For times~$i \ge \tau +1$ and for sets
$           X \in {\mathcal{F}}^{z}_{i}
$,
% the conditional probabilities
$
         {\mathbb{P}}
      \left[
               \left.
             \zeta_i \in X
                \right|
                       {\mathscr{D}}_{{i}^-}^{\rm{con}}
      \right]
$
is independent of
$
   { \left\{
       u_i
     \right\}
   }_{i=\tau}^{T}
$
because the encoding maps~$\xi_i$ are controls-forgetting from
time~$\tau$. Since $\zeta_t -\widehat{\zeta}_{t\left| t \right. }
= x_t -\widehat{x}_{t\left| t \right. }$, for all $t$, the lemma follows.
\hfill
\end{proof}
%In our paper, we restrict ourselves to quadratic costs, such as
%For example:~
%$\psi\left( x \right) :=  {x}^2$.
\begin{definition}[Controls affine from time~$\tau$]
%Fix the time~$t=\tau$.
A controller affine from time~$\tau$ takes the following form:
\begin{align}
    \mathcal{K}_i^{{\rm{mult}}, \, \tau}
   \Bigl(
          \mathscr{D}_{i^-}^{\rm{con}}
   \Bigr)
&
   =
   \begin{cases}
        u_i^{\dagger} ,
           &  \text{if} \ \ i < \tau , \\
        u_i^{{\rm{a f f}}} =  k_i \widehat{x}_{ i  | i} + d_i ,
           & \text{if} \ \ i \geq  \tau ,
   \end{cases}
\label{eqn:affineMultiplexedLaw}
\end{align}
where the controls~$u_i^{\dagger}$ are generated by an admissible
strategy
$
  \{
      \mathcal{K}_i^{\dagger} (\cdot)
  \}_{i=0}^T
$, the controls~$u_i^{\rm{a f f}}$
are generated by an affine strategy
$
\{
     \mathcal{K}_i^{{\rm{a f f}}} (\cdot)
\}_{i=0}^T
$,
% \end{align*}
with the gains $ \{k_i \}_0^T$ and offsets $ \{d_i\}_0^T $ computed
offline, and
$
  \widehat{x}_{ i  | i} =
    {\mathbb{E}}
\bigl[ \left.  x_i \right| \left\{ y_j \right\}_{j=0}^{i} \bigr] .
$
 % and stored in the memory of the controller.
\label{definition:multiplexedAffineControls}
\end{definition}
%%%%%%%%%%%%%%%%%%%%%%%%%%%%%%%%%%%%%%%%%%%%%%%%%%%%%%%%%%%%%%%%%%%%
%%%%%%
%%%%%%   Backward dynamic programming
%%%%%%%%%%%%%%%%%%%%%%%%%%%%%%%%%%%%%%%%%%%%%%%%%%%%%%%%%%%%
\subsection{Preliminary lemmas} ~\label{preliminaryLemmasSection}
%%%%%%%%%%%%%%%%%%%%%%%%%%%%%%%%%%%%%%%%%%%%%%%%%%%%%%%%%%%%%%%%
The main result ahead is Theorem~\ref{theorem:CElawsForCFencoders}
that states that it is optimal for Design problem~\ref{fullyDynamicControls} to
apply a separated design and certainty equivalence controls.
In this subsection, we do some necessary ground work towards
proving that result.

Once we are prescribed an admissible encoder, the
controls~${\left\{ u_j \right\}}_{j=i}^{T}$ affect only the cost-to-go:
$
  {\mathbb{E}} \left[
                 x_{T+1}^2
              \right]
  + \sum_{j=i}^{T}
  {\mathbb{E}} \left[
                 p x_j^2 + q u_j^2
              \right]
$.
In the classical single agent LQ problem, the `prescribed encoder' is
simply the linear observation process with prescribed signal-to-noise
ratios. There, this cost-to-go can be expressed
as a quadratic function of~${\left\{ u_j \right\}}_{j=i}^{T},
{\left\{ x_j \right\}}_{j=i}^{T}$ and
$ {\left\{ \widehat{x}_{j\left| j \right.}  \right\}}_{j=i}^{T}$.
But in our two agent LQ problem, because of the dual effect, the cost to
go may have a non-quadratic dependence on the
controls~${\left\{ u_j \right\}}_{j=i}^{T}$. However we show that
by restricting to controls-forgetting encoders and affine controls, the
cost-to-go does get a quadratic dependence on controls.
We use this reasoning and dynamic programming to show that
for time~$t=i$ going backwards from~$T$ the following conclusions fall
out:
\begin{itemize}
\item{it is optimal at time~$t=i$ to apply as control a linear function of
      $ \widehat{x}_{i\left| i \right.} $, and,
     }
\item{it is optimal at time~$t=i$ to apply an
encoding map that is controls-forgetting from time~$i-1$.}
\end{itemize}
%%%%%%%%%%%%%%%%%%%%%%%%%%%%%%%%%%%%%%%%%%%%%%%%%%%%%%%%%%%%%%%%%%%%%%
%%%%%%%%%%%%%%%%%%%%%%%%%%%%%%%%%%%%%%%%%%%%%%%%%%%%%%%%%%%%%%
\begin{lemma}[Optimal control at time~$t=T$]
The optimal control policy at time~$t=T$ is the linear law:
$
u^*_T
 =
- {\frac{a}{1 + q}}
{\widehat{x}}_{ \left. T \right| T }
$,
and the optimum cost-to-go
$
   V^*_T \left( {\mathscr{D}}_{ T^{-} }^{\rm{con}} \right)
   =
   \min_{u_t}
   {\mathbb{E}} \left[  x_{T +1 }^2 + q u_T^2
        \left|  {\mathscr{D}}_{ T^{-} }^{\rm{con}}  \right.
               \right]
$
is the expected value of a quadratic in $x_T$ and
${\widehat{x}}_{ \left. T \right| T }$.
\label{lemma:optimalUT}
\end{lemma}
%%%%%%%%%%%%%%%%%%%%
\begin{proof}
At time~$T^-$, one is given ${\mathscr{D}}_{ T^{-} }^{\rm{con}}$,
and is asked to pick~$u_T$ to minimize the cost-to-go
\begin{align*}
V_T \left( u_T  ;  {\mathscr{D}}_{ T^{-} }^{\rm{con}}  \right)
&
=
{\mathbb{E}} \left[  x_{T +1 }^2 + q u_T^2
                \left|  {\mathscr{D}}_{ T^{-} }^{\rm{con}}  \right.
             \right] , \\
&
 = \sigma_w^2 +
{\mathbb{E}} \left[ a^2 \,  x_{T }^2 +  2 \, a \, x_{T } u_{T }
                     + \left( 1 + q \right)  u_T^2
                \left|  {\mathscr{D}}_{ T^{-} }^{\rm{con}}  \right.
             \right] ,
\\
&
 = \sigma_w^2 + {\frac{a^2}{1 + q}}
{\mathbb{E}} \left[ \left. q \,  x_{T }^2  +
                      {\left( x_T - {\widehat{x}}_{ \left. T \right| T }
                       \right)
                      }^2
                  \right|
                   {\mathscr{D}}_{ T^{-} }^{\rm{con}}
            \right]
+
              \left( 1 + q  \right)
                      {\left( u_T -
                 {\frac{a}{1+q}}  {\widehat{x}}_{ \left. T \right| T }
                       \right)
                      }^2 ,
\end{align*}
and this lets us prove the Lemma.
\end{proof}
%%%%%%%%%%%%%%%%%%%%%%%%%%%%%%%%%%%%%%%%%%%%%%%%%%%%%%%%%%%%%%%%%%%%%%%
%%
%%%%%%%%%%%%%%%%%%%%%%%%%%%%%%%%%%%%%%%%%%%%%%%%%%%%%%%%%%%%%%
\begin{lemma}[Optimal~$\xi_i$ for separated, quadratic cost-to-go]
\label{lemmaOnQuantizerDistortion}
Fix the time~$t=i$. Consider the dynamic encoder-controller design
problem~(Design problem~\ref{fullyDynamicControls}), for the linear
plant~(\ref{plant}), and the performance cost~(\ref{objectiveFunction}). Suppose that we apply an admissible
controller~$\widetilde{\mathcal{K}}$ along with an
encoder~$\xi_{t}^{CF, i}$ that is controls-forgetting from time~$i$.
Furthermore, suppose that the partial sets of policies:
\begin{gather*}
\ \ \ \ \ \ \ \: \,
\left\{
         \xi_{i+1}^{CF, i} \left(  \cdot \right) ,
          \ldots ,  \:
         \xi_{T}^{CF, i} \left(  \cdot \right)
\right\}
\\
\left\{
        \widetilde{\mathcal{K}}_{i} \left(  \cdot \right) ,
        \widetilde{\mathcal{K}}_{i+1} \left(  \cdot \right) , \ \
          \ldots ,
        \widetilde{\mathcal{K}}_{T} \left(  \cdot \right)
\right\}
\end{gather*}
are chosen such that the following three properties hold:
\begin{enumerate}
\item{the cost-to-go at time~$i$ takes the separated form:
\begin{align*}
     {\mathbb{E}}
    \left[  \left.
       x_{T+1}^2 + p \sum_{j= i}^T x_i^2 + q \sum_{j=i}^T u_i^2
            \right|  \mathscr{D}_{i^+}^{\rm{con}}
    \right]
&
    =
     {\mathbb{E}}
    \left[  \left.
    J_{i}^{\rm{con}} \left( u_i , x_i \right)
            \right|  \mathscr{D}_{i^+}^{\rm{con}}
    \right]
    +
     {\mathbb{E}}
    \left[  \left.
    \Gamma_{i+1}
            \right|  \mathscr{D}_{i^+}^{\rm{con}}
    \right]  ,
\end{align*}
where,
$
    J_{i}^{\rm{con}} \left( u_i , x_i \right)
    =
    \overline{\alpha} + \alpha \, \sigma_w^2
    + \overline{\beta} \, x_i + \widetilde{\beta} \, x_i^2
    + \overline{\nu} \, \widehat{x}_{ \left. i \right| i}
    + \widehat{\nu} \, x_i \widehat{x}_{ \left. i \right| i}
    + \widetilde{\nu} \, \widehat{x}_{ \left. i \right| i}^2 ,
$
and the term $\Gamma_{i+1}$ is a weighted sum of future distortions
and depends only on the random sequence
$
  \left\{
      x_j -\widehat{x}_{j\left| j \right. }
  \right\}_{j=i+1}^{T}
$,
}
\item{the coefficients of the quadratic~$ J_{i}^{\rm{con}}$ may depend on the
control policies~$\bigl\{ \widetilde{\mathcal{K}}_{j}
\left( \cdot \right)\bigr\}_{i}^{T}$ but not on the partial set of
encoding maps
$
\bigl\{
       \xi_{j}^{CF, i} \left(  \cdot \right)
\bigr\}_{i+1}^{T},
$
and,}
\item{the term~$ \Gamma_{i+1} $ depends on the encoding maps
$
\bigl\{
       \xi_{j}^{CF, i} \left(  \cdot \right)
\bigr\}_{i+1}^{T}
$
but not on the partial set of  control
policies~$\bigl\{ \widetilde{\mathcal{K}}_{j}
\left( \cdot \right)\bigr\}_{i}^{T}$.
}
\end{enumerate}
Then, it is optimal to apply an encoding map at time~$t=i$
that does not depend on the data:
$\bigl( u_{i-1} , \, \bigl\{ \widetilde{\mathcal{K}}_{j}
\left( \cdot \right)\bigr\}_{i}^{T} \, \bigr)$.
It also follows that the shapes of the  encoding maps
$
\bigl\{
       \xi_{j}^{CF, i} \left(  \cdot \right)
\bigr\}_{i+1}^{T}
$
and their performance do not depend on the control~$u_{i-1}$.
\end{lemma}
%%%%%%%%%%%%%%%%%%%%%%%%%%%%%%%%%%%%%%%%%%%%%%%%%%%%%%%%%%%%%%%%%%%%%%%%
\begin{proof}
The proof exploits three facts: Firstly the special form of
$    J_{i}^{\rm{con}} \left( u_i , x_i \right) $  makes the encoder's
performance cost at time~$i$ a sum of a quadratic distortion between
$x_i$ and $ \widehat{x}_{ \left. i \right| i  }$, and a term gathering
distortions at later times. Secondly the minimum of the sum distortion
depends only on the intrinsic shape of the conditional
density~$\rho_{ \left. i \right| i-1  } \left( \cdot \right)$ and not
on its mean.  Thirdly, these facts and the controls-forgetting nature
of later encoding maps allows the encoder to `ignore' the
control~$u_{i-1}$. We now start by writing the cost-to-go as:
\begin{align}
     {\mathbb{E}}
    \left[  \left.
    J_{i}^{\rm{con}} \left( u_i , x_i \right)
    +
    \Gamma_{i+1}
            \right|  \mathscr{D}_{i^+}^{\rm{con}}
    \right]
&
    =
     {\mathbb{E}}
    \left[  \left.
    \overline{\alpha} + \alpha \, \sigma_w^2
    + \overline{\beta} \, x_i + \widetilde{\beta} \, x_i^2
    + \overline{\nu} \, \widehat{x}_{ \left. i \right| i}
    + \widehat{\nu} \, x_i \, \widehat{x}_{ \left. i \right| i}
    + \widetilde{\nu} \, \widehat{x}_{ \left. i \right| i}^2
            \right|  \mathscr{D}_{i^+}^{\rm{con}}
    \right] \nonumber
\\
&
    \ \ \ \
    +
     {\mathbb{E}}
    \left[  \left.
    \Gamma_{i+1}
            \right|  \mathscr{D}_{i^+}^{\rm{con}}
    \right], \nonumber
\\
&
    =
    \overline{\alpha} + \alpha \, \sigma_w^2
    +
     {\mathbb{E}}
    \left[  \left.
     \left(  \overline{\beta} + \overline{\nu} \right)  x_i
    + \left( \widehat{\nu} + \widetilde{\nu} +\widetilde{\beta} \right) x_i^2
            \right|  \mathscr{D}_{i^+}^{\rm{con}}
    \right] \nonumber
\\
&
    \ \ \ \
    -
     \left( \widehat{\nu} + \widetilde{\nu}  \right)
     {\mathbb{E}}
    \left[  \left.
        x_i^2 - \widehat{x}_{ \left. i \right| i}^2
            \right|  \mathscr{D}_{i^+}^{\rm{con}}
    \right]
    +
     {\mathbb{E}}
    \left[  \left.
    \Gamma_{i+1}
            \right|  \mathscr{D}_{i^+}^{\rm{con}}
    \right] , \nonumber
\\
&
    =
    \overline{\alpha} + \alpha \, \sigma_w^2
    +
     {\mathbb{E}}
    \left[  \left.
     \left(  \overline{\beta} + \overline{\nu} \right)  x_i
    + \left( \widehat{\nu} + \widetilde{\nu} +\widetilde{\beta} \right) x_i^2
            \right|  \mathscr{D}_{i^+}^{\rm{con}}
    \right] \nonumber
\\
&
    \ \ \ \
    -
     \left( \widehat{\nu} + \widetilde{\nu}  \right)
     {\mathbb{E}}
    \left[  \left.
       {\left( x_i - \widehat{x}_{ \left. i \right| i}  \right)}^2
            \right|  \mathscr{D}_{i^+}^{\rm{con}}
    \right]
    +
     {\mathbb{E}}
    \left[  \left.
    \Gamma_{i+1}
            \right|  \mathscr{D}_{i^+}^{\rm{con}}
    \right]. \label{writingCostUsingEstimationError}
\end{align}
Given the data~${\mathscr{D}}_{ { \left( i -1\right) }^{+}}^{\rm{con}} $
the part of the cost above that depends
on the encoding map~$\xi_i \left( \cdot \right)$ is
\begin{gather*}
    -
     \left( \widehat{\nu} + \widetilde{\nu}  \right)
     {\mathbb{E}}
    \left[  \left.
       {\left( x_i - \widehat{x}_{ \left. i \right| i}  \right)}^2
            \right|  \mathscr{D}_{i^+}^{\rm{con}}
    \right]
    +
     {\mathbb{E}}
    \left[  \left.
    \Gamma_{i+1}
            \right|  \mathscr{D}_{i^+}^{\rm{con}}
    \right].
\end{gather*}
Notice that the first term is the quantization variance of the
quantizer~$\xi_i \left( \cdot \right)$. This reduction of the encoder's
performance cost  to a sum of current and future quantization
distortions is possible because the term
$ J^{\rm{con}}_{i} \left( u_i , x_i \right)$ has been assumed to
be quadratic in $ x_i$ and $\widehat{x}_{ \left. i \right| i  }$.
The reduced
performance cost of the encoder is a function only of the
quantizer~$\xi_i \left( \cdot \right)$ and the conditional density
   $\rho_{ \left.  i \right| i - 1 }
    \left(  x \left| {\mathscr{D}}_{ {\left( i -1 \right) }^{-} }^{\rm{con}}
              \right.
   \right)
   $.
Indeed, given the data~${\mathscr{D}}_{ {\left( i -1 \right)}^{-} }^{\rm{con}}$
this cost is the following average:
\begin{align*}
 \Gamma_{i}
  \left(  \xi_{i} \left(  \cdot \right) ;
    {\mathscr{D}}_{ { \left(  i -1 \right) }^{+} }^{\rm{con}}
  \right)
&
 =
 \sum_{\text{cells} ~\Delta }
 {\mathbb{P}} \left[  x_{i} \in \Delta
                     \left|
           {\mathscr{D}}_{{\left(i -1\right)}^{-}}^{\rm{con}}, u_{i-1} \right.
            \right] \cdot
\Bigl\{ {\mathbb{E}} \left[   \Bigl.
 \Gamma_{i+1}
  \left(
    {\mathscr{D}}_{ {  i  }^{-} }^{\rm{con}}
  \right)
                             \Bigr|
    {\mathscr{D}}_{ { \left( i -1 \right)  }^{+}}^{\rm{con}},
    x_{i} \in \Delta
  \right]
\Bigr\}
\\
&
 \:
 +
 \sum_{\text{cells} ~\Delta }
 {\mathbb{P}} \left[  x_{i} \in \Delta
                     \left|
      {\mathscr{D}}_{ {  \left(  i -1 \right) }^{-} }^{\rm{con}}, u_{i-1}\right.
            \right] \cdot
\Bigl\{
  \lambda \
   {\mathbb{E}}
   \left[
          { \left( x_{i} - {\widehat{x}}_{\left.i\right| i}\right)}^2
        \left| {\mathscr{D}}_{ {  \left(  i -1 \right)}^{+} }^{\rm{con}}
                    , x_{i} \in \Delta
        \right.
   \right]
\Bigr\},
\end{align*}
where $\lambda = - \left( \widehat{\nu} + \widetilde{\nu} \right)$.
The cost $\Gamma_{i}$ does depend on both $\xi_{i}\left( \cdot
\right) $ and $u_i$, but for given data~$
 {\mathscr{D}}_{ { \left( i -1 \right) }^{-} }^{\rm{con}}
$
and control~$u_{i-1}$, the minimum of $\Gamma_{i}$ over all admissible
quantizers $\xi_{i} \left( \cdot \right) $ may possibly depend on
$
  {\mathscr{D}}_{ {  \left(  i -1 \right) }^{-} }^{\rm{con}}
$
but not on the control~$u_{i-1}$. To see this consider two arbitrary
possible values~$u, \widetilde{u}$ for $u_{i-1}$. Suppose that one is
given the quantizer
\begin{align*}
 \xi \left( x \right)
&
 =
 \begin{cases}
   1   & \text{if} \ x \in \left( -\infty , \delta_1 \right) , \\
   2   & \text{if} \ x \in \left( \delta_1 , \delta_2 \right), \\
   \vdots & \vdots \\
   N  & \text{if} \ x \in \left( \delta_{N-1}  , + \infty \right),
 \end{cases}
\end{align*}
meant for quantizing a random variable with the
density~$
   \rho_{ \left.  i \right| i -1 }
    \left(  x \left|
      {\mathscr{D}}_{ {  \left(  i -1 \right) }^{-} }^{\rm{con}} , u_{i-1} = u
              \right.
   \right)
   $.
Consider the quantizer~$\widetilde{\xi}$ constructed by taking each
cell~$\Delta = \left( \underline{\delta},  \overline{\delta} \right)$
in~${\xi}$, and generating a new cell~$\widetilde{\Delta} = \left(
\underline{\delta}  - u + \widetilde{u} , \, \overline{\delta} - u
+ \widetilde{u} \right)$, and stipulating that the new
quantizer~$\widetilde{\xi}$ assigns to the cell~$\widetilde{\Delta}$
the same channel input that the quantizer~${\xi}$ assigns to~$\Delta$.

Because of the linear evolution:
$
    x_i  =  a x_{i-1}  + u_{i-1}  +  w_{i-1},
$
and because the random variable $w_{i-1}$ is independent of the data
${\mathscr{D}}_{ {\left( i - 1 \right) }^{+} }^{\rm{con}}$, we have the
convolution relations:
\begin{align*}
   \rho \left( x \right)
&
  =
  {
    \left.
       \rho_{ \left.  i \right| i - 1 }
           \left(  {\frac{ \cdot - u } { a } }  \right)
       \circledast
       \rho_{ w }  \left(  \cdot \right)
    \right|
  }_{x} \ \text{and,} \\
     \widetilde{\rho} \left( x \right)
&
 =
 {
 \left.
  \rho_{ \left.  i \right| i - 1 }
     \left(  {\frac{ \cdot - \widetilde{u} } { a } }  \right)
       \circledast
  \rho_{ w }  \left(  \cdot \right)
 \right|
 }_{x},
\end{align*}
leading to the following symmetry w.r.t.~translations:
\begin{align}
\label{translationSymmetryAtAnyTime}
   \rho_{ \left.  i \right| i -1 }
    \left(  x - u  \left|
 {\mathscr{D}}_{ { \left(  i -1 \right) }^{-} }^{\rm{con}} , u_{i-1} = u
              \right.
   \right)
&
   =
   \rho_{ \left.  i \right| i -1 }
    \left(  x - \widetilde{u} \left|
     {\mathscr{D}}_{ { \left(  i -1 \right) }^{-} }^{\rm{con}}
             , u_{i-1} = \widetilde{u}
              \right.
   \right) .
\end{align}
Then we get the following equalities for each pair of cells~$\Delta ,
\widetilde{\Delta}$
\begin{align*}
 {\mathbb{P}} \left[  x_{i} \in \Delta
                     \left|
      {\mathscr{D}}_{ { \left(  i -1 \right) }^{-} }^{\rm{con}} ,
        u_{i-1} = u \right.
            \right]
&
 =
 {\mathbb{P}} \left[ x_{i} \in \widetilde{\Delta}
                     \left|
          {\mathscr{D}}_{ { \left(  i -1 \right) }^{-} }^{\rm{con}}
                                , u_{i-1} = \widetilde{u} \right.
            \right] ,
\\
 \Gamma_{i+1}
  \Bigl(
    {\mathscr{D}}_{ { \left(  i -1 \right) }^{-} }^{\rm{con}} , u_{i-1} = u,
    x_{i} \in \Delta
  \Bigr)
&
 =
 \Gamma_{i+1}
  \Bigl(
    {\mathscr{D}}_{ { \left(  i -1 \right) }^{-} }^{\rm{con}} ,
    u_{i-1} = \widetilde{u},
    x_{i} \in \widetilde{\Delta}
  \Bigr) , \ \text{and,}
\\
   {\mathbb{E}}
   \left[
          { \left( x_{i} - {\widehat{x}}_{\left.i\right| i}\right)}^2
        \left|
    {\mathscr{D}}_{ { \left(  i -1 \right) }^{-} }^{\rm{con}} , u_{i-1} = {u},
    x_{i} \in {\Delta}
        \right.
   \right]
&
  =
   {\mathbb{E}}
   \left[
          { \left( x_{i} - {\widehat{x}}_{\left.i\right| i}\right)}^2
        \left|
    {\mathscr{D}}_{ { \left(  i -1 \right)}^{-} }^{\rm{con}} ,
           u_{i-1} = \widetilde{u},
     x_{i} \in \widetilde{\Delta}
        \right.
   \right].
\end{align*}
Then the performance of any quantizer ${\xi}$ designed for $u_{i-1} =
u$ can be matched by $\widetilde{\xi}$ for $u_{i-1}=\widetilde{u}$,
and vice versa. Hence, we can conclude that for any~$u, \widetilde{u}$,
\begin{align*}
 \inf_{ \xi } \,
 \Gamma_{i}
  \Bigl( \xi \left(  \cdot \right)  ;
    {\mathscr{D}}_{ { \left(  i -1 \right) }^{-} }^{\rm{con}} , u_{i-1} = u,
  \Bigr)
&
 =
\inf_{\xi} \,
 \Gamma_{i}
  \Bigl(  \xi \left(  \cdot \right)  ;
    {\mathscr{D}}_{ { \left(  i -1 \right) }^{-} }^{\rm{con}} ,
        u_{i-1} = \widetilde{u},
  \Bigr).
\end{align*}
Notice that this optimal encoder now become controls-forgetting from
time~$i-1$.
\end{proof}
%%%%%%%%%%%%%%%%%%%%%%%%%%%%%%%%%%%%%%%%%%%%%%%%%%%%%%%%%%%%%%%%%
As the optimal control $u_T^*$ is a linear function on $\widehat{x}_{T|T}$,
the encoder~$\xi_T$ begets a performance cost that is quadratic in $x_T,
\widehat{x}_{T|T}$. Then the above Lemma renders the optimal encoding map
$\xi_T^*$ to be controls-forgetting from time $T-1$.
This reduction also holds at earlier times.
%%%%%%%%%%%%%%%%%%%%%%%%%%%%%%%%%%%%%%%%%%%%%%%%%%%%%%%%%%%%%%%%%
\begin{lemma} [Encoder separation for affine controls]
If the two conditions hold:
(A)~{for any admissible control strategy, an admissible encoder strategy
minimizing the performance cost~(\ref{objectiveFunction}) exists,
and }
(B)~{we apply as control strategy one affine from time~$\tau$:
$\mathcal{K}_i^{{\rm{mult}}, \tau} \Bigl(
\mathscr{D}_{i^+}^{\rm{con}}
\Bigr)
$} (from defn.~\ref{definition:multiplexedAffineControls}),
% \ref{eqn:affineMultiplexedLaw}},
then the following two conclusions hold:
(a)~{an encoder that is controls-forgetting from time~$\tau$
minimizes the partial LQ cost:
\begin{gather*}
     {\mathbb{E}}
    \left[  \left.
       x_{T+1}^2 + p \sum_{i= \tau + 1}^T x_i^2 + q \sum_{i= \tau}^T u_i^2
            \right|  \mathscr{D}_{\tau^+}^{\rm{con}}
    \right],
\end{gather*}
and,
}
(b)~{the shapes of the minimizing encoding maps from time~$\tau$
and their performance are independent of the data:
$
    \left\{
           \, u_{\tau - 1}^{\dagger} ,\,
          \{k_i\}_{i= \tau}^T , \, \{d_i \}_{i= \tau}^T
   \right\}.
$
}
\label{theorem:controlsForgettingEncoderIsOptimal}
\end{lemma}
%%%%%%%%%%%%%%%%%%%%%%%%%%%%%%%%%%%%%%%%%%%%%%%%%%%%%%%%%%%%%%%%%%%%%%%%%
\begin{proof}
We prove by mathematical induction. For a given control strategy, define:
\begin{align*}
   W_{T}
&
   =
     {\mathbb{E}}
    \left[  \left.
       x_{T+1}^2 + p x_T^2  + q  u_T^2
            \right|  \mathscr{D}_{{\left( T - 1 \right)}^+}^{\rm{con}}
    \right],
&
   W_{T}^*
&
   =
   \inf_{\xi_T\left( \cdot \right)}
   W_{T} ,
\\
   W_i
&
   =
     {\mathbb{E}}
    \left[  \left.
       p x_i^2  + q  u_i^2
            \right|  \mathscr{D}_{{\left( i - 1 \right)}^+}^{\rm{con}}
    \right]
    +
     {\mathbb{E}}
    \left[
               W_{i+1}^* \left(  \mathscr{D}_{{ i }^+}^{\rm{con}}
                         \right)
            \left| \mathscr{D}_{{\left( i -1 \right)}^+}^{\rm{con}} ,
                       \xi_i \left( \cdot \right)
           \right.
    \right] ,
&
   W_{i}^*
&
   =
   \inf_{\xi_{i}\left( \cdot \right)}
   W_{i}.
\end{align*}
%%%%%%%%%%%%%%%%%%%%%%%%%%%%%  Induction hypothesis
\begin{hypothesis}%{\it{Induction hypothesis}
%~[for time~$i$]
 For some time~$t=i$ such that~$\tau \le i < T$, we have
the following three assumptions: (1)~for every $j\ge i+1$, the optimal
value function~$
   W_{j}^* \left(  {\mathscr{D}}_{ {\left( j -1 \right)}^{-} }^{\rm{con}}
           \right)
$
takes the form:
\begin{gather*}
% \label{inductionHypothesisExpression}
 \alpha_j \sigma_w^2 + \overline{\alpha}_j
 +
 \widetilde{\beta}_j
   {\mathbb{E}}
   \left[
          {  x_{j} }^2
        \left| {\mathscr{D}}_{ {j }^{-} }^{\rm{con}} \right.
   \right]
 +
 \overline{\beta}_j \,
           \widehat{x}_{j \left| j \right.}
 +
 %{\underset
         %  { \triangleq \
         %                 \widetilde{\Gamma}_j \left(
         %                          \xi_j \left( \cdot \right)
         %           ; {\mathscr{D}}_{ {\left( j-1 \right) }^{+} }^{\rm{con}}
         %                                      \right)
         %  }
 % {\underbrace
 % {
  {\mathbb{E}}\left[  \left.
       \widetilde{\Gamma}_{j+1}^*
     \left(
         {\mathscr{D}}_{ { \left( j+ 1\right) }^{-} }^{\rm{con}}
     \right)
                      \right|
         {\mathscr{D}}_{ { j }^{-} }^{\rm{con}}
             \right]
 +
   {\widetilde{\lambda}}_j
   {\mathbb{E}}
   \left[
          { \left( x_{j} - {\widehat{x}}_{\left. j \right| j}\right)  }^2
        \left| {\mathscr{D}}_{ {j}^{-} }^{\rm{con}} \right.
   \right]
 % }
 % }
 % }
 ,
\end{gather*}
% where the $\left\{ \alpha_j \right\}, \left\{ \overline{\alpha}_j
% \right\}, \left\{ \widetilde{\beta}_j \right\}, \left\{
% \overline{\beta}_j \right\}, \left\{ \widetilde{\lambda}_j \right\}$
 where the $\alpha_j , \overline{\alpha}_j
 , \widetilde{\beta}_j ,
 \overline{\beta}_j , \widetilde{\lambda}_j $
are known non-negative real numbers for~$j\ge i+1$, (2)~for each such
$j$, the non-negative function $
 \widetilde{\Gamma}_{j+1}^*
 \left(
   {\mathscr{D}}_{ { j }^{-} }^{\rm{con}}
   \right)
$
is assumed to be independent of the partial waveform
 $\left\{ u_{j},  u_{j+1}, \ldots, u_{T} \right\}$, and (3)~the optimal
partial set of encoding maps~$
    {  \left\{
                \xi^*_{j} \left(  \cdot \right)
       \right\}
    }_{i+1}^{T}
$ is a set that is controls-forgetting from time~$i$.
\end{hypothesis}
%%%%%%%%%%%%%%%%%%%%%%%%%%%%%%%%%%%%%%%%%%%%%%%%%%%%%%%%%%%%%%%%%%%%%
\noindent
We will now show: if this hypothesis holds for time~$i$, then it holds
for time~$i-1$. Assuming that the partial set of optimal encoding maps~$
    {  \left\{
                \xi^*_{j} \left(  \cdot \right)
       \right\}
    }_{i+1}^{T}
$ are employed, we get:
\begin{align*}
  W_{i}
&
  =
     {\mathbb{E}}
    \left[  \left.
       p \, x_i^2  + q \, u_i^2
            \right|  \mathscr{D}_{{\left( i - 1 \right)}^+}^{\rm{con}}
    \right]
   +
     {\mathbb{E}}
    \left[
               W_{i+1}^* \left(  \mathscr{D}_{{ i }^+}^{\rm{con}}
                         \right)
            \left| \mathscr{D}_{{\left( i -1 \right)}^+}^{\rm{con}} ,
                       \xi_i \left( \cdot \right)
           \right.
    \right] ,
\\
&
   =
     p \, {\mathbb{E}}
    \left[
       x_i^2
            \left|  \mathscr{D}_{{\left( i - 1 \right)}^+}^{\rm{con}} \right.
    \right]
   +
     q \,  {\mathbb{E}}
    \left[
       u_i^2
            \left|  \mathscr{D}_{{\left( i - 1 \right)}^+}^{\rm{con}} \right.
    \right]
   +
  \alpha_{i+1} \sigma_w^2 + \overline{\alpha}_{i+1}
  +
  \widetilde{\beta}_{i+1}
   {\mathbb{E}}
   \left[
          {  x_{i+1} }^2
        \left| {\mathscr{D}}_{ {\left(i+1\right)}^{-} }^{\rm{con}} \right.
   \right]
\\
&
  \ \ \
 +
 \overline{\beta}_{i+1}
   {\mathbb{E}}
   \left[
           x_{i+1}
        \left| {\mathscr{D}}_{ {\left(i+1\right) }^{-} }^{\rm{con}} \right.
   \right]
 +
  {\mathbb{E}}\left[  \left.
       \widetilde{\Gamma}_{i+1}^*
     \left(
         {\mathscr{D}}_{ { \left( i+ 1\right) }^{-} }^{\rm{con}}
     \right)
                      \right|
         {\mathscr{D}}_{ { i }^{-} }^{\rm{con}}
             \right] ,
\\
&
  =
  \alpha_{i} \sigma_w^2 + \overline{\alpha}_{i}
  +
  \widetilde{\beta}_{i}
   {\mathbb{E}}
   \left[
          {  x_{i}^2 }
        \left| {\mathscr{D}}_{ {i}^{-} }^{\rm{con}} \right.
   \right]
 +
 \overline{\beta}_{i}
   {\mathbb{E}}
   \left[
           x_{i}
        \left| {\mathscr{D}}_{ {i }^{-} }^{\rm{con}} \right.
   \right]
\\
&
  \ \ \
  +
  {\mathbb{E}}\left[  \left.
       \widetilde{\Gamma}_{i+1}^*
     \left(
         {\mathscr{D}}_{ { \left( i+ 1\right) }^{-} }^{\rm{con}}
     \right)
                      \right|
         {\mathscr{D}}_{ { i }^{-} }^{\rm{con}}
             \right]
 +
   {\widetilde{\lambda}}_i
   {\mathbb{E}}
   \left[
          { \left( x_{i} - {\widehat{x}}_{\left. i \right| i}\right)  }^2
        \left| {\mathscr{D}}_{ {i}^{-} }^{\rm{con}} \right.
   \right]
 ,
\end{align*}
where, the coefficients:
\begin{align*}
   \alpha_i
&
   =
   \alpha_{i+1} + \widetilde{\beta}_{i+1},
&
   \overline{\alpha}_i
&
   =
   \overline{\alpha}_{i+1} + \widetilde{\beta}_{i+1} d_i^2
      + q \, d_i^2 + \overline{\beta}_{i+1} d_i,
\\
   \overline{\beta}_{i}
&
  =
   2 \left( q \, k_i d_i + a\widetilde{\beta}_{i+1} d_i
             + \widetilde{\beta}_{i+1} k_i d_i
     \right) ,
&
   \widetilde{\beta}_{i}
&
   =
   p_i + a^2 \widetilde{\beta}_{i+1} + k^2_i \widetilde{\beta}_{i+1}
      + 2ak_i \widetilde{\beta}_{i+1} + q \, k_i^2 \widetilde{\beta}_{i+1},
\\
   \widetilde{\lambda}_{i}
&
   =
   q \, k_i^2  + k_i^2 \widetilde{\beta}_{i+1}
      + 2ak_i \widetilde{\beta}_{i+1} .
&
&
\end{align*}
We have thus:
$
   W_i
   =
   {\mathbb{E}} \left[
                   {\text{A quadratic in}} \ x_i,
                       \widehat{x}_{i\left| i\right.}
                \right]
   +
   {\mathbb{E}} \left[
                   {\text{ Future distortions }}
                \right].
$
This and the fact that the encoder is controls-forgetting from time~$t=i$ meet the requirements of Lemma~\ref{lemmaOnQuantizerDistortion}.
Then we get the optimal encoding map~$\xi^*_i$ to be controls-forgetting
from time~$t=i-1$, and
\begin{align*}
    \widetilde{\Gamma}_i
&
    =
    \min_{\xi}
  {\mathbb{E}}\left[  \left.
       \widetilde{\Gamma}_{i+1}^*
     \left(
         {\mathscr{D}}_{ { \left( i+ 1\right) }^{-} }^{\rm{con}}
     \right)
                      \right|
         {\mathscr{D}}_{ { i }^{-} }^{\rm{con}}
             \right]
 +
   {\widetilde{\lambda}}_i \,
   {\mathbb{E}}
   \left[
          { \left( x_{i} - {\widehat{x}}_{\left. i \right| i}\right)  }^2
        \left| {\mathscr{D}}_{ {i}^{-} }^{\rm{con}} \right.
   \right]
\end{align*}
is independent of the partial set of controls
$
   { \left\{
       u_j
     \right\}
   }_{j=i-1}^{T}
$. From this it follows that the induction hypothesis is
also true for time $i-1$.
\end{proof}
%%%%%%%%%%%%%%%%%%%%%%%%%%%%%%%%%%%%%%%%%%%%%%%%%%%%%%%%%%%%%%%%%%%%%%%%%
\begin{lemma}[Certainty equivalence controls for controls-forgetting encoders]
Fix the switch time~$\tau$. If the encoder is preassigned to be one that
is controls-forgetting from time~$\tau$,
then the partial LQ cost
\begin{gather*}
     {\mathbb{E}}
    \left[  \left.
       x_{T+1}^2 + p \sum_{i= \tau + 1}^T x_i^2 + q \sum_{i= \tau}^T u_i^2
            \right|  \mathscr{D}_{\tau^-}^{\rm{con}}
    \right],
\end{gather*}
is minimized by the following control laws with a linear form:
For~$i\ge\tau$:
$
   u_i^*
   =
   k^*_i \,  {\widehat{x}}_{\left.i\right| i}
$.
\label{theorem:CElawsAreOptimal}
\end{lemma}
%%%%%%%%%%%%%%%%%%%%%%%%%%%%%%%%%%%%%%%%%%%%%%%%%%%%%%%%%%%%%%%%%%%%%%%%
\begin{proof}
Define the following cost-to-go at time~$t=T-1$:
$
   V_{T-1}
   =
   {\mathbb{E}}
      \left[
               W_{T} \left(  \epsilon_T \left( \cdot \right) ;
                   \mathscr{D}_{ {\left( T - 1 \right)}^+}^{\rm{con}}
                      \right)
      \right].
$
Because of Lemma~\ref{lemma:optimalUT},
\begin{align*}
   V_{T-1}
&
   =
   \sigma_w^2
   +
   \left(  p + {\frac{a^2 q}{q+1}} \right)
   {\mathbb{E}}
   \left[
          x_{T}^2
      \left| {\mathscr{D}}_{{\left( T-1\right)}^{-}}^{\rm{con}}, u_{T-1}\right.
   \right]
   +
   {\mathbb{E}}\left[  \left.
          {\left(  x_T - \widehat{x}_{ T \left|T \right. }  \right)}^2
                      \right|
         {\mathscr{D}}_{ { T }^{-} }^{\rm{con}}
             \right].
\end{align*}
Because the encoder is controls-forgetting from time~$\tau$, the last
term, which is the distortion due to the encoder~$\xi_T$, is independent
of the partial set of controls
$
   { \left\{
       u_i
     \right\}
   }_{i= \tau + 1}^{T}
$.
Hence the only part of~$V_{T-1}$ that depends on the control~$u_{T-1}$
is the quadratic
\begin{align*}
&
   q \, u^2 +
   \left(  p + {\frac{a^2 q}{q+1}} \right)
   {\mathbb{E}}
   \left[
          x_{T}^2
      \left| {\mathscr{D}}_{{\left( T-1\right)}^{-}}^{\rm{con}}, u_{T-1}\right.
   \right]
\\
&
   =
   q \, u^2 +
   \left(  p + {\frac{a^2 q}{q+1}} \right)
   \left\{
       a^2
       {\mathbb{E}}
       \left[
              x_{T}^2
          \left| {\mathscr{D}}_{{\left( T-1\right)}^{-}}^{\rm{con}}, u_{T-1}\right.
       \right]
       +
       2 a \, \widehat{x}_{ T-1 \left|T -1\right. } \, u_{T-1}
       +
       u_{T-1}^2
       +
       \sigma_w^2.
   \right\}
\end{align*}
Hence the best control law is:
$
    u_{T-1}^*
=
    - {\frac{a \left( p  +  {\frac{a^2 q}{q+1}} \right) }
             { q +  p  +  {\frac{a^2 q}{q+1}} }
      }
     \widehat{x}_{ T-1 \left|T -1\right. }
$,
and the resulting value function:
\begin{align*}
      V_{T-1}^*
&
      =
      \left( 1 + p  +  {\frac{a^2 q}{q+1}} \right)  \sigma_w^2
      +
     {\tfrac{a^2 q \left( p  +  {\frac{a^2 q}{q+1}} \right) }
             { q +  p  +  {\frac{a^2 q}{q+1}} }
      }
       {\mathbb{E}}
       \left[
              x_{T}^2
          \left| {\mathscr{D}}_{{\left( T-1\right)}^{-}}^{\rm{con}}, u_{T-1}\right.
       \right]
\\
&
   \  \ \
      +
      {\frac{ a^2 \left( p  +  {\frac{a^2 q}{q+1}} \right)}
            { q + p  +  {\frac{a^2 q}{q+1}}  }
      }
   {\mathbb{E}}\left[  \left.
          {\left(  x_{T-1} - \widehat{x}_{ T-1 \left|T-1 \right. }  \right)}^2
                      \right|
         {\mathscr{D}}_{ { \left( T - 2 \right) }^{+} }^{\rm{con}}
             \right]
       +
   {\mathbb{E}}\left[  \left.
          {\left(  x_T - \widehat{x}_{ T \left|T \right. }  \right)}^2
                      \right|
         {\mathscr{D}}_{ { T }^{-} }^{\rm{con}}
             \right].
\end{align*}
Repeating this procedure backwards in time, we get for times~$i \ge
\tau$, the optimal control laws are:
$
    u_{i}^*
=
    - k^*_i
     \widehat{x}_{ i \left| i\right. }
$,
where
$
    k^*_i = a \frac{\beta_{i+1}}{q + \beta_{i+1}}, \,
    \beta_i = p +   {\frac{a^2 q \beta_{i+1}}{ q + \beta_{i+1} }} ,
$
 and,
$
    \beta_{T+1} = 1.
$
\end{proof}
%%%%%%%%%%%%%%%%%%%%%%%%%%%%%%%%%%%%%%%%%%%%%%%%%%%%%%%%%%%%%%%%%%%%%%%%

%%%%%%%%%%%%%%%%%%%%%%%%%%%%%%%%%%%%%%%%%%%%%%%
%%\subsubsection{The obtained optimal controls are certainty equivalence controls}
%%%%%%%%%%%%%%%%%%%%%
\subsection{Main theorem}
Lemma~\ref{theorem:controlsForgettingEncoderIsOptimal} implies that for a
pre-assigned controller affine from time zero, there exist optimal
encoding maps that are
controls-forgetting from time zero. Lemma~\ref{theorem:CElawsAreOptimal} is
complementary. It implies that for a pre-assigned encoder that is controls
forgetting from time zero, the optimal control laws have linear forms.

For Design problem~\ref{fullyDynamicControls} an optimal pair of strategies have a similar
simplified structure. It is optimal to apply a combination of controls-forgetting
encoding and control laws linear in $\widehat{x}_{i|i}$. In general, this
controls-forgetting encoder does not minimize the aggregate squared estimation error.
The goal accomplished by an optimal encoder is slightly different. It is to
minimize a sum of state estimation errors with the time-varying weights~$\lambda_i$.
%%%%%%%%%%%%%%%%%%%%%%%%%%%%%%%%%%%%%%%%%%%%%%%%%%%%%%%%%%%%%%%
\begin{theorem}[Optimality of separation and certainty equivalence]
For Design problem~\ref{fullyDynamicControls}, with the discrete alphabet
channel of constant alphabet size, the quadratic performance
cost~(\ref{objectiveFunction}) is
minimized by applying the linear control laws
\begin{align}
\label{optimalControlLaw}
  u_t^*
&
  =
  -      k^*_t \,
         {\widehat{x}}_{\left.t\right| t}
\end{align}
in combination with the following encoder which is controls-forgetting from
time~$0$:
\begin{align}
\label{optimalEncoder}
    \epsilon_t^*
        \left( \, \zeta_t \,  ;
               \left\{ z_i \right\}_0^{t-1},
               \left\{ \epsilon_i  \left( \cdot \right) \right\}_0^{t-1}
        \right)
&
 =
 {\underset{ \epsilon \left( \cdot \right)
           }
           { \arg\inf
           }
 }
  \
 \Gamma_{i}
  \Bigl(  \epsilon \left( \cdot \right)
                \,  ;  \,
               \left\{ z_i \right\}_0^{t-1},
               \left\{ \epsilon_i  \left( \cdot \right)  \right\}_0^{t-1}
  \Bigr),
\end{align}
where,
$
    k^*_i = a \frac{\beta_{i+1}}{q + \beta_{i+1}},
    \beta_i = p +   {\frac{a^2 q \beta_{i+1}}{ q + \beta_{i+1} }} ,
    \beta_{T+1} = 1, \ \text{and} \
    \lambda_i = {\frac{a^2 \beta_{i+1}^2 }{ q + \beta_{i+1} }}
$ and where,
\begin{align*}
  \Gamma_t
&
  =
  \lambda_t \
  {\mathbb{E}} \left[
        { \left( \zeta_{t} - \widehat{\zeta}_{\left.t\right| t}
          \right)}^2
    \left|
                 \epsilon_t \left( \cdot\right) ,
                 {\mathscr{D}}_{ { \left( t-1\right)  }^{+} }^{\rm{con}}
   \right.
                            \right]
  +
  {\mathbb{E}} \left[
        \Gamma_{t+1}^* \left(
                          \overline{x}_0, \sigma_0^2,
               \left\{ z_i \right\}_0^{t},
               \left\{ \epsilon_i  \left( \cdot \right)  \right\}_0^{t}
                      \right)
              \right],
\\
  \Gamma_T
&
  =
  {\mathbb{E}} \left[
        { \left( \zeta_{T} - \widehat{\zeta}_{\left.T\right| T}
          \right)}^2
    \left|
                 \epsilon_T \left( \cdot\right) , \,
                          \overline{x}_0, \sigma_0^2,
               \left\{ z_i \right\}_0^{T-1},
               \left\{ \epsilon_i  \left( \cdot \right)  \right\}_0^{T-1}
   \right.
                            \right] ,
\\
 \Gamma_t^*
&
 =
 {\underset{ \epsilon \left( \cdot \right)
           }
           { \inf
           }
 }
  \
 \Gamma_t \left( \epsilon  \right) .
\end{align*}
Moreover, this control law is a certainty equivalence law.
\label{theorem:CElawsForCFencoders}
\end{theorem}
%%%%%%%%%%%%%%%%%%%%%%%%%%%%%%%%%%%%%%%%%%%%%%%%%%%%%%%%%%%%%%%
\begin{proof}
Starting with the result of Lemma~\ref{lemma:optimalUT} as a seed,
repeatedly apply in sequence
Lemmas~\ref{theorem:controlsForgettingEncoderIsOptimal},
\ref{theorem:CElawsAreOptimal}. This proves optimality of the above
combination.
Lemma~\ref{lemma:affineCElaws} implies that the controls laws of
\eqref{optimalControlLaw} are indeed certainty equivalence control
laws as per van der Water and Willems~\cite{vanDerWaterWillems1981}.
\end{proof}
%%%%%%%%%%%%%%%%%%%%%%%%%%%%%%%%%%%%%%%%%%%%%%%%%%%%%%%%%%%%
The optimal controller splits into a least square estimator
computing~$
         {\widehat{x}}_{\left.t\right| t}
$
and a time-dependent gain. Computing~${\widehat{x}}_{\left.t\right| t}$
is intrinsically hard because
quantization is a nonlinear operation. If one ignores this computational
burden, then, at least formally, the optimal controller resembles that
for the classical LQG optimal control problem.

%%%%%%%%%%%%%%%%%%%%%%%%%%%%%%%%%%%%%%%%%%%%%%%%%%%%%%%%%%%%%%%%%%%%
Note that in general the sequence of weights $ \left\{  \lambda_i
\right\}_0^T $ depends on the parameters of the performance cost
including the control penalty coefficient~$q$. In the two special
cases:
\begin{enumerate}
 \item{the coefficients $q=0$, $p=1$, or}
 \item{the quantity $p + a^2 q - q > 0$ and the following equality holds:
       \begin{align*}
             %  {\frac{1}{2}}
             p + a^2 q -q
            + \sqrt{ { \left(  p + a^2 q -q \right)   }^2   +   4pq  }
       &
            = 2,
       \end{align*}
      }
\end{enumerate}
it turns out that the weights~$\beta_i \equiv 1 \, \forall i$, and hence
the weights $\lambda_i \equiv \frac{a^2}{q+1} \, \forall i$.
Thus in these special cases, optimal encoders `ignore' the
parameters of the performance cost and simply minimize the usual
aggregate squared error in state estimation.

%%%%%%%%%%%%%%%%%%%%%%%%%%%%%%%%%%%%%%%%%%%%%%%%%%%%%%%%%%%%%%%%%%
\subsection{Extension to the multivariable case\label{multivariableCase}}

Theorem~\ref{theorem:CElawsForCFencoders} can be extended to situations where the state, control, and noise signals are vectors, as well as where the objective function~\eqref{objectiveFunction} includes cross terms involving the state and control. We can also extend to the case where the sensor has access only to partial and noisy observations of the state. To carry out these extensions, we need no more than the standard arguments of LQG control. Below, we
mention only the key steps corresponding to the lemmas of Section~\ref{preliminaryLemmasSection}.
Consider a partially observed, linear
multivariable plant:
\begin{align}
 x_{t+1} & =  A x_t + B u_t + {E} w_t, \ \
 y_t   \  = C x_t + {D} v_t,
\label{multivariablePlant}
\end{align}
where the state~$x_t\in{\mathbb{R}}^n$,
the control~$u_t\in{\mathbb{R}}^m$,
the output~$y_t\in{\mathbb{R}}^p$,
the process noise~$w_t\in{\mathbb{R}}^{l_1}$, and
the measurement noise~$v_t\in{\mathbb{R}}^{l_2}$.
Let the two noise sequences~$\left\{ w_t \right\}, \left\{ v_t \right\}$
be IID sequences that are mutually independent of each other.

For any matrix~$M$, let $M^{\boldmath{\prime}}$ denote its transpose.
For all times, let $ {\mathbb{E}} [ w_t w_t^{\prime} ] = \Sigma_w$
and $ {\mathbb{E}} [v_t v_t^{\prime} ] = \Sigma_v$.
Let the performance objective be defined as
\begin{align}
J_{\text{general}} & =
       {\mathbb{E}} \left[ x_{T+1}^{\prime} \, P_{T+1} \, x_{T+1} \right]
   + \sum_{i=0}^{T}
       {\mathbb{E}} \left[
                      \begin{pmatrix}
                           x_i^{\prime}  & u_i^{\prime}
                      \end{pmatrix}
                      \begin{bmatrix}
                          P   & R^{\prime} \\
                          R   & Q
                      \end{bmatrix}
                      \begin{pmatrix}
                          x_i  \\ u_i
                      \end{pmatrix}
                   \right],
\label{modifiedObjectiveFunction}
\end{align}
where $P$ and $\begin{bmatrix}
P   & R^{\prime} \\
R   & Q
\end{bmatrix} $
are symmetric and positive semi-definite, and $Q$ is symmetric and
positive definite. It is easy to see that an extension of
Lemma~\ref{lemma:optimalUT} holds for the multivariable case.
Precisely, the optimal control law at the terminal decision time~$T$ is
$
u^*_T = -  K_T \, {\widehat{x}}_{T | T }
$ where
$
K_T  =  - {( B^{\prime} P_{T+1} B + Q )}^{-1} ( R + B^{\prime} P_{T+1} A )
%  \, {\widehat{x}}_{T | T },
$
and the optimal cost-to-go
$
   V^*_T \left( {\mathscr{D}}_{ T^{-} }^{\rm{con}} \right)
   =
   \textrm{trace} ( {E} \Sigma_w {E}^{\prime} ) +
 {\mathbb{E}} \left[
{ (x_T - {\widehat{x}}_{T|T} ) }^{\prime} M_T (x_T - {\widehat{x}}_{T|T} )
        \left|  {\mathscr{D}}_{ T^{-} }^{\rm{con}}  \right.
             \right]
+
   {\mathbb{E}} \left[  x_{T }^{\prime}
                ( P  + A^{\prime} P_{T+1} A - M_T )
                      x_{T }
%   R + B^{\prime} P_{T+1} A
        \left|  {\mathscr{D}}_{ T^{-} }^{\rm{con}}  \right.
               \right],
$
where the matrix
 $ M_{T}  =  K_T^{\prime} ( B^{\prime} P_{T+1} B + Q )  K_T $.
% is defined to be the filtering error
% covariance
Next, the following generalization of
Lemma~\ref{lemmaOnQuantizerDistortion} can be proved.
\begin{lemma}[Multivariable version of
Lemma~\ref{lemmaOnQuantizerDistortion}]
Assume the hypothesis of Lemma~\ref{lemmaOnQuantizerDistortion}
but for the multivariable, partially observed plant~\eqref{multivariablePlant}, the objective function~\eqref{modifiedObjectiveFunction}, and
the following new definition of $J^{\text{con}}_i$:
\begin{align*}
J^{\text{con}}_i & =
     \alpha  + \beta^{\prime} \,  x_i + {\widehat{\beta}}^{\prime}
     \, {\widehat{x}}_{i|i}
      +
                      \begin{pmatrix}
                           x_i^{\prime}  & {{\widehat{x}}_{i|i} }^{\prime}
                      \end{pmatrix}
                      \begin{bmatrix}
                         {\widehat{P}}   & {\widehat{R}}^{\prime} \\
                         {\widehat{R}}   & {\widehat{Q}}
                      \end{bmatrix}
                      \begin{pmatrix}
                          x_i  \\ {\widehat{x}}_{i|i}
                      \end{pmatrix}.
\end{align*}
Then, it is optimal to apply an encoding map at time~$t=i$
that does not depend on the data:
$\bigl( u_{i-1} , \, {\{\widetilde{\mathcal{K}}_{j}
\left( \cdot \right)\}}_{i}^{T} \,\bigr)$.
It also follows that the shapes of the  encoding maps
$
\bigl\{
       \xi_{j}^{CF, i} \left(  \cdot \right)
\bigr\}_{i+1}^{T}
$
and their performance do not depend on the control~$u_{i-1}$.
\end{lemma}

\emph{Proof Sketch:} We can rewrite the part of the cost-to-go that depends
on on the control~$u_{i-1}$. As in~\eqref{writingCostUsingEstimationError}, it is possible to
rewrite this in such a way that the only dependence on~$
\widehat{x}_{ \left. i \right| i} $ is through a quadratic form of the
estimation error~$x_i
- \widehat{x}_{ \left. i \right| i} $:
\begin{align*}
     {\mathbb{E}}
    \left[  \left.
    J_{i}^{\rm{con}} %  \left( u_i , x_i \right)
    +
    \Gamma_{i+1}
            \right|  \mathscr{D}_{i^+}^{\rm{con}}
    \right]
&
    = \alpha  +
     {\mathbb{E}}
    \left[
    \left.
   {\left(  \beta+
{\widehat{\beta}}\right) }^{\prime}   x_i
     +     x_i^{\prime}
      \left( {\widehat{P}} + {\widehat{R}} + {\widehat{R}}^{\prime} + {\widehat{Q}} \right) x_i
    \right|  \mathscr{D}_{i^+}^{\rm{con}}
   \right]
%    +
%     {\mathbb{E}}
%    \left[  \left.
%    \Gamma_{i+1}
%            \right|  \mathscr{D}_{i^+}^{\rm{con}}
%    \right]
\\
&
    \ \ \ \
     +
     {\mathbb{E}}
    \left[  \left.
       { \left( x_i - \widehat{x}_{ \left. i \right| i} \right)}^{\prime}
    \left( {\widehat{R}} + {\widehat{R}}^{\prime} + {\widehat{Q}} \right)
      \left( x_i - \widehat{x}_{ \left. i \right| i} \right)
 +   \Gamma_{i+1}
            \right|  \mathscr{D}_{i^+}^{\rm{con}}
    \right].
\end{align*}
The part of the RHS that depends on $\xi_{i}$ is:
\begin{gather*}
     {\mathbb{E}}
    \left[  \left.
    \Gamma_{i+1}
            \right|  \mathscr{D}_{i^+}^{\rm{con}}
    \right]
+
     {\mathbb{E}}
    \left[  \left.
       { \left( x_i - \widehat{x}_{ \left. i \right| i} \right)}^{\prime}
    \left( {\widehat{R}} + {\widehat{R}}^{\prime} + {\widehat{Q}} \right)
      \left( x_i - \widehat{x}_{ \left. i \right| i} \right)
            \right|  \mathscr{D}_{i^+}^{\rm{con}}
    \right].
\end{gather*}
The minimum of this quantity over different~$\xi_i$ will be
independent of~$u_{i-1}$ if the density~$\rho_{ \left. i \right| {i-1} }$
is symmetric w.r.t. translations in the control.
If the matrix~$A$ is invertible, then
$
\rho_{ \left.  A x_i \right| z_0^{i-1} } \left(  x \right)
=
\rho_{ \left. i -1 \right| {i-1} } \left(  A^{-1} x \right).
$
Let $u, {\widetilde{u}}$ be two possible values for~$u_{i-1}$.
Then: %like in the scalar case, we get
% have the following relations:
%  the various
% conditional densities of the state given encoder outputs:
\begin{align*}
%   \rho_{ \left.  A x_i \right| z_0^{i-1} } \left(  x \right)
% & =
%   \int_{\zeta \in {\mathbb{R}}^n, A \zeta = x}
%       { \rho_{ \left. i \right| {i-1} } \left(  \zeta \right) {d{\zeta}}}
%\\
   \rho_{ \left. i \right| {i-1} } \left( x \right)
&
  =
  {
    \left.
       \rho_{ \left.  A x_i \right| z_0^{i-1} }
           \left(  { \cdot - B u }   \right)
       \circledast
       \rho_{ {E}  w }  \left(  \cdot \right)
    \right|
  }_{x} \ \text{and,}
\\
   {\widetilde{\rho}}_{ \left. i \right| {i-1} } \left( x \right)
&
  =
  {
    \left.
       \rho_{ \left.  A x_i \right| z_0^{i-1} }
           \left(  { \cdot - B {\widetilde{u}} }   \right)
       \circledast
       \rho_{ {E}  w }  \left(  \cdot \right)
    \right|
  }_{x}.
\end{align*}
If the following three conditions hold:
%\begin{enumerate}
{(1)~the matrix~$A$ is invertible,}
{(2)~the conditional density
$ {\rho}_{ \left. i-1 \right| {i-1} } $
 is a `well-behaved' function, for example, a function
of bounded variation, and}
{(3)~the noise random variables $w_i, v_i$ have `well-behaved'
densities,}
%\end{enumerate}
then it is straightforward  to deduce the following symmetry w.r.t.
translations:
\begin{align*}
%\label{translationSymmetryAtAnyTime}
   \rho_{ \left.  i \right| i -1 }
    \left(  x - u  \left|
 {\mathscr{D}}_{ { \left(  i -1 \right) }^{-} }^{\rm{con}} , u_{i-1} = u
              \right.
   \right)
&
   =
   \rho_{ \left.  i \right| i -1 }
    \left(  x - \widetilde{u} \left|
     {\mathscr{D}}_{ { \left(  i -1 \right) }^{-} }^{\rm{con}}
             , u_{i-1} = \widetilde{u}
              \right.
   \right) .
\end{align*}
If the matrix~$A$ is not invertible, or if any of the relevant densities
have Dirac-delta functions, then too, this symmetry property holds.
Proving that needs some slightly more delicate arguments. The rest is similar to the proof of Lemma~\ref{lemmaOnQuantizerDistortion}.
\hfill $\square$
%\end{proof}

The remaining lemmas of Section~\ref{preliminaryLemmasSection} are
straightforward to generalize to the multivariable case.
Moreover, our results clearly extend also to the case of  deterministic,
time-varying coefficients of the plant dynamics and of the objective function.
% as well as to deterministic, time-varying coefficients for the terms of the objective function.

%%%%%%%%%%%%%%%%%%%%%%%%%%%%%%%%%%%%%%%%%%%%%%%%%%%%%%%%%%%%%%%%%%
%%%%%%%%%%%%%%%%%%%%%%%%%%%%%%%%%%%%%%%%
%%%%%%%%%%%%%%%%%%%%%%%%%%%%%%%%%%%%%%%%%%%%%%%%%%%%%%%%%%%%%%%
\section{Dynamic designs for other models of channels
\label{extensions}
}
Our results for Design problem~\ref{fullyDynamicControls} extend to other channel models. In this section, we study a handful of channel models, all
coming from within three broad classes of messaging a sequence of real
numbers. These are: (1)~quantized messaging, (2)~unquantized but
irregular, event-triggered sampling, and (3) unquantized messaging
corrupted by additive channel noise. For each of these channel models, we find that the dynamic LQ design problem gets a separated optimal solution despite the existence of a dual effect in the corresponding networked control systems. To obtain this design simplification, we also assume that at all times, the channel output is perfectly visible to the encoder. Thus in each one of our channel models, there will be an ideal, delay-free feedback channel copying the actual inputs for the controller back to the encoder.

Our results also clearly extend to the case where we allow
deterministic, time-varying coefficients for the plant equation, and for the quadratic performance costs. These results also apply to the case where the quantizer word-lengths at different times are deterministic but time-varying.
In this section, we use the performance cost in~\eqref{objectiveFunction}, where the communication cost $J^{\textrm{Comm}}$ takes a positive functional form depending on the channel model.
%\begin{align}
%\label{eqn:repeatedObjectiveFunction}
% J %\left(\mathcal{E}^{T-1}_0, \mathcal{K}_0^{T-1} \right)
%&
% =
% \mathbb{E} \left[  x^2_{T+1} +  p \sum^{T}_{i=1} x^2_i + q \sum^{T}_{i=0} u^2_i
%            \right]
%+
%% \mathbb{E} \left[
%   J^{\textrm{Comm}}
%%            \right] .
%\end{align}
To show these extensions for all the other channel models we study, we only need to find the appropriate versions of Lemma~\ref{lemmaOnQuantizerDistortion}.
Once this is done, all the steps in the proofs for
Lemmas~\ref{theorem:controlsForgettingEncoderIsOptimal}~-~\ref{theorem:CElawsAreOptimal} and Theorem~\ref{theorem:CElawsForCFencoders} can be repeated with. For each of the channel models we consider, an encoder that is controls-forgetting from time~$0$ will be optimal in combination with the certainty equivalence control laws of \eqref{optimalControlLaw}. %In what follows we will use the following symbols $\eta, \eta_t, \phi\left( \cdot \right) , \varphi\left( \cdot \right)$ and attach different meanings to them, depending on the channel model.
%%%%%%%%%%%%%%%%%%%%%%%%%%%%%%%%%%%%%%%%%%%%%%%%%%%%%%%%%%%%%
\subsection{Quantizer with its rate chosen real-time}
%%%%%%%%%%%%%%%%%%%%%%%%%%%%%%%%%%%%%%%%%%%%%%%%%%%%%%%%%%%%%%
We describe below Design problem~\ref{fullyDynamicControls} for quantized control where the quantization rate is to be chosen real-time. %We describe the situation where
The rate has an expense attached, and there may be both a common upper bound on the sizes of individual codewords and a separate upper bound on the average data rate over the entire horizon.
\subsubsection{Communication cost}
The channel is an error-free, discrete alphabet channel with a variable sized alphabet.
With each channel use, the size of the alphabet $\eta_t$, as well as the codeword $\nu_t \in \left\{ 1, 2, \ldots , \eta_t \right\}$, must be chosen causally by the encoder. %At every time~$t$, the channel input and output are the same: $\iota_t = z_t$. And
Let $\phi\left( \eta \right) = \log_2{\eta}$ be a measure of the data rates, and let the positive integer
$\overline{\eta}$ denote an upper limit on the alphabet size at any time.
Then the communication cost incurred at time~$t$ can be described
thus:
\begin{align*}
    \varphi_t \left( \eta_t \right)
&
    =
    \begin{cases}
    \phi \left(\eta_i\right)  & \ \text{if} \ \eta_i \le \overline{\eta},
    \\
    + \infty  & \ \text{if} \ \eta_i > \overline{\eta}.
    \end{cases}
\end{align*}
Let the positive real number ${\mathscr{R}} \le \overline{\eta}$ denote an upper
limit on the average data rate over the entire horizon. We define the communication
cost as follows:
\begin{align}
     J^{\textrm{Comm}}
&
     =
     \begin{cases}
 m \cdot \mathbb{E} \left[
     \sum_{i=0}^{T}  \varphi \left( \eta_i \right)
            \right]
      &
     \ \text{if} \ \sum_{i=0}^{T}  \varphi \left( \eta_i \right)
       \le  {\mathscr{R}} \cdot \left( T + 1 \right) ,
     \\
     + \infty &
     \ \text{if} \ \sum_{i=0}^{T}  \varphi \left( \eta_i \right)
       >  {\mathscr{R}} \cdot \left( T + 1 \right) ,
     \end{cases} \label{variableRateEncJComm}
\end{align}
where $m$ is a fixed non-negative scalar.
It is easy to see that the signals
$
         x_t, \,
$ $
         \left \{\iota_j \right \}_0^{t-1} ,
$ $
         \left \{ z_j \right \}_0^{t-1},
$ $
         \left \{\xi_j \right \}_0^{t-1}
$
are sufficient statistics for encoding decisions, where of course
$z_t = \left( \eta_t , \nu_t \right)$.
We now present a suitable version of
Lemma~\ref{lemmaOnQuantizerDistortion}.
%%%%%%%%%%%%%%%%%%%%%%%%%%%%%%%%%%%%%%%%%%%%%%%%%%%%%%%%%%%%%%%%%%%%%%
\begin{lemma}
[Variable rate controls-forgetting encoder optimal for affine controls]
\label{lemmaQuantizerRateChosenRealTime}
Fix time~$t=i$ and apply control laws affine from time~$i$.
Suppose that for all times~$j> i$
 we have optimal encoding policies~${\mathcal{E}}_j^* \left( \cdot \right) $
(rules for variable alphabet sizes~$\eta_j$ as well as actual
quantization maps) such that their shapes and performances are independent of the partial control
waveform $\left\{ u_{i},  \ldots, u_{T} \right\}$.  Then, for all
times~$j> i-1$ we get optimal encoding policies~${\mathcal{E}}_j^* \left(
 \cdot \right) $ such that their shapes and performances are independent of the slightly
longer waveform $\left\{ u_{i-1},  u_{i},  \ldots, u_{T} \right\}$.
\end{lemma}
%%%%%%%%%%%%%%%%%%%%%%%
\begin{proof}
Consider the encoder choice at time~$t=i$. For any fixed alphabet size $\eta$,
let ${\mathcal{E}}^{\eta *} \left( \cdot \right) $ be the encoder
possessing the two properties: (1)~its alphabet size equals~$\eta$,
and (2)~this encoder in combination with optimal policies for
the later encoders $ \left\{ {\mathcal{E}}_j^*  \right\}_{j=i+1}^T $
(meaning policies for variable alphabet sizes and quantization maps)
achieves the lowest possible values for the performance costs.
Where by performance cost of the encoder we mean those parts
of the performance cost that, once affine control policies are fixed,
depend on the encoder.

For every fixed~$\eta$, we know that ${\mathcal{E}}^{\eta *} \left( \cdot
\right) $ and the statistics of its outputs are independent of the
policy for control~$u_{i-1}$. Hence when this quantizer is used in
combination with an optimal set of later encoders, the quantization
distortion at time~$t=i$, and the statistics of channel outputs at
all times~$j \ge i$ become independent of the control value~$u_{i-1}$.
Likewise the communication costs incurred at times ~$j \ge i$ become
independent of the control value~$u_{i-1}$.
Since every admissible choice of $\eta_t$ leads to this property, the Lemma
is proved.
\end{proof}
\noindent
We now present the main result:
%%%%%%%%%%%%%%%%%%%%%%%%%%%%%%%%%%%%%%%%%%%%%%%%%
\begin{theorem}[Optimality of separation and certainty equivalence]
\label{theorem:CElawsForCFencodersVariableQuantizer}
For Design problem~\ref{fullyDynamicControls}, with the discrete alphabet
channel of variable alphabet size, the performance
cost~(\ref{objectiveFunction}) with communication cost~\eqref{variableRateEncJComm} is
minimized by applying the linear control laws
\begin{align*}
% \label{optimalControlLaw}
  u_t^*
&
  =
  -      k^*_t \,
         {\widehat{x}}_{\left.t\right| t}
\end{align*}
in combination with the following encoder which is controls-forgetting from
time~$0$:
\begin{align*}
%  \label{optimalEncoder}
    \epsilon_t^*
        \left( \, \zeta_t \,  ;
               \left\{ z_i \right\}_0^{t-1},
               \left\{ \epsilon_i  \left( \cdot \right) \right\}_0^{t-1}
        \right)
&
 =
 {\underset{ \epsilon \left( \cdot \right)
           }
           { \arg\inf
           }
 }
  \
 \Gamma_{i}
  \Bigl(  \epsilon \left( \cdot \right)
                \,  ;  \,
               \left\{ z_i \right\}_0^{t-1},
               \left\{ \epsilon_i  \left( \cdot \right)  \right\}_0^{t-1}
  \Bigr),
\end{align*}
where,
$
    k^*_i = a \frac{\beta_{i+1}}{q + \beta_{i+1}},
    \beta_i = p +   {\frac{a^2 q \beta_{i+1}}{ q + \beta_{i+1} }} ,
    \beta_{T+1} = 1, \ \text{and} \
    \lambda_i = {\frac{a^2 \beta_{i+1}^2 }{ q + \beta_{i+1} }}
$ and where,
\begin{align*}
  \Gamma_T
&
  =
 \begin{cases}
  + \infty ,
  &
  \rm{if}
  \ \sum_{i=0}^{T} \varphi \left( \eta_i \right) > {\mathcal{R}}
                    \cdot \left( T + 1 \right) ,
  \\
  {\mathbb{E}} \left[
        { \bigl( \zeta_{T} - \widehat{\zeta}_{\left.T\right| T}
          \bigr)}^2
         + m \cdot
             \varphi \left( \eta_{{}_T} \right)
    \left|
                 \epsilon_T \left( \cdot\right) , \,
                    %      \overline{x}_0, \sigma_0^2,
               \left\{ z_i, \epsilon_i  \left( \cdot \right) \, \right\}_0^{T-1}
%               \left\{ \epsilon_i  \left( \cdot \right)  \right\}_0^{T-1}
   \right.
                            \right]
  ,
  & \rm{otherwise,}
 \end{cases}
\\
  \Gamma_t
&
  =
  \lambda_t \
  {\mathbb{E}} \left[
        { \bigl( \zeta_{t} - \widehat{\zeta}_{\left.t\right| t}
          \bigr)}^2
         + m \cdot
             \varphi \left( \eta_t \right)
    \left|
                 \epsilon_t \left( \cdot\right) ,
                 {\mathscr{D}}_{ { \left( t-1\right)  }^{+} }^{\rm{con}}
   \right.
                            \right]
  +
  {\mathbb{E}} \left[
        \Gamma_{t+1}^* \left(
%                          \overline{x}_0, \sigma_0^2,
               \left\{ z_i , \epsilon_i  \left( \cdot \right)  \, \right\}_0^{t}
%               \left\{ \right\}_0^{t}
                      \right)
              \right],
\\
 \Gamma_t^*
&
 =
 {\underset{ \epsilon \left( \cdot \right)
           }
           { \inf
           }
 }
  \
 \Gamma_t \left( \epsilon  \right) .
\end{align*}
Moreover, this control law is a certainty equivalence law.
\end{theorem}
%%%%%%%%%%%%%%%%%%%%%%%%%%%%%%%%%%%%%%%%%%%%%%%%%%%%%%%%%%%%%%%
\begin{proof}
Starting with the result of Lemma~\ref{lemma:optimalUT} as a seed,
repeatedly apply in sequence
Lemmas~\ref{lemmaQuantizerRateChosenRealTime},
\ref{theorem:CElawsAreOptimal}. This proves optimality of the above
combination.
Lemma~\ref{lemma:affineCElaws} implies that the controls laws of
\eqref{optimalControlLaw} are indeed certainty equivalence control
laws as per van der Water and Willems~\cite{vanDerWaterWillems1981}.
\end{proof}
%%%%%%%%%%%%%%%%%%%%%%%%%%%%%%%%%%%%%%%%%%%%%%%%%%%%%%%%%%%%
%%%%%%%%%%%%%%%%%%%%%%%%%%%%%%%%%%%%%%%%%%%%%%%%%%%%%%%%%%%%%%
\subsection{Event-triggered sampling}
The second model provides instantaneous, error-free transmission of any input real number. It is suitable only with systems working in real-time, since it has infinite capacity in the Shannon sense. To make this channel model represent a bottleneck, one must limit how often the channel can be used over prescribed time intervals. This we do by charging a communication cost for transmissions. This channel model is suitable for loops with event-triggered sampling. We now summarize parallel developments for event-triggered messaging.
%Here the encoder transmits unquantized real numbers at selected event-triggered times. The best signal to sample and transmit is the state signal~$x_t$.
\subsubsection{Communication cost}
The channel is an ideal, delay-free continuous valued one with no amplitude constraints. We will stipulate that the input to the channel is either a special silence symbol or a real number. In either case, the output will be a faithful reproduction of the input. Hence, %$\iota_t \equiv z_t$. Then
the encoder for event-triggered sampling can be represented by the following map from plant output to channel input
\begin{align*}
  z_i
&
  =
  \begin{cases}
       x_i  & \text{if} \ x_i  \notin {\mathcal{S}}_i \\
       \text{SILENCE}  & \text{if} \ x_i  \in {\mathcal{S}}_i ,
  \end{cases}
\end{align*}
where policies for the silence sets~${\mathcal{S}}_i$ have to be measurable w.r.t. the filtration generated by the data~${\mathscr{D}}_{ { \left( i -1 \right) }^{+} }^{\rm{con}}$. Let $\eta_i$ denote the random number of state samples transmitted up to and including time~$t=i$. Then we can write
$
\eta_t = \sum_{i=0}^t  {\mathds{1}}_{ \left\{ x_t \notin
                                        {\mathcal{S}}_i\right\} }
$.
Let the non-negative number~$N_0 \le T + 1 $ denote an initial budget of
samples. This initial budget is a hard limit, and the total number of samples taken over the entire horizon can never exceed~$N_0$. Then we define the communication cost as follows:
\begin{align}
\label{communicationCost:ETsampling}
     J^{\textrm{Comm}}
&
     =
     \begin{cases}
 m \cdot \mathbb{E} \left[
     \eta_T
            \right]
     &
     \text{if} \ \eta_T
       \le N_0 ,
     \\
     + \infty
     &
     \text{if} \ \eta_T
       > N_0 ,
     \end{cases}
\end{align}
where $m$ is a fixed non-negative scalar.
It is easy to see that the signals
%\begin{gather*}
$
         x_t, \,
         \left \{ z_j \right \}_0^{t-1},
         \left \{\xi_j \right \}_0^{t-1} ,
         \left \{ \eta_j \right \}_0^{t-1}
$
%\end{gather*}
are sufficient statistics for sampling decisions. Note also that the
record of sample counts~$
         \left \{ \eta_j \right \}_0^{t-1}
$
can be causally deduced from the record of channel outputs~$
         \left \{ z_j \right \}_0^{t-1}
$.

If we set~$N_0$ to be a finite number less than the horizon length~$T+1$ and set the
multiplier~$m$ to zero, then we get a design problem with a fixed budget~$N_0$ and no
cost attached to any number of samples within the budget. If instead we set the
multiplier~$m$ to be some positive number and set the bound~$N_0$ to be~$T+1$,
then we get a design problem with no budget constraint but with a communication cost
growing linearly with the number of samples taken
over the entire horizon. These two kinds of design problems
and their hybrids will all be simultaneously
studied by examining the general case where $m$
can be any nonnegative number, and~$N_0$ any positive
number.
%%%%%%%%%%%%%%%%%%%%%%%%%%%%%%%%%%%%%%%%%%%%%%%%%%%%%%%%%%%%%%%%%%%%%%
\begin{lemma}
[Controls-forgetting sampler optimal for affine controls]
% [Word lengths chosen real time]
\label{lemmaCFsamplerForAffineControlsETsampling}
Fix time~$t=i$ and apply control laws affine from time~$i$.
Suppose that for all times~$j> i$
 the optimal silence sets~${\mathcal{S}}_j^* \left( \cdot \right) $
and their performances are independent of the partial control
waveform $\left\{ u_{i},  \ldots, u_{T} \right\}$.  Then, for all
times~$j> i-1$ the optimal silence sets~${\mathcal{S}}_j^* \left(
 \cdot \right) $  and their performances are independent of the slightly
longer waveform $\left\{ u_{i-1},  u_{i},  \ldots, u_{T} \right\}$.
\end{lemma}
%%%%%%%%%%%%%%%%%%%%%%%
%%%%%%%%%%%%%%%%%%%%%%%%%%%%%%%%%%%%%%%%%%%%%%%%%%%%%%%%%%%%%%%%%%%%%%%%%%%
\begin{proof}
As with proving Lemmas 3, 4 we carry out two steps. First we show that because the
cost-to-go is quadratic, the quantizer's objective at time $i$ is to minimize a
sum $\Gamma_i$ of current and future estimation distortions. Second we show that
the minimum of this sum distortion is independent of the control $u_{i-1}$. Thus
the encoder becomes controls-forgetting from time $i-1$.
\end{proof}
\noindent
The main result for event-triggered sampling is presented below.
%except for cosmetic changes to do with the nomenclature of the running communication cost.
%%%%%%%%%%%%%%%%%%%%%%%%%%%%%%%%%%%%%%%%%%%%%%%%%
\begin{theorem}[Optimality of separation and certainty equivalence for
event-triggered sampling]
\label{theorem:CElawsForETsampler}
For Design problem~\ref{fullyDynamicControls}, with the even-triggered
messaging  channel, the performance cost~(\ref{objectiveFunction}) with
communication cost~\eqref{communicationCost:ETsampling} is
minimized by applying the linear control laws
\begin{align*}
% \label{optimalControlLaw}
  u_t^*
&
  =
  -      k^*_t \,
         {\widehat{x}}_{\left.t\right| t}
\end{align*}
in combination with the following silence set which is controls-forgetting from
time~$0$:
\begin{align*}
%  \label{optimalEncoder}
    \mathcal{S}_t^*
        \left( \, \zeta_t \,  ;
               \left\{ z_i \right\}_0^{t-1},
               \left\{ \mathcal{S}_i  \left( \cdot \right) \right\}_0^{t-1}
        \right)
&
 =
 {\underset{ \mathcal{S} \left( \cdot \right)
           }
           { \arg\inf
           }
 }
  \
 \Gamma_{i}
  \Bigl(  \mathcal{S} \left( \cdot \right)
                \,  ;  \,
               \left\{ z_i \right\}_0^{t-1},
               \left\{ \mathcal{S}_i  \left( \cdot \right)  \right\}_0^{t-1}
  \Bigr),
\end{align*}
where,
$
    k^*_i = a \frac{\beta_{i+1}}{q + \beta_{i+1}},
    \beta_i = p +   {\frac{a^2 q \beta_{i+1}}{ q + \beta_{i+1} }} ,
    \beta_{T+1} = 1, \ \text{and} \
    \lambda_i = {\frac{a^2 \beta_{i+1}^2 }{ q + \beta_{i+1} }}
$ and where,
\begin{align*}
  \Gamma_T
&
  =
 \begin{cases}
  + \infty ,
  &
  \rm{if}
  \ \eta_T > N_0,
  \\
  {\mathbb{E}} \left[
        { \bigl( \zeta_{T} - \widehat{\zeta}_{\left.T\right| T}
          \bigr)}^2
         + m \cdot
             \varphi \left( \eta_{{}_T} \right)
    \left|
                 \mathcal{S}_T \left( \cdot\right) , \,
                    %      \overline{x}_0, \sigma_0^2,
               \left\{ z_i, \mathcal{S}_i  \left( \cdot \right) \, \right\}_0^{T-1}
%               \left\{ \mathcal{S}_i  \left( \cdot \right)  \right\}_0^{T-1}
   \right.
                            \right]
  ,
  & \rm{otherwise,}
 \end{cases}
\\
  \Gamma_t
&
  =
  \lambda_t \
  {\mathbb{E}} \left[
        { \bigl( \zeta_{t} - \widehat{\zeta}_{\left.t\right| t}
          \bigr)}^2
         + m \cdot
             \varphi \left( \eta_t \right)
    \left|
                 \mathcal{S}_t \left( \cdot\right) ,
                 {\mathscr{D}}_{ { \left( t-1\right)  }^{+} }^{\rm{con}}
   \right.
                            \right]
  +
  {\mathbb{E}} \left[
        \Gamma_{t+1}^* \left(
%                          \overline{x}_0, \sigma_0^2,
               \left\{ z_i , \mathcal{S}_i  \left( \cdot \right)  \, \right\}_0^{t}
%               \left\{ \right\}_0^{t}
                      \right)
              \right],
\\
 \Gamma_t^*
&
 =
 {\underset{ \mathcal{S} \left( \cdot \right)
           }
           { \inf
           }
 }
  \
 \Gamma_t \left( \mathcal{S}  \right) .
\end{align*}
Moreover, this control law is a certainty equivalence law.
\end{theorem}
%%%%%%%%%%%%%%%%%%%%%%%%%%%%%%%%%%%%%%%%%%%%%%%%%%%%%%%%%%%%%%%
\begin{proof}
Starting with the result of Lemma~\ref{lemma:optimalUT} as a seed,
repeatedly apply in sequence
Lemmas~\ref{lemmaCFsamplerForAffineControlsETsampling},
\ref{theorem:CElawsAreOptimal}. This proves optimality of the above
combination.
Lemma~\ref{lemma:affineCElaws} implies that the controls laws of
\eqref{optimalControlLaw} are indeed certainty equivalence control
laws as per van der Water and Willems~\cite{vanDerWaterWillems1981}.
\end{proof}
%%%%%%%%%%%%%%%%%%%%%%%%%%%%%%%%%%%%%%%%%%%%%%%%%%%%%%%%%%%%
%%%%%%%%%%%%%%%%%%%%%%%%%%%%%%%%%%%%%%%%%%%%%%%%%%%%%%%%%%%%%%%%%%%%%%%%%%%
\subsection{Messaging over an noisy linear channel}
This model is a generalization of the classical additive white Gaussian
noise~(AWGN) channel, where we let the channel noise be coloured and non-Gaussian. This channel accepts  real valued inputs $\iota_t$ and delivers outputs $z_t$ with noise added. For $0 \le t \le T$:
\begin{align*}
   z_t &= \iota_t + \chi_t,
\end{align*}
where the channel noise process $\{\chi_i \}$ is IID with mean zero and variance $\sigma_{\chi}^2 < \infty$. At time~$t$, the noise $\chi_t$ is independent of the state, controls and process noises up to and including time $t$. For this style of messaging, we describe a model that allows the encoder to choose the SNR for each message. Naturally the model will also specify costs incurred for choosing message SNRs. %This model asks for an encoder that maps raw sensor measurements into a sequence of real valued channel inputs; such a map must be causal, but can be nonlinear and memory-based. To allow only finite-rate communications, we charge energy costs for inputs to the channel.

\subsubsection{Communication cost}
Let the real-valued even function $\phi (\cdot)$ increase with
increasing magnitude of argument, and let $\phi (0) = 0$. An example is
the function $\phi (\iota) = \iota^2$. Let $\overline{\iota}$ denote an upper limit on inputs to the channel. Then the communication cost incurred at a time~$t$ can be described thus:
\begin{align*}
    \varphi_t
&
    =
    \begin{cases}
      \phi (\iota_t)
    &
      \text{if} \ \lvert \iota_t \rvert \leq \overline{\iota},
    \\
      + \infty
    &
      \text{if} \ \lvert \iota_t \rvert > \overline{\iota}.
    \end{cases}
\end{align*}
Let ${\mathcal{P}} \leq \phi (\overline{\iota})$ denote an upper limit on the average power of channel inputs over the entire horizon. We define the communication cost from time $t$ to the horizon end as follows:
\begin{align}
\label{communicationCostForadditiveNoiseChannel}
     J^{\textrm{Comm}}
&
    =
    \begin{cases}
 m \cdot  \mathbb{E} \left[
        \sum_ {j=t}^T \varphi
                 \left( \iota_j \right)
            \right]
      &
       \text{if} \ \sum_{j=0}^T \varphi \left( \iota_j \right)
                    \leq {\mathcal{P}} \cdot \left(T+1 \right)
      \\
       + \infty
      &
       \text{if} \ \sum_{j=0}^T \varphi
          \left( \iota_j \right) > {\mathcal{P}} \cdot \left(T+1 \right).
    \end{cases}
\end{align}
where $m$ is a fixed non-negative scalar.
%%%%%%%%%%%%%%%%%%%%%%%%%%%%%%%%%%%%%%%%%%%%%%%%%%%%%%%%%%%%%%%%%%
\subsubsection{Sufficient statistics and scope for the dual effect}
It is straightforward to see that
\begin{gather*}
         x_t, \,
         \left \{\iota_j \right \}_0^{t-1},
         \left \{\xi_j \right \}_0^{t-1} ,
         \left \{ z_j \right \}_0^{t-1}
\end{gather*}
are sufficient statistics at the encoder.
As with quantized and event-triggered messaging, here too there is scope for
the dual effect since the encoding map may be nonlinear.

Clearly there is no dual effect introduced if the upper limit on inputs is removed,
and the encoder implements an affine encoder.
But in general, there is scope for introducing the dual effect.
If the encoder implements the quadratic encoder:
\begin{align*}
\xi_t^{quadratic} &= \eta x_t^2 ,
\end{align*}
then there is a second-order dual effect. Another example of an
admissible encoder that introduces the dual effect in the loop is one
that implements the piecewise-constant encoder:
\begin{align*}
     \xi_t
&
    =
    \begin{cases}
      - {\overline{\iota}}
    &
      \text{if} \ x_t \in \left( - \infty , -\theta \right),
    \\
       0
    &
      \text{if} \ x_t \in \left( -\theta , +\theta \right),
    \\
       \overline{\iota}
    &
      \text{if} \ x_t \in \left( +\theta , - \infty \right),
    \end{cases}
\end{align*}
where the threshold~$\theta$ is fixed. In fact, this encoder has nearly the
same input-output behaviour as the encoders considered in
examples~\ref{example:dualEffect} and~\ref{example:dualEffectInnovationEncoding}.
Using this parallel, one can setup an example of a loop with an additive
noise~(AN)
channel such that the dual effect is present. And, when there is a finite, hard limit on
amplitudes of channel inputs, then the dual effect is  present for any
encoder other than the trivial ones of the form:~$\xi_t \equiv
{\rm{constant}}$. As with other types of messaging, we can show that even though
the dual effect is present, the dynamic encoder-controller problem has a separated
solution and certainty equivalence controls are optimal.
%%%%%%%%%%%%%%%%%%%%%%%%%%%%%%%%%%%%%%%%%%%%%%%%%%%%%%%%%%%%%%%%%%%%%%
\begin{lemma}
[Controls-forgetting compander optimal for affine controls]
% [Word lengths chosen real time]
\label{lemmaCFcompanderForAffineControlsANchannel}
Fix time~$t=i$ and apply control laws affine from time~$i$.
Suppose that for all times~$j> i$
 the optimal encoding policies~${\mathcal{E}}_j^* \left( \cdot \right) $
and their performances are independent of the partial control
waveform $\left\{ u_{i},  \ldots, u_{T} \right\}$.  Then, for all
times~$j> i-1$ the optimal encoding policies~${\mathcal{E}}_j^* \left(
 \cdot \right) $  and their performances are independent of the slightly
longer waveform $\left\{ u_{i-1},  u_{i},  \ldots, u_{T} \right\}$.
\end{lemma}
%%%%%%%%%%%%%%%%%%%%%%%
%%%%%%%%%%%%%%%%%%%%%%%%%%%%%%%%%%%%%%%%%%%%%%%%%%%%%%%%%%%%%%%%%%%%%%%%%%%
\begin{proof}
As with proving Lemmas 3,4 we carry out two steps. First we show that because the
cost-to-go is quadratic, the quantizer's objective at time $i$ is to minimize a
sum $\Gamma_i$ of current and future estimation distortions. Second we show that
the minimum of this sum distortion is independent of the control $u_{i-1}$. Thus
the optimal encoder becomes controls-forgetting from time $i-1$.
\end{proof}
\noindent
The main result for communication over a noisy linear channel is
presented below.
\begin{theorem}[Optimality of separation and certainty equivalence for
additive noise channel]
\label{theorem:CElawsForCFcompander}
For Design problem~\ref{fullyDynamicControls}, with the additive noise
channel, the performance
cost~\eqref{objectiveFunction} with communication
cost~\eqref{communicationCostForadditiveNoiseChannel} is
minimized by applying the linear control laws
\begin{align*}
  u_t^*
&
  =
  -      k^*_t \,
         {\widehat{x}}_{\left.t\right| t}
\end{align*}
in combination with the following compander which is controls-forgetting from
time~$0$:
\begin{align*}
    \epsilon_t^*
        \left( \, \zeta_t \,  ;
               \left\{ z_i \right\}_0^{t-1},
               \left\{ \epsilon_i  \left( \cdot \right) \right\}_0^{t-1}
        \right)
&
 =
 {\underset{ \epsilon \left( \cdot \right)
           }
           { \arg\inf
           }
 }
  \
 \Gamma_{i}
  \Bigl(  \epsilon \left( \cdot \right)
                \,  ;  \,
               \left\{ z_i \right\}_0^{t-1},
               \left\{ \epsilon_i  \left( \cdot \right)  \right\}_0^{t-1}
  \Bigr),
\end{align*}
where,
$
    k^*_i = a \frac{\beta_{i+1}}{q + \beta_{i+1}},
    \beta_i = p +   {\frac{a^2 q \beta_{i+1}}{ q + \beta_{i+1} }} ,
    \beta_{T+1} = 1, \ \text{and} \
    \lambda_i = {\frac{a^2 \beta_{i+1}^2 }{ q + \beta_{i+1} }}
$ and where,
\begin{align*}
  \Gamma_T
&
  =
 \begin{cases}
  + \infty ,
  &
  \rm{if}
  \ \sum_{i=0}^{T} \varphi \left( \eta_i \right) > {\mathcal{P}}
                    \cdot \left( T + 1 \right) ,
  \\
  {\mathbb{E}} \left[
        { \bigl( \zeta_{T} - \widehat{\zeta}_{\left.T\right| T}
          \bigr)}^2
         + m \cdot
             \varphi \left( \eta_{{}_T} \right)
    \left|
                 \epsilon_T \left( \cdot\right) , \,
                    %      \overline{x}_0, \sigma_0^2,
               \left\{ z_i, \epsilon_i  \left( \cdot \right) \, \right\}_0^{T-1}
%               \left\{ \epsilon_i  \left( \cdot \right)  \right\}_0^{T-1}
   \right.
                            \right]
  ,
  & \rm{otherwise,}
 \end{cases}
\\
  \Gamma_t
&
  =
  \lambda_t \
  {\mathbb{E}} \left[
        { \bigl( \zeta_{t} - \widehat{\zeta}_{\left.t\right| t}
          \bigr)}^2
         + m \cdot
             \varphi \left( \eta_t \right)
    \left|
                 \epsilon_t \left( \cdot\right) ,
                 {\mathscr{D}}_{ { \left( t-1\right)  }^{+} }^{\rm{con}}
   \right.
                            \right]
  +
  {\mathbb{E}} \left[
        \Gamma_{t+1}^* \left(
%                          \overline{x}_0, \sigma_0^2,
               \left\{ z_i , \epsilon_i  \left( \cdot \right)  \, \right\}_0^{t}
%               \left\{ \right\}_0^{t}
                      \right)
              \right],
\\
 \Gamma_t^*
&
 =
 {\underset{ \epsilon \left( \cdot \right)
           }
           { \inf
           }
 }
  \
 \Gamma_t \left( \epsilon  \right) .
\end{align*}
Moreover, this control law is a certainty equivalence law.
\end{theorem}
%%%%%%%%%%%%%%%%%%%%%%%%%%%%%%%%%%%%%%%%%%%%%%%%%%%%%%%%%%%%%%%
\begin{proof}
Starting with the result of Lemma~\ref{lemma:optimalUT} as a seed,
repeatedly apply in sequence
Lemmas~\ref{lemmaCFcompanderForAffineControlsANchannel},
\ref{theorem:CElawsAreOptimal}. This proves optimality of the above
combination.
Lemma~\ref{lemma:affineCElaws} implies that the controls laws of
\eqref{optimalControlLaw} are indeed certainty equivalence control
laws as per van der Water and Willems~\cite{vanDerWaterWillems1981}.
\end{proof}
%%%%%%%%%%%%%%%%%%%%%%%%%%%%%%%%%%%%%%%%%%%%%%%%%%%%%%%%%%%%%%%%%%%%%%%%%%%
%%%%%%%%%%%%%%%%%%%%%%%%%%%%%%%%%%%%%%%%%%%%%%%%%%%%%%%%%%%%%%%%%%%%%%%%%%%
% \subsection{Vector valued and time-varying systems}
%

 We might also add that for all of the above channel models, the
 results for Design problem~\ref{fullyDynamicControls} can also be extended to
 the case of vector valued states with only partial, noisy
 linear observations available at the sensor (encoder).
 Such a situation is no more complicated than that one where
 the encoder observes the state perfectly. In the partially
 observed case, the role of the `state' falls on the estimate
 produced by the encoder's Kalman filter.
%%%%%%%%%%%%%%%%%%%%%%%%%%%%%%%%%%%%%%%%%%%%%%%%%%%%%%%%%%%%%%%%%%%%%%%
%%%%%%%%%%%%%%%%%%%%%%%%%%%%%%%%%%%%%%%%%%%%%%%%%%%%%%%%%%%%%%%%%%%%%%%
%%%%%%%%%%%%%%%%%%%%%%%%%%%%%%%%%%%%%%%%%%%%
\section{Constrained encoder-controller design} \label{Examples}
We now use our understanding of the dynamic encoder-controller design problem~(Design problem~\ref{fullyDynamicControls}) to examine the constrained encoder-controller design problem~(Design problem~\ref{RestrictedControlsEncoders}) and the hold-waveform-controller and encoder design problem~(Design problem~\ref{ZOHcontrols}). In this section, we show that, in general, separation in design of encoder and controller is not optimal for these design problems. We do this by presenting a counterexample for each of these design problems. Some of these counterexamples illustrate that the distortion term in the cost-to-go lacks symmetry w.r.t.~translations~(\ref{translationSymmetryAtAnyTime}). Recall that this property was instrumental in ensuring separation in the dynamic encoder-controller design problem (see proof of Lemma~\ref{lemmaOnQuantizerDistortion}).

Thus, we begin with Example~\ref{Ex:CertaintyEquivalence_EncDes}, which
illustrates,
% shows,
through explicit calculations, that symmetry w.r.t.~translations does indeed occur in the dynamic encoder-controller design problem. Next, we impose a set of constraints on the decision makers of the closed-loop system in Examples~\ref{Ex:ConstrainedEncoder_NoSep}-\ref{Ex:ConstrainedIntervalControl_NoSep}, which have the effect of removing the symmetry w.r.t.~translations. For these cases, we show that separation in design is no longer optimal. In Example~\ref{example:ZOHetSamplingOneSample}, we illustrate that separation is not optimal when the control signals are held constant over random epochs.

\subsection{Symmetry~w.r.t.~translations leads to separation}

We present a simple example of a dynamic encoder-controller design problem; the encoder is specified in a parametric form, but the choice of the parameters can be dynamic, with no restrictions on the set of parameters. We show that the optimal controller uses the certainty equivalence law.
\begin{example} \label{Ex:CertaintyEquivalence_EncDes}
For the linear plant~(\ref{plant}), with initial state $x_0$ given by a
zero mean Gaussian with variance $\sigma_x^2$, and process noise $w_k$
given by a zero mean Gaussian with finite variance $\sigma^2_w$, let the
horizon length be $T=2$. Let the cost coefficients~$p$ and $q$ remain
unspecified. Let the channel alphabet be the discrete set~$\left\{1,2\right\}$. The controller receives a quantized version of the state, denoted $z_k$ and given by
\begin{equation*}
z_k = \begin{cases}
1 & \textrm{if} \; x_k \le \delta_k \; , \\
2 & \textrm{otherwise} \; .
\end{cases}
\end{equation*}
The quantizer thresholds $\delta_0$ and $\delta_1$ are to be chosen along with the control signals $u_0$ and $u_1$, to jointly minimize the two-step horizon control cost.
%\end{example}

%\emph{Solution: }
%\begin{proof}[Demonstration of certainty equivalence]
We use dynamic programming to find the optimal values for $u_1$, $\delta_1$ and $u_0$, and $\delta_0$, in the specified order. From Lemma~\ref{lemma:optimalUT}, we know that $u_1^*$ is given by the certainty equivalence law as $-\frac{a}{q+1} \widehat{x}_{1|1}$, where the MMSE estimate of $x_1$ is given by $\widehat{x}_{1|1} = \mathbb{E} \left[x_1 \big| \left\{ z_i \right\}_0^1 \right]$.

Then, let us consider the cost-to-go at the previous time step,
\begin{align}
V_0 = \min_{u_0,\delta_1} \mathbb{E} \bigg[ a^2(p+a^2) x_0^2 &+ (q+p+a^2)u_0^2 + 2a(p+a^2)x_0u_0 - \frac{a^2}{q+1}\widehat{x}^2_{1|1} \; \bigg| \; z_0 \bigg] + \kappa \; , \label{Eq:CostToGoV0}
\end{align}
where $\kappa = (1+p+a^2)\sigma_w^2$. The above cost-to-go is to be minimized by selecting a suitable $u_0$ and $\delta_1$ simultaneously. To do this, we first need to find an expression for $\mathbb{E} \left[ \widehat{x}_{1|1}^2 {}\big|{} z_0 \right]$. The encoder outputs at times $0,1$ tell us the quantization cells in which $x_0$ and $x_1$ lie. We use this information to find an expression for the estimate $\widehat{x}_{1|1}$, as shown in Appendix~\ref{App:ExplicitCalc}, and rewrite the cost-to-go as
\begin{equation} \label{Eq:CostToGoV0_Rewritten}
\begin{aligned}
V_0 = \min_{u_0,\delta_1} \mathbb{E} &\left[ a^2(p+a^2) x_0^2 + \overbrace{(q+p+a^2\frac{q}{q+1})u_0^2 + 2a(p+a^2\frac{q}{q+1})x_0u_0}^{\textrm{function of } u_0} \bigg| z_0 \right] \\
&\quad \quad \quad \quad \quad \quad \; - \underbrace{\frac{a^2}{q+1} \frac{\sum_{j=1}^N \vartheta^2\left( \frac{\varpi_{j-1}-u_0}{\sigma_2},\frac{\varpi_{j}-u_0}{\sigma_2} \right)}{\mathbb{P}\left(x_0 \in \left(\theta_{i-1},\theta_i \right) \right)}}_{\triangleq \Gamma_1: \quad \textrm{function of } u_0 \textrm{ and } \mathcal{E}_1} + (1+p+a^2)\sigma_w^2 \; ,
%&\quad \quad \quad \quad \quad \quad \; - \underbrace{\frac{a^2}{q+1} \frac{\sum_{j=1}^N \vartheta^2\left( \frac{\delta_{j-1}-u_0}{\sigma_2},\frac{\delta_{j}-u_0}{\sigma_2} \right)}{\mathbb{P}\left(x_0 \in \left(\theta_{l-1},\theta_l \right) \right)}}_{\triangleq \Gamma_1: \quad \textrm{function of } u_0 \textrm{ and } \mathcal{E}_1} + (1+p+a^2)\sigma_w^2 \; ,
\end{aligned}
\end{equation}
where $\sigma_2^2 = \sigma_w^2 + a^2 \sigma_x^2$. The term $\vartheta(\underline{r},\bar{r})$ in the above equation is given by
\begin{equation} \label{Eq:lambda}
\begin{aligned}
\vartheta(\underline{r},\bar{r}) = \bigg[ &-a\sigma_x g\left(\frac{\theta_i}{\sigma_x} \right) G\left(r\frac{\sigma_2}{\sigma_w} - \theta_i\frac{a}{\sigma_w} \right) - \sigma_2 g(r)G\left(\frac{\theta_i}{\sigma_1} - r\frac{a\sigma_x}{\sigma_w}\right) \\
&+ a\sigma_x g\left(\frac{\theta_{i-1}}{\sigma_x}\right)G\left(r\frac{\sigma_2}{\sigma_w} - \theta_{i-1}\frac{a}{\sigma_w}\right) + \sigma_2 g(r)G\left(\frac{\theta_{i-1}}{\sigma_1} - r\frac{a\sigma_x}{\sigma_w}\right) \bigg]_{r=\underline{r}}^{\bar{r}} \; ,
%\vartheta(\underline{r},\bar{r}) = \bigg[ &-a\sigma_x g\left(\frac{\theta_l}{\sigma_x} \right) G\left(r\frac{\sigma_2}{\sigma_w} - \theta_l\frac{a}{\sigma_w} \right) - \sigma_2 g(r)G\left(\frac{\theta_l}{\sigma_1} - r\frac{a\sigma_x}{\sigma_w}\right) \\
%&+ a\sigma_x g\left(\frac{\theta_{l-1}}{\sigma_x}\right)G\left(r\frac{\sigma_2}{\sigma_w} - \theta_{l-1}\frac{a}{\sigma_w}\right) + \sigma_2 g(r)G\left(\frac{\theta_{l-1}}{\sigma_1} - r\frac{a\sigma_x}{\sigma_w}\right) \bigg]_{r=\underline{r}}^{\bar{r}} \; ,
\end{aligned}
\end{equation}
where $\sigma_1^2 = \sigma_x^2 \sigma_w^2 / \sigma_2^2$ and $g(\cdot)$ and $G(\cdot)$ are the probability distribution function and cumulative distribution function, respectively, of the standard normal distribution. The quantization cells for $x_0$ and $x_1$ are denoted by $(\theta_{i-1},\theta_i)$ and $(\varpi_{j-1},\varpi_{j})$ corresponding to the encoder outputs $z_0 = i$ and $z_1 = j$, respectively.

The quantization distortion term $\Gamma_1$ in~\eqref{Eq:CostToGoV0_Rewritten} possesses symmetry~w.r.t.~translations, as defined in~(\ref{translationSymmetryAtAnyTime}). Thus, for any value of the control signal $u_0$, the minimum value is given by $\Gamma_1^*(\mathcal{E}_1)$, a term that depends only on the encoder.
Then, the cost-to-go with respect to the control signal $u_0$ comprises of only the terms in the first row in (\ref{Eq:CostToGoV0_Rewritten}). Hence, we obtain separation. Furthermore, the optimal control signal is given by the certainty equivalence law, $u_0^{\textrm{CE}} = -\frac{a(p+a^2\frac{q}{q+1})}{p+q+a^2\frac{q}{q+1}} \widehat{x}_{0|0}$. Thus, the certainty equivalence property holds for this setup.
\end{example}

\begin{figure}[tb]
\begin{center}
\includegraphics*[scale=0.435,viewport=20 0 1050 400]{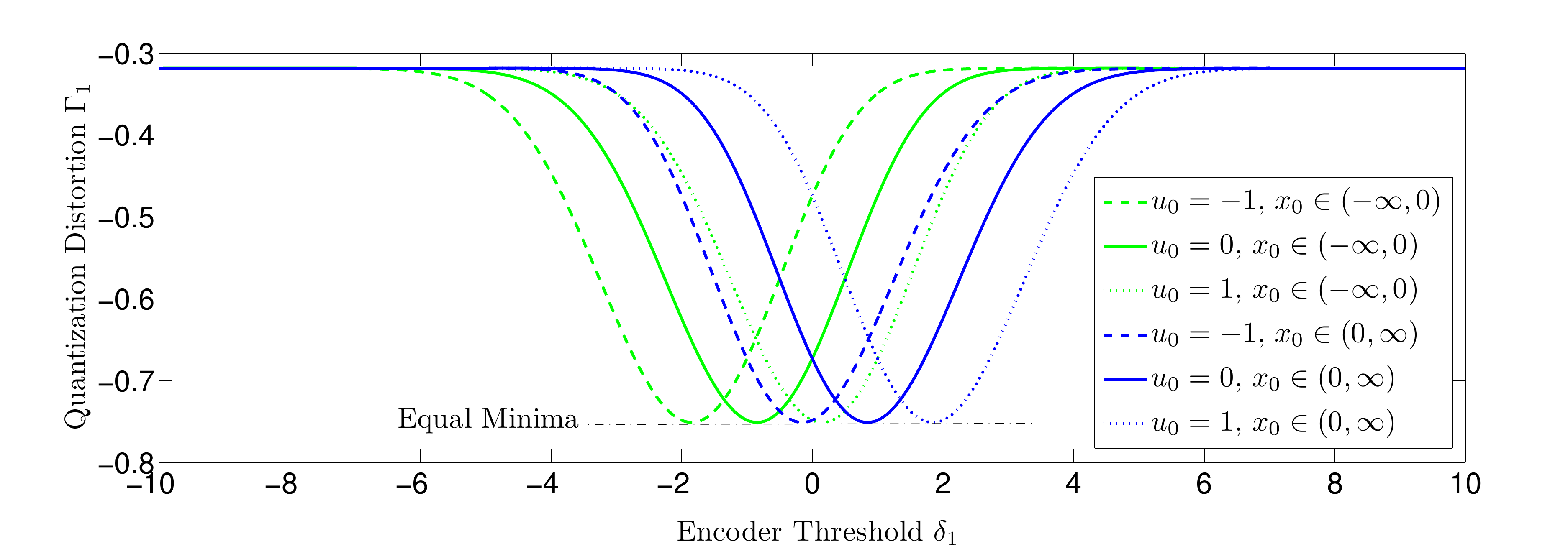}
\caption{This plot illustrates the \emph{symmetry w.r.t.~translations} of the quantization distortion term $\Gamma_1$ in (\ref{Eq:CostToGoV0_Rewritten}). Different values of $u_0$ result in the same minimum value for $\Gamma_1$ at different values of $\delta_1$, thus resulting in separation and certainty equivalence in Example~\ref{Ex:CertaintyEquivalence_EncDes}. }
\label{Fig:TranslationSymmetry}
\end{center}
\vspace{-5mm}
\end{figure}
We illustrate symmetry~w.r.t.~translations in
Figure~\ref{Fig:TranslationSymmetry}. For the choice of parameters
$a=1$, $p = 1$ and $q = 1$, we evaluate the quantization distortion term
$\Gamma_1$ from the above example and show that the minimum that this
function attains over the range of the quantizer threshold $\delta_1$ is
invariant for different values of $u_0$. To evaluate the cost-to-go, we
make an arbitrary choice: $\delta_0 = 0$, for the quantizer threshold at
time $k=0$, and we compute the estimates and probabilities using this choice.

\subsection{Optimal constrained encoder}
We now impose a restriction on the choice of encoder parameters. The one-bit quantizer that we consider in the previous example selects two semi-infinite intervals as the quantizer cells, $\Delta_1 = (-\infty,\delta_k]$ and $\Delta_2 = (\delta_k,\infty)$. We restrict the choice of the quantizer threshold to a constraint set, such that $\delta_k \in \Theta$. In the following example, we see that separation is lost for this constrained optimization problem.

\begin{example} \label{Ex:ConstrainedEncoder_NoSep}
Consider the same setup as in Example~\ref{Ex:CertaintyEquivalence_EncDes}, with the restriction that the quantizer threshold be chosen from the set $\Theta = (-1,1)$. The quantizer thresholds $\delta_0 \in \Theta$ and $\delta_1 \in \Theta$ are to be chosen along with the control signals $u_0$ and $u_1$, to jointly minimize the two-step horizon control cost.

We follow the same procedure as before. The optimal control signal $u_1$ is given by the certainty equivalence law as $u_1^* = u_1^{\textrm{CE}}$. This gives us the same cost-to-go $V_0$ from (\ref{Eq:CostToGoV0}). Evaluating $\Gamma_1$ for the parameters $a=1$, $p = 1$ and $q = 1$, we plot it over a range of quantizer thresholds $\delta_1 \in \Theta$, for three arbitrary choices of $u_0$, in Figure~\ref{Fig:RestrEnc_NoSep}. By restricting the range of quantizer thresholds to $\Theta$, we do not permit all the curves to reach their minima from~Figure~\ref{Fig:TranslationSymmetry}. In particular, the minima for $u_0 = -1$, when $x_0 \in (-\infty,0)$, and $u_0 = 1$, when $x_0 \in (0,\infty)$ are higher than before. Thus, the minimum value of $\Gamma_1$ obtained over the range of $\delta_1$ now varies depending on the choice of $u_0$. Consequently, there is no longer a symmetry w.r.t.~translations, and separation cannot be achieved using the proof of Theorem~\ref{theorem:controlsForgettingEncoderIsOptimal}. Furthermore, the optimal control signal $u_0^*$ must be chosen along with $\delta_1^*$ to optimize the entire cost-to-go including the term $\Gamma_1$. Thus, $u_0^*$ does not just minimize a quadratic expression in this problem, and cannot be chosen independently of the encoding policy. Hence, separation in design of the controller and encoder is no longer optimal.
\end{example}
\begin{figure}[tb]
\begin{center}
\includegraphics*[scale=0.415,viewport=20 0 1075 400]{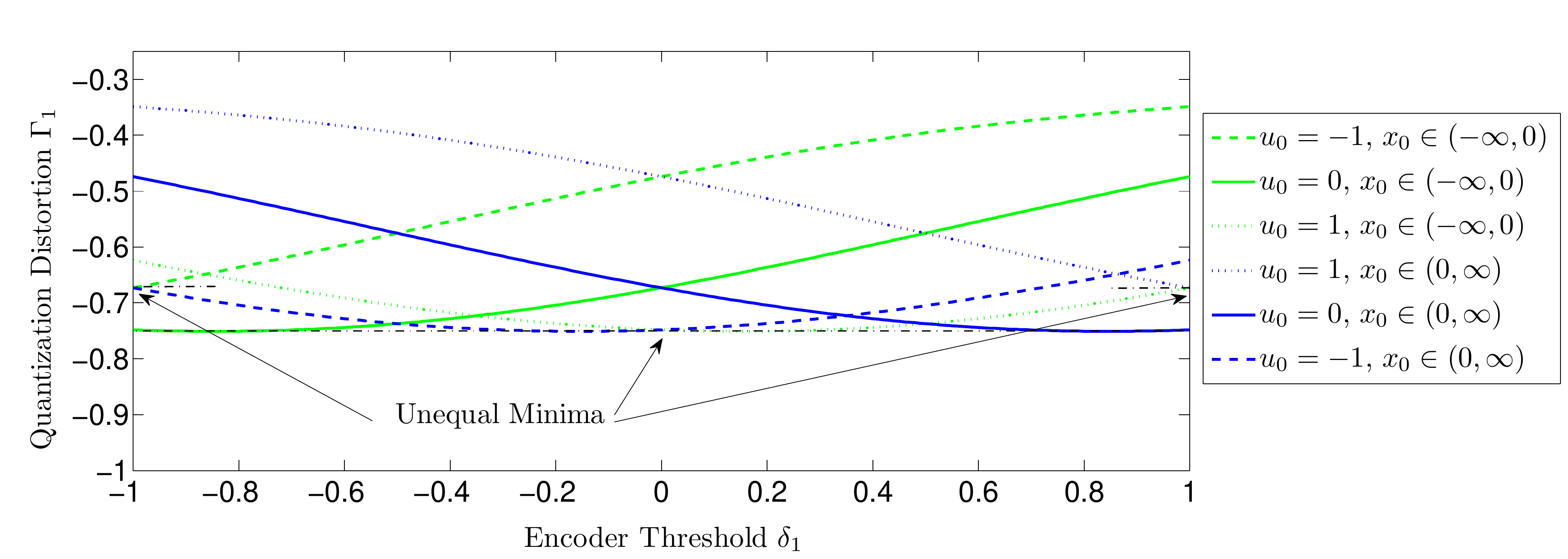}
\caption{This plot illustrates the lack of \emph{symmetry w.r.t.~translations} of $\Gamma_1$, when the quantizer thresholds are restricted to be chosen from an interval, such as in Example~\ref{Ex:ConstrainedEncoder_NoSep}. Different values of $u_0$ do not result in the same minimum value for $\Gamma_1$ over the range of $\delta_1$, thus resulting in a lack of separation and certainty equivalence. }
\label{Fig:RestrEnc_NoSep}
\end{center}
\vspace{-5mm}
\end{figure}

\subsection{Optimal constrained controller}
We now remove the restriction on the encoder parameters, and instead impose the following restriction on the controller: the controls are required to have limited range. Specifically, the control values at ever time step must come from a specified constraint set $\mathcal{U}$. We present two versions of this constraint: in case~$1$, our constrained control set $\mathcal{U}$ is discrete, and in case~$2$, the constrained control set is an interval $\mathcal{U} = \left(u_{\min},u_{max}\right)$.

\begin{example} \label{Ex:ConstrainedControl_NoSep}
Consider the same setup as in Example~\ref{Ex:CertaintyEquivalence_EncDes}, with the restriction that the control signal be chosen from a discrete set $\mathcal{U} = \{-1,0,1\}$. The quantizer thresholds $\delta_0$ and $\delta_1$ are to be chosen along with the control signals $u_0 \in \mathcal{U}$ and $u_1 \in \mathcal{U}$, to jointly minimize the two-step horizon control cost.

%\emph{Solution:}
%\begin{proof}[Illustration of Lack of separation]
The unconstrained minimizer for the cost-to-go at the terminal time is given by the certainty equivalent value $u_1^{\textrm{CE}}$. The best we can do, given the constraint set $\mathcal{U}$, is to choose the control value from the discrete set $\mathcal{U}$ that results in the lowest cost-to-go. Using this principle, we find the optimal control signal $u_1^*$ to be
\begin{equation*}
u_1^* = \begin{cases}
-1 & \widehat{x}_{1|1} \ge \frac{q+1}{2a} \; , \\
0 & \frac{q+1}{2a} \ge \widehat{x}_{1|1} \ge -\frac{q+1}{2a} \; ,\\
1 & \widehat{x}_{1|1} \le -\frac{q+1}{2a} \; .
\end{cases}
\end{equation*}
The optimality regions are identified by comparing $\min_{u_1 \in \mathcal{U}} V_1(u_1)$ evaluated at each permissible value of $u_1$, and determining the switching points.

The cost-to-go $V_0$, obtained by averaging over the three different cost-to-go functions obtained at time $k=1$, is given by
\begin{align*}
V_0 = \min_{u_0,\delta_1} \mathbb{E} \bigg[ a^2(p+a^2) x_0^2 &+ (q+p+a^2)u_0^2 + 2a(p+a^2)x_0u_0 + (-2a\widehat{x}_{1|1} + q + 1) \mathds{1}_{\left\{ \widehat{x}_{1|1} \ge \frac{q+1}{2a} \right\}} \\
&+ (2a\widehat{x}_{1|1} + q + 1) \mathds{1}_{\left\{ \widehat{x}_{1|1} \le -\frac{q+1}{2a} \right\}} \bigg| z_0 \bigg] + (1+p+a^2)\sigma_w^2 \; .
\end{align*}

We denote the terms in the above cost-to-go that directly depend on the choice of the encoder threshold $\delta_1$ as $\Gamma^{\textrm{RC}}_1$. Using the expression for $\widehat{x}_{1|1}$ and the posterior density for $x_1$ from Appendix~\ref{App:ExplicitCalc}, we compute $\Gamma^{\textrm{RC}}_1$ as
\begin{align*}
\Gamma^{\textrm{RC}}_1 &=  \mathbb{E} \bigg[(-2a\widehat{x}_{1|1} + q + 1) \mathds{1}_{\left\{ \widehat{x}_{1|1} \ge \frac{q+1}{2a} \right\}} + (2a\widehat{x}_{1|1} + q + 1) \mathds{1}_{\left\{ \widehat{x}_{1|1} \le -\frac{q+1}{2a} \right\}} \bigg| z_0 \bigg] \\
&= \sum_{j=1}^N \frac{\mathbb{P}\left(x_0 \in \left(\theta_{i-1},\theta_i \right), x_1 \in \left(\varpi_{j-1},\varpi_{j} \right) \right)}{\mathbb{P}\left(x_0 \in \left(\theta_{i-1},\theta_i \right) \right)} \bigg( (-2a\widehat{x}_{1|1} + q + 1) \mathds{1}_{\left\{ \widehat{x}_{1|1} \ge \frac{q+1}{2a} \right\}} \\
&\quad \quad \quad \quad \quad \quad \quad \quad \quad \quad \quad \quad \quad \quad \quad \quad \quad \quad + (2a\widehat{x}_{1|1} + q + 1) \mathds{1}_{\left\{ \widehat{x}_{1|1} \le -\frac{q+1}{2a} \right\}} \bigg) \; .
\end{align*}

Evaluating the above expression for parameters $a=1$, $p = 1$ and $q = 1$, and some arbitrary choice of quantizer threshold $\delta_0$, we plot $\Gamma^{\textrm{RC}}_1$ over a range of quantizer thresholds $\delta_1$, for different choices of $u_0$ from the set $\mathcal{U}$, in Figure~\ref{Fig:RestrCtrl_NoSep}. Notice that the minimum values of $\Gamma^{\textrm{RC}}_1$ obtained over the range of $\delta_1$ vary depending on the choice of $u_0$. In other words, there is no symmetry w.r.t.~translations. Consequently, a separation in design of the controller and encoder is no longer optimal.
\end{example}
%\end{proof}
%\qed \newline
\begin{figure}[tb]
\begin{center}
%\hspace{-1cm}
\includegraphics*[scale=0.45,viewport=20 0 1025 400]{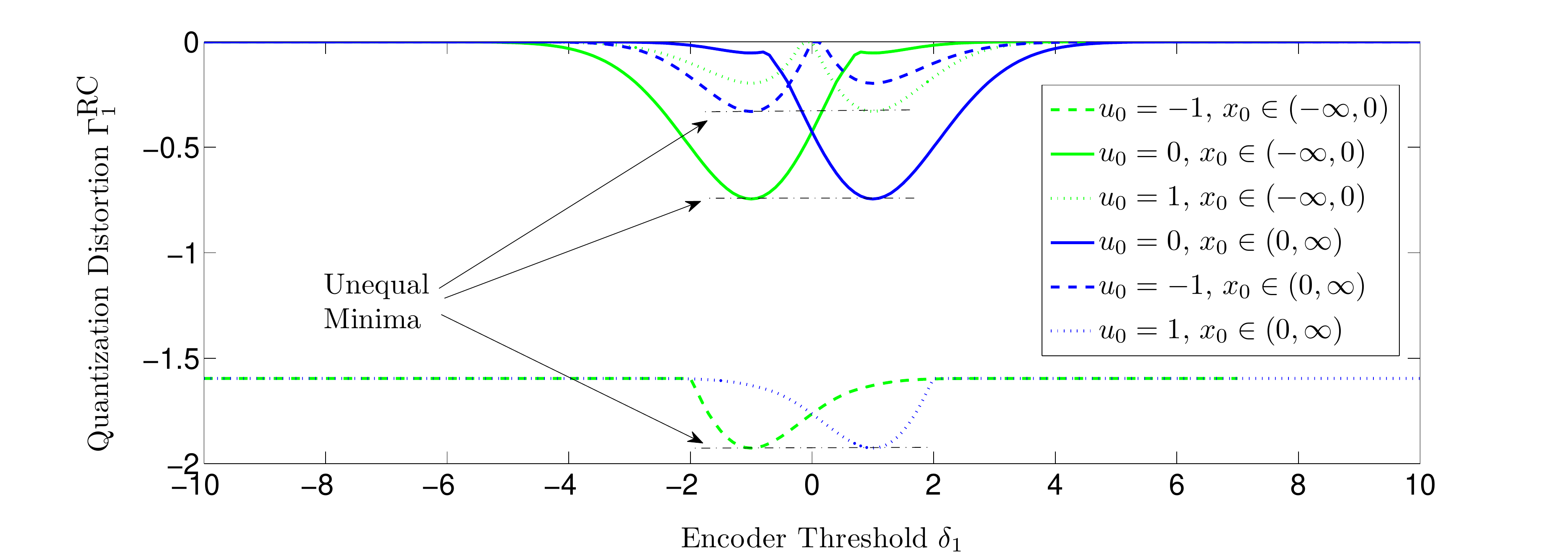}
\caption{This plot illustrates the lack of \emph{symmetry w.r.t.~translations} for $\Gamma^{\textrm{RC}}_1$, when the controls are restricted to be chosen from a discrete set $\mathcal{U}$, such as in Example~\ref{Ex:ConstrainedControl_NoSep}. Different values of $u_0$ do not result in the same minimum value for $\Gamma^{\textrm{RC}}_1$ over the range of $\delta_1$, thus resulting in the lack of separation and certainty equivalence. }
\label{Fig:RestrCtrl_NoSep}
\end{center}
\vspace{-5mm}
\end{figure}

We now present a slight variation in the restriction on the controller, and reconfirm that separation in design of controller and encoder is not optimal.
\begin{example} \label{Ex:ConstrainedIntervalControl_NoSep}
Consider the same setup as in Example~\ref{Ex:CertaintyEquivalence_EncDes}, with the restriction that the control signal be chosen from an interval $\mathcal{U} = (u_{\min},u_{\max})$. The quantizer thresholds $\delta_0$ and $\delta_1$ are to be chosen along with the control signals $u_0 \in \mathcal{U}$ and $u_1 \in \mathcal{U}$, to jointly minimize the two-step horizon control cost.

%\emph{Solution:}
%\begin{proof}[Illustration of Lack of separation]
As in the solution to the previous example, note that the unconstrained minimizer for the cost-to-go $V_1$ is the certainty equivalent value $u_1^{\textrm{CE}}$. The best we can do, given the constraint set $\mathcal{U}$, is to choose the control signal closest to the unconstrained value. This follows from the convexity of the quadratic cost-to-go. Using this principle, we find the optimal control signal $u_1^*$ to be
\begin{equation*}
u_1^* = \begin{cases}
u_{\min} & u_1^{\textrm{CE}} \le u_{\min} \; , \\
u_1^{\textrm{CE}} & u_{\min} \le u_1^{\textrm{CE}} \le u_{\max} \; , \\
u_{\max} & u_1^{\textrm{CE}} \ge u_{\max} \; .
\end{cases}
\end{equation*}
Evaluating the cost-to-go $V_1$ using $u_1^*$, and reusing quantities derived in Appendix~\ref{App:ExplicitCalc}, we can write up the cost-to-go $V_0$ as before. More interesting to us are the terms in this expression that directly depend on the choice of the quantizer threshold $\delta_1$, as given by
\begin{align*}
\Gamma^{\textrm{IC}}_1 = \mathbb{E} \bigg[(2a\widehat{x}_{1|1}u_{\min} + (q + 1)u_{\min}^2) \mathds{1}_{\left\{ \widehat{x}_{1|1} \ge -\frac{q+1}{a}u_{\min} \right\}} - \frac{a^2}{q+1} \widehat{x}_{1|1}^2 &\mathds{1}_{\left\{ -\frac{q+1}{a}u_{\max} \le \widehat{x}_{1|1} \le -\frac{q+1}{a}u_{\min} \right\}} \\
+ (2a\widehat{x}_{1|1}u_{\max} + (q + 1)u_{\max}^2) &\mathds{1}_{\left\{ \widehat{x}_{1|1} \le -\frac{q+1}{a}u_{\max} \right\}} \bigg| z_0 \bigg] \; .
\end{align*}

Evaluating this expression for parameters $a=1$, $p = 1$, $q = 1$, $u_{\min} = -2$ and $u_{\max} = 2$, and some arbitrary choice of quantizer threshold $\delta_0$, we plot $\Gamma^{\textrm{IC}}_1$ over a range of quantizer thresholds $\delta_1$, for different choices of $u_0$ from the set $\mathcal{U}$, in Figure~\ref{Fig:RestrIntCtrl_NoSep}. Notice that the minimum value of $\Gamma^{\textrm{IC}}_1$ obtained over the range of $\delta_1$ varies depending on the choice of $u_0$. Thus, there is no  symmetry w.r.t.~translations, and a separation in design is no longer optimal.
\end{example}

\begin{figure}[tb]
\begin{center}
\includegraphics*[scale=0.45,viewport=20 50 1500
550]{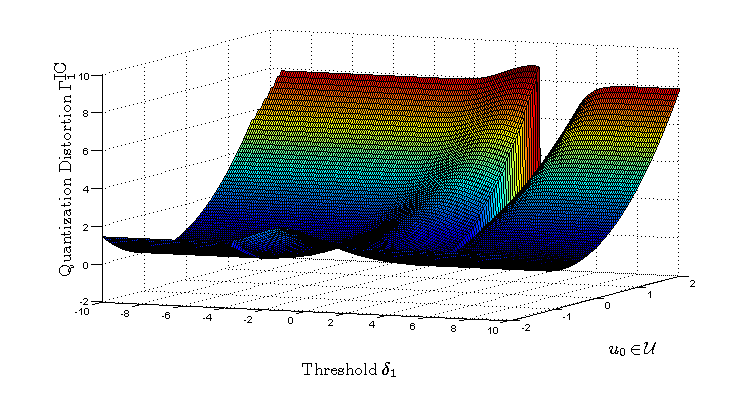}
\caption{This plot illustrates the lack of \emph{symmetry w.r.t.~translations} of $\Gamma^{\textrm{IC}}_1$, when the controls are restricted to be chosen from an interval, such as in Example~\ref{Ex:ConstrainedIntervalControl_NoSep}. Different values of $u_0$ do not result in the same minimum value for $\Gamma^{\textrm{IC}}_1$ over the range of $\delta_1$, thus resulting in lack of separation and certainty equivalence. }
\label{Fig:RestrIntCtrl_NoSep}
\end{center}
\vspace{-5mm}
\end{figure}

In both the above examples, the constrained set $\mathcal{U}$ did not contain the certainty equivalent values of the control signal $u_1$ for at least some values of $\delta_1$. The resulting cost-to-go $V_0$ was altered, such that the symmetry w.r.t.~translations was lost. Consequently, separation no longer holds. The restriction removed the certainty equivalence property during time step $k=1$, but the resulting cost and the information pattern resulted in the lack of separation itself at time step $k=0$. A similar problem setup has been explored in \cite{bernhardsson1999dualControlWithOnlyTwoGains}, where the control gain is restricted to be chosen from two given values. The dual effect has been shown for this problem setup as well.

%%%%%%%%%
%%%%%%%%%
\subsection{Zero order hold and event-triggered sampling}
%%%%%%%%%%%%%%%%%%%%%%%%%%%%%%%%%%%%%%%%%%%%%%%%%%%%%%%%%%%%%%%%%
%%%%%%%%%%%%%%%%%%%%%%%%%%%%%%%%%%%%
We study numerically two cases of control under event-triggered sampling.
Basically these are problems with a sampling budget of exactly one.
For the controller, we must design a whole waveform to be applied up to the
time when the first sample is received. We are already given the control law to be
applied from this random sampling time to the end time. For the encoder, we must
design an envelope to generate exactly one sample between time~$t=1$ and~$t=T$.

We study two examples, and in both of them, the encoder is allowed to be
dynamic. In the first example, the control waveform up to the first sample time is
pre-assigned, and it has a particular linear dependence on the Kalman predictor.
In the second example, the control waveform up to the first sample time must be a
zero order hold waveform.
%%%%%%%%%%%%%%%%%%%%%%%%%%%%%%%%%%%%%%%%%%%%%%%%%%%%%%%%%%%%%%%
\begin{example}[Fixed linear control law up to an event-triggered sample]
For the scalar linear plant~(\ref{plant}), let the coefficient~$a=1$,
and let the initial state~$x_0=2$, and  $\sigma_0=0$,
and let this information be known to the encoder and the controller.
This simply means that~$z_0=x_0.$ This information is prestored at the controller.
Let the variance $\sigma_w^2=0.5^2$.
Let the horizon end~$T=4$, and let $p = 1, q= 0.2$.
The control law is fixed to be:
\begin{align*}
     u_t
&
     =
       \begin{cases}
         k_t^* \, {\mathbb{E}} \left[ x_t \left|
                           x_0, \{u_i \}_0^{t-1} \right.
                         \right],
        & \text{for} \ 0 \leq t \leq \tau - 1,
        \\
         k_t^* \, {\mathbb{E}} \left[ x_t \left| x_{\tau},
              \{u_i \}_{\tau}^{t-1} \right. \right],
        & \text{for} \ \tau \leq t \leq T,
                \end{cases}
\end{align*}
where the gains~$k_t^*$ are the ones from the certainty equivalence
law~(\ref{optimalControlLaw}),
and~$\tau$ satisfies $1 \leq \tau \leq T$ and is the first and
only sample time, which is chosen by encoder.
Choose a policy (sampling envelope) which comprises silence sets
 $\left\{  {\mathcal{S}}_1 , \ldots , {\mathcal{S}}_T \right\}$
giving:
\begin{gather*}
\tau = \min \Bigl\{
               T, \,
               \min_{t \geq 1} \{t: x_t \notin {\mathcal{S}}_t\}
            \Bigr\}.
\end{gather*}
\label{example:etSamplingOneSample}
\end{example}
%%%%%%%%%%%%%%%%%%%%%%%%%%%%%%%%%%%%%%%%%%%%%%%%%%%%%%%%%%%%%%%
Next we consider an example of a design problem with a zero order hold control. Here we specialize
to the case where the control's hold epochs are forced to be exactly the
inter-sample intervals.
%%%%%%%%%%%%%%%%%%%%%%%%%%%%%%%%%%%%%%%%%%%%%%%%%%%%%%%%%%%%%%%
\begin{figure}
\centering
\subfloat[Optimal sampling envelope for
Example~\ref{example:etSamplingOneSample}]
{%
\includegraphics[width=0.45\textwidth]{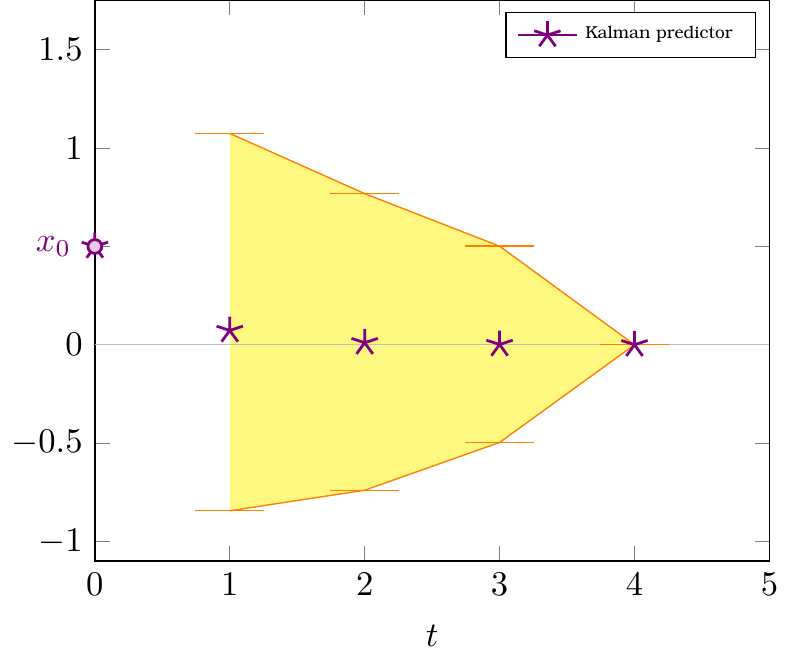}
\label{ZOH:withBestDeterministicControl}}
\subfloat[Optimal sampling envelope for
Example~\ref{example:ZOHetSamplingOneSample}]
{%
\includegraphics[width=0.45\textwidth]{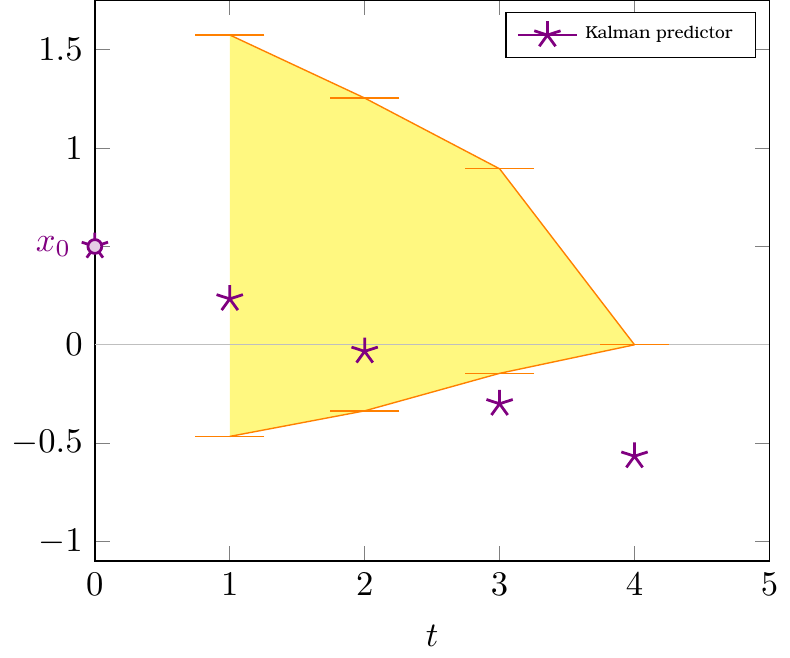}
\label{fig:withBestZOH}}
\caption{
Event-triggered sampling with exactly one sample after time~$t=0$
}
\label{fig:ZOHeventTriggered}
\end{figure}
%%%%%%%%%%%%%%%%%%%%%%%%%%%%%%%%%%%%%%%%%%%%%%%%%%%%%%%%%%%%%%%%
\begin{example}[zero order hold control up to an event-triggered sample]
% \label{Ex:nCE_ZOH_ET}
Consider the same setup as in Example~\ref{example:etSamplingOneSample}
but
% For the scalar linear plant~(\ref{plant}), let the coefficient~$a=1$,
% and let the initial state~$x_0=2$, and  $\sigma_0=0$,
% and let this information be known to the encoder and the controller.
% This simply means that~$z_0=x_0.$ This information is prestored at the controller.
% Let the variance $\sigma_w^2=0.5^2$.
% Let the horizon end~$T=4$, and let the coefficients~$p = 1, q= 0.2$.
there are exactly two epochs; and they must be precisely
$ \left\{ 0, 1, \ldots \tau-1 \right\}$ and $ \left\{ \tau, \ldots, T
\right\} $, where $\tau$ is the first and only sample time, and is chosen to occur
at or later than time~$t=1$. The control laws over the second epoch are fixed
to have the form:
$
    u_t
% &
    =
    k_i^* \,
    {\mathbb{E}} \left[ x_{\tau} \left| x_{\tau_{0}} \right. \right],
            \ \text{for} \ \tau \leq t \leq T,
$
where the gains~$k_t^*$ are the ones from the certainty equivalence
law~\eqref{optimalControlLaw}.
%\end{align*}
Pick: (1) a control law for the first epoch having the zero-order hold form:
\begin{align*}
     u_t
&
     =
     {\mathcal{K}}^0 \left( x_0 \right),
     \ \text{for} \ 0 \leq t \leq {\tau-1},
\end{align*}
and (2) a sampling envelope which comprises silence sets
 $\left\{  {\mathcal{S}}_1 , \ldots , {\mathcal{S}}_T \right\}$
for generating the sample time:
\begin{gather*}
\tau = \min \Bigl\{
               T, \,
               \min_{t \geq 1} \{t: x_t \notin {\mathcal{S}}_t\}
            \Bigr\} .
\end{gather*}
\label{example:ZOHetSamplingOneSample}
\end{example}
%%%%%%%%%%%%%%%%%%%%%%%%%%%%%%%%%%%%%%%%%%%%%%%%%%%%%%%%%%%%%%%
The optimal sampling envelope of the zero order hold control
example~(Example~\ref{example:ZOHetSamplingOneSample}) is shown in
Figure~\ref{fig:withBestZOH}. This is pictorial evidence that the dual effect
is present in the loop. This becomes clear from the reasoning below.

Supposing the dual effect were absent, then the encoder's goal would have been to
pick the sample time $\tau$ to minimize a weighted sum of squared estimation
errors up to time $\tau - 1$.  The envelope optimal for that objective will be a
sequence of silence set symmetric about the means
${\mathbb{E}} \left[ x_t \left| x_0, \{u_i\}_0^{\tau - 1}\right. \right]$.
When the plant noise is Gaussian, Hajek and
others~\cite{hajekPagingITW,hajekPagingJournal,lipsaMartins2011,nayyarAndThreeProfessorsTransactions2013}
predict that a symmetric sequence of silence is optimal. They also imply that a sequence
of silence sets that are not symmetric about the respective means
 ${\mathbb{E}} \left[ x_t \left| x_0, \{u_i\}_0^{\tau - 1}\right. \right]$ will
lead to suboptimal state estimation.

Since the optimal envelope computed numerically is clearly non-symmetric about
the means ${\mathbb{E}} \left[ x_t \left| x_0, \{u_i\}_0^{\tau - 1}
\right. \right]$,
there must be a dual effect in the loop, which is exploited by this optimal pair
of sampler and zero order hold controller.

%%%%%%%%%%%%%%%%%%%%%%%%%%%%%%%%%%%%%%%%%%%%%%%%%%%%%%%%%%%
%%%%%%%%%%%%%%%%%%%%%%%%%%%%%%%%%%%%%%%%%%%%%%%%%%%%%%%%%
%%%%%%%%%%%%%%%%%  Discussion of literature
%%%%%%%%%%%%%%%%%%%%%%%%%%%%%%%%%%%%%%%%%%%%%%%%%%%%%%%%%
\section{Conclusions%
\label{discussionSection}
}
%%%%%%%%%%%%%%%%%%%%%%%%%%%%%%%%%%%%%%%%
In this paper, we have seen through examples that the dual effect is present in the plant-encoder-channel combination. Hence in general, it is suboptimal to apply a controls-free encoder, or to apply an affine controller. It has long been known that for the design problem with a static encoder, separation is not optimal, and that the optimal control laws are nonlinear~\cite{curry1970book}. Recent interest in the dynamic design problem was due to Borkar and Mitter~\cite{borkarMitter1997} who describe advantages obtained by applying controls-forgetting encoders. Many papers state that the separated design is optimal for the dynamic design problem for the various channel models we have treated. We have shown by dynamic programming that these statements are indeed correct. This is an instance of the optimal decision policies `ignoring' the presence of the dual effect. But a separated design need not be optimal for other design problems. In particular, for event-triggered sampling the dynamic design problem has a separated design, but the zero order hold control design problem does not have a separated solution. This is at least partly surprising because, separated design is optimal for the classical LQG partially observed control with or without the zero order hold control restriction.

An interesting aspect of our results is that we have shown that
separation and certainty equivalence are optimal for Design problem $2$,
despite the dual effect being present in the networked control system of
Section~\ref{twoAgentProblemStatement}. To understand this result, we
now examine two implementations of the optimal encoder-controller pair for this design problem, and using these, we draw out some subtle points concerning dual effect and optimality of separation and certainty equivalence.
%%%%%%%%%%%%%%%%%%%%%%%%%%%%%%%%%%%%%%%%%%%%%%%%%%%%%%%%%%%
\begin{figure}  % [ht]
\centering
\includegraphics[width=0.75\textwidth]{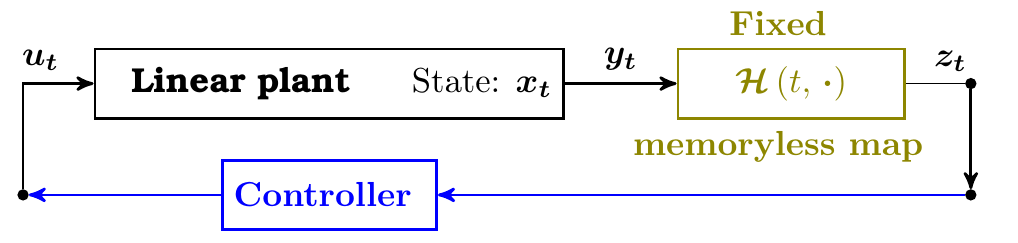}
\caption{Setup of Bar-Shalom and Tse}
\label{blockDiagramBarShalomTse}
\end{figure}
%%%%%%%%%%%%%%%%%%%%%%%%%%%%%%%%%%%%%%%%%%%%%%%%%%%%%%%%%%%
Bar-Shalom and Tse~\cite{barShalomtse1974dualEffectPaper}
consider the loop shown in Figure~\ref{blockDiagramBarShalomTse}.
At the sensor, instead of our dynamic encoder, they place a  nonlinear map.
This sensor map is time-varying but memoryless and its exact functional form
is given. For this setup, they have a result stating the mutual
exclusivity of the dual effect and optimality of certainty equivalence controls.
In their setting, if the linear `plant'  is such that the effect of controls is
never felt at the observation signal~$y_t$, then clearly there is no dual effect.
This happens in the case where the so-called `plant' has a sub-system that
produces the `plant' output after explicitly removing the effect of controls.

However, for our setup~(Figure~\ref{simpleLoop}), the sensor has a dynamic encoder
even after one performs the equivalence transformation by subtracting out the
effect of controls. The use of `innovation coding' leads to the closed loop shown in
Figure~\ref{blockDiagramEquivalentEncoder}. The crucial difference from the setup of
Bar-Shalom and Tse is that rather than being a memoryless nonlinear map, the
encoder $ \widetilde{\xi}_t $ is a dynamical system. Hence the Theorem
of Bar-Shalom and Tse does not apply. But it springs the following question:
Does the plant-sensor combination in the closed
loop of Figure~\ref{blockDiagramEquivalentEncoder} have a dual effect if an
encoder is used that is optimal for the dynamic
design problem~?
To answer this question, one needs to interpret carefully what it means to
implement an optimal encoder. For different interpretations, one gets
different answers.
%%%%%%%%%%%%%%%%%%%%%%%%%%%%%%%%%%%%%%%%%%%%%%%%%%%%%%%%%%%
%%%%%%%%%%%%%%%%%%%%%%%%%%%%%%%%%%%%%%%%%%%%%%%%%%%%%%%%%%%
\begin{figure}  % [ht]
\centering
\subfloat[No dual effect present because encoder is controls-forgetting
from time $0$]{%
\label{blockDiagram:noDualEffect}
\includegraphics[width=0.90\textwidth]{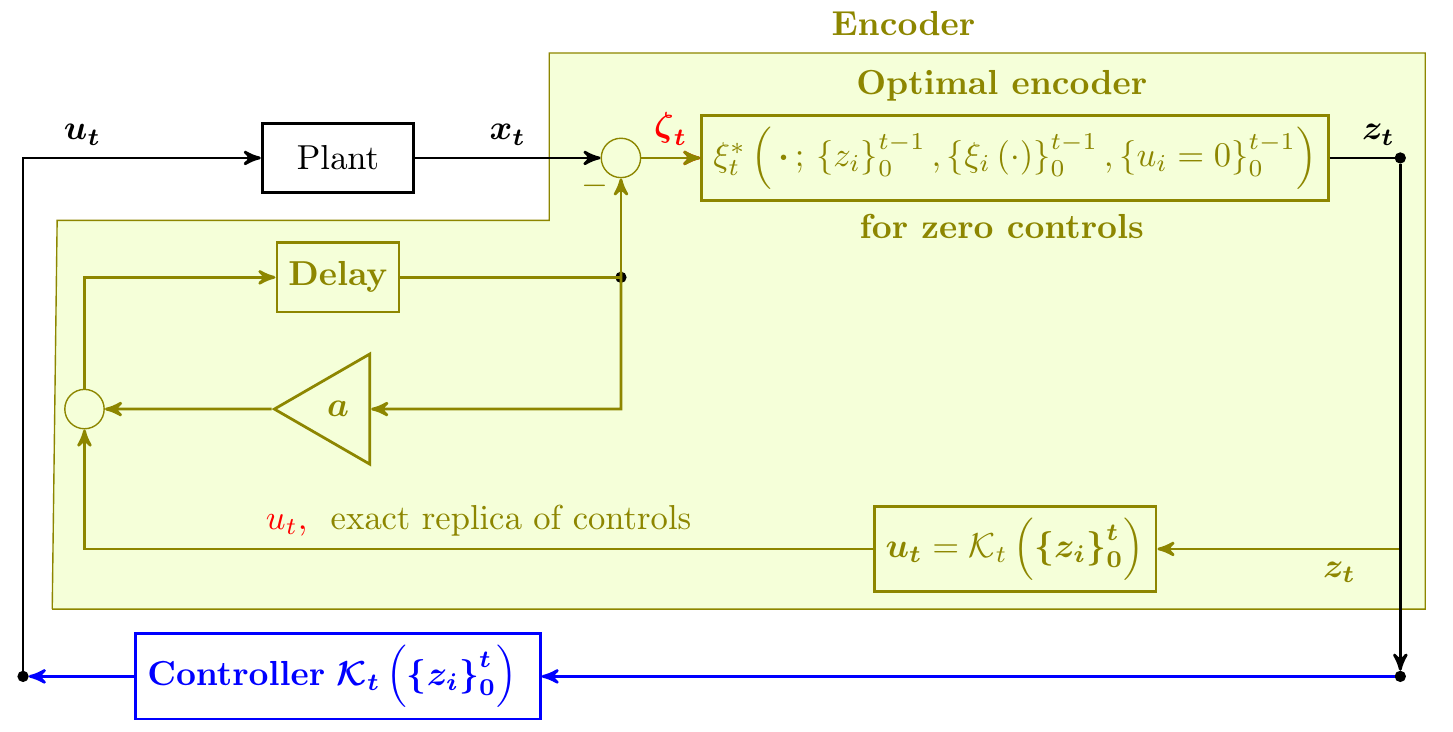}
}
\quad
%%%%%%%%%%%%%%%
\subfloat[Has dual effect because $\widetilde{\zeta}_t =
          \zeta_t +  {\sum_{i=0}^{t-1}{ a^{t-1-i}
                       \left\{   {\mathcal{K}}_i \left( z_0^i \right)
                               - {\mathcal{K}}_i^* \left( z_0^i \right)
                      \right\}  }}  \neq \zeta_t$ in general.]{%
\label{blockDiagram:withDualEffect}
\includegraphics[width=0.90\textwidth]{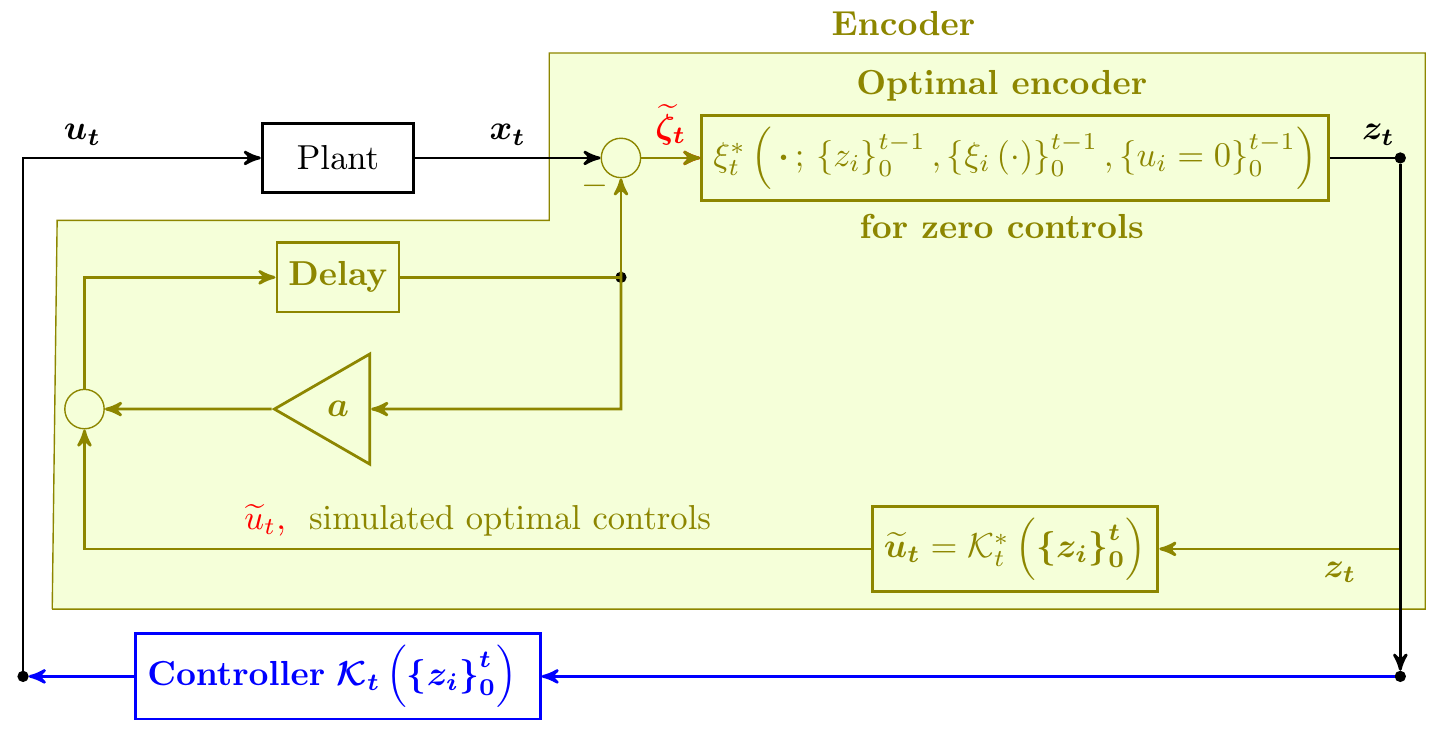}
}
\caption{Two interpretations of implementing an optimal
 encoder ${\mathcal{E}}^*$.
}
\end{figure}
%%%%%%%%%%%%%%%%%%%%%%%%%%%%%%%%%%%%%%%%%%%%%%%%%%%%%%%%%%%
%\subsection{Person by person optimality}
Assume that we are implementing the  feedback loop of Figure~\ref{simpleLoop}
with the optimal encoder and any admissible controller.

The first interpretation of
what it means to implement an optimal encoder, is the following:
  The encoder stores the actual set of control
policies used by the controller, and uses this
to carry out the innovation encoding, and on the
result applies the sequential quantizer
$
    \xi_t^*
        \left( \, \vec{\cdot}  \, ;  \,
               \left\{ z_i \right\}_0^t,
               \left\{ \xi_i  \left( \cdot \right)  \right\}_0^{t-1},
               \left\{ u_i = 0 \right\}_0^{t-1} %  = {0}_0^{t-1}
        \right)
$.
This is equivalent to the block diagram of Figure~\ref{blockDiagram:noDualEffect}.
No matter what the actual control policies are,
the controls have no influence
on the input to the sequential quantizer
$
    \xi_t^*
        \left( \, \vec{\cdot}  \, ;  \,
               \left\{ z_i \right\}_0^t,
               \left\{ \xi_i  \left( \cdot \right)  \right\}_0^{t-1},
               \left\{ u_i = 0 \right\}_0^{t-1} %  = {0}_0^{t-1}
        \right)
$.
Clearly, because of exact cancellation of controls,
the encoder implemented is controls-forgetting, and there is no dual effect in
the loop of Figure~\ref{blockDiagram:noDualEffect}.

The second interpretation is the following: The encoder does not pay attention
to the actual control policy being used. Instead, it assumes that the controller
is applying the certainty equivalence laws~(\ref{optimalControlLaw}).
It subtracts out the effect of the these certainty equivalence control laws.
To the residue~$\widetilde{\zeta}_{t}$, it applies the sequential
quantizer
$
    \xi_t^*
        \left( \, \vec{\cdot}  \, ;  \,
               \left\{ z_i \right\}_0^t,
               \left\{ \xi_i  \left( \cdot \right)  \right\}_0^{t-1},
               \left\{ u_i = 0 \right\}_0^{t-1} %  = {0}_0^{t-1}
        \right)
$.
Clearly this encoder is not controls-forgetting. But yet when used in
combination with the certainty equivalence laws of
(\ref{optimalControlLaw}), it leads to minimum performance cost.

On the other hand, when this encoder is used in combination with a general
admissible control law, there is potential mismatch between the encoder's
assumption and the actual controller behaviour. The effect of the controls
is not  absent in the input to the sequential quantizer
$
    \xi_t^*
        \left( \, \vec{\cdot}  \, ;  \,
               \left\{ z_i \right\}_0^t,
               \left\{ \xi_i  \left( \cdot \right)  \right\}_0^{t-1},
               \left\{ u_i = 0 \right\}_0^{t-1} %  = {0}_0^{t-1}
        \right)
$.
This situation is shown in Figure~\ref{blockDiagram:withDualEffect}.
Clearly, there is a dual effect in this loop.

This leads to an interesting consequence.
If a pair of encoding and control strategies is optimal,
then the individual strategies that are components
of the pair must be person-by-person optimal.
Since the combination
of certainty equivalence controls and the corresponding optimal encoder
is optimal, it follows that the certainty equivalence
controls must be optimal for the single-agent control
problem obtained by fixing the encoder to be the
optimal one. Since the second
interpretation of implementing the optimal encoder
is perfectly valid, it turns out that certainty equivalence
controls can be optimal even though the dual effect is
present in the loop. Thus we can conclude that
the Theorem of Bar-Shalom and Tse cannot generalize to
the scenario where sensors implement dynamic encoders.
%%%%%%%%%%%%%%%%%%%%%%%%%%%%%%%%%%%%%%%%%%%%%%%%%%%%%%%%%%%%%%%%%%%%%%%%
\subsection*{Acknowledgements}
Adam Molin first suggested to us that for event-triggered systems, the separated solution may not be optimal for the
zero order hold control design problem, and yet be optimal for the dynamic
encoder-controller design
problem.
% Toivo Henningsson and Karl J.~{\AA}str{\"o}m told us of the
% suitability of log-concave random variables to model signals arising
% in event-triggered sampling. To these colleagues, and
To Lei Bao, Mikael Skoglund, Henrik Sandberg, Ashutosh Nayyar,
John~S.~Baras, and Armand Makowski we are grateful for many discussions over the
last few years.
C.R. and K.H.J. greatfully acknowledge support from the Swedish Research Council
and the Knut and Alice Wallenberg Foundation.

%%%%%%%%%%%%%%%%%%%%%%%%%%%%%%%%%%%%%%%%%%%%%%%%%%%%%%%%%%%%%%%%%%%%%%%%%%%%%%%%%%%%%%
\bibliographystyle{siam}
% \bibliography{mabenRefs}

\def\polhk#1{\setbox0=\hbox{#1}{\ooalign{\hidewidth
  \lower1.5ex\hbox{`}\hidewidth\crcr\unhbox0}}} \def\cprime{$'$}

%%%%%%%%%%%%%%%%%%%%%%%%%%%%%%%%%%%%%%%%%%%%%%%%%%%%%%%%%%%%%%%%%%%%%%%%%%%%%%%%%%%%%%
\appendix

%\noindent
%{\large\bf{Appendix. Calculations for Example~\ref{example:predictionErrorDependence}}}
\section{Calculations for Example~\ref{example:predictionErrorDependence}} \label{App:ExplicitCalcEx3}

In Example~\ref{example:predictionErrorDependence}, we explicitly show the dependence of the second moments of~$\xbar{w}_0$ and $x_1 - {\widehat{x}}_{1\left|1\right.} $ on the applied controls when using a quantizer in its predictive form. Below, we show how to compute these terms. From the definition of $\xbar{w}_t$, we have:
\begin{align*}
\xbar{w}_0  & = {\mathbb{E}} \left[ { {x}}_{1}
                       \left| z_0, z_1 \right. \right]
               -  {\mathbb{E}} \left[ { {x}}_{1}
                       \left| z_0\right. \right] \\
& = {\mathbb{E}} \left[ { {x}}_{1}
                       \left| z_0, z_1 \right. \right]
               -  {\mathbb{E}} \left[ { {x}}_{0}
                       \left| z_0\right. \right]  - u_0.
\end{align*}
% \subsubsection*{Calculating $
% {\mathbb{E}} \left[ \, \xbar{w}_0 \left| z_0 = -1 \right. \right]
% $ }
We can find an expression for the term $\mathbb{E} \left[ { {x}}_{0} \left| z_0\right. \right]$, as shown below.
\begin{align*}
{\mathbb{P}} \left[ z_0 = -1 \right]
    & = \int_{-\infty}^0{ {\frac{1}{\sqrt{2\pi} \sigma_0 } }
                          {e^{- {\left(  \theta-\mu_0 \right) }^2 / {2 \sigma_0^2 }  }}  {d\theta}                         }
, \\
   & = {\frac{1}{2 }} \left[ 1 + {\rm{erf}} \left( {\frac{-\mu_0}{\sqrt{2}\sigma_0}} \right)  \right]
,
\end{align*}
\begin{align*}
{\mathbb{P}} \left[ z_0 = -1 \right] \times
  {\mathbb{E}} \left[ x_0 \left| z_0 = -1 \right. \right]
    & = \int_{-\infty}^0{ {\frac{1}{\sqrt{2\pi} \sigma_0 } }
                         \; \theta \, {e^{- {\left( \theta-\mu_0 \right) }^2 / {2 \sigma_0^2 }  }}  {d\theta}                         }
, \\
& = \mu_0 \times {\frac{1}{2 }} \left[ 1 +   {\rm{erf}} \left( {\frac{-\mu_0}{\sqrt{2}\sigma_0}} \right)  \right] - {\frac{\sigma_0}{\sqrt{2\pi} }
 \left[ e^{- {\mu_0}^2 /2\sigma_0^2{} } \right]}, \\
& = \mu_0 \times {\mathbb{P}} \left[ z_0 = -1 \right]
 - {\frac{\sigma_0}{\sqrt{2\pi} }
 \left[ e^{- {\mu_0}^2 /2\sigma_0^2{} } \right]}
\end{align*}
where $ {\rm{erf}} \left( x \right) \triangleq {\frac{2}{\sqrt{\pi} } }
\int_{0}^x{ e^{-t^2} {dt} }.$
We also have:
\begin{align}
{\mathbb{P}} \left[ z_1 = -1 ,\, z_0 = -1 \right]
    & = \int_{-\infty}^0\int_{-\infty}^0{ {\frac{1}{2\pi\sigma_0\sigma_w}}
        {e^{- {\left(  r-\mu_0 \right) }^2 / {2 \sigma_0^2 }  }}
        \;
       {e^{- {\left(  s-r-u_0 \right) }^2 / {2 \sigma_w^2 }  }}
        \, {dr} {ds}          }
, \notag \\
{\mathbb{P}} \left[ z_1 = -1 \left| z_0 = -1 \right. \right] & =
{\frac{{\mathbb{P}} \left[ z_1 = -1 ,\, z_0 = -1 \right] }
{ {\mathbb{P}} {\left[ z_0 = -1 \right] } }
}
. \notag
\end{align}
We can now find an expression for the terms ${\widehat{x}}_{1\left|1\right.} = \mathbb{E} \left[ { {x}}_{1} \left| z_0, z_1 \right. \right]$ and ${\mathbb{E}} \left[ x_1^2 \left| z_0 = -1 ,\, z_1 = -1 \right.  \right]$ as
\begin{align}
{\mathbb{E}} \left[ x_1 \left| z_0 = -1 ,\, z_1 = -1 \right.  \right]
    & = {\frac{D_1(-\infty,0)}{{\mathbb{P}} \left[ z_1 = -1, \, z_0 = -1 \right]}} \label{Eq:condX1} \\
    D_1(-\infty,0) &= \int_{-\infty}^0\int_{-\infty}^0{ {\frac{1}{2\pi\sigma_0\sigma_w}}
       \, s \,
       {e^{- {\left(  r-\mu_0 \right) }^2 / {2 \sigma_0^2 }  }}
        \;
       {e^{- {\left(  s-r-u_0 \right) }^2 / {2 \sigma_w^2 }  }}
        \, {dr} {ds}          } \notag \\
{\mathbb{E}} \left[ x_1^2 \left| z_0 = -1 ,\, z_1 = -1 \right.  \right]
    & = {\frac{D_2(-\infty,0)}{{\mathbb{P}} \left[ z_1 = -1, \, z_0 = -1 \right]}} \label{Eq:conditionalSecondMonet} \\
    D_2(-\infty,0) &= \int_{-\infty}^0\int_{-\infty}^0{ {\frac{1}{2\pi\sigma_0\sigma_w}}
       \, s^2 \,
       {e^{- {\left(  r-\mu_0 \right) }^2 / {2 \sigma_0^2 }  }}
        \;
       {e^{- {\left(  s-r-u_0 \right) }^2 / {2 \sigma_w^2 }  }}
        \, {dr} {ds}          } \notag
,
%& =
%,
\end{align}
with similar expressions for the event $z_0 = -1 \, {\rm{AND}} \, z_1 = +1$. We then can compute:
\begin{align*}
{\mathbb{E}} \left[ \xbar{w}_0^2\left| z_0 = -1 \right.  \right]
 & = {\mathbb{P}} \left[ z_1 = -1 \left| z_0 = -1 \right. \right]
     \times
    {\Bigl(
    {\mathbb{E}} \left[ x_1 \left| z_1 = -1 ,\, z_0 = -1 \right.  \right]
    -  {\mathbb{E}} \left[ x_0 \left| z_0 = -1 \right.  \right] - \alpha
     \Bigr)}^2
\\
 & \,  + {\mathbb{P}} \left[ z_1 = +1 \left| z_0 = -1 \right. \right]
     \times
    {\Bigl(
    {\mathbb{E}} \left[ x_1 \left| z_1 = +1 ,\, z_0 = -1 \right.  \right]
    -  {\mathbb{E}} \left[ x_0 \left| z_0 = -1 \right.  \right] - \alpha
     \Bigr)}^2,
\\
 & = {\mathbb{P}} \left[ z_1 = -1 \left| z_0 = -1 \right. \right]
     \times
    {\Bigl(
    {\mathbb{E}} \left[ x_1 \left| z_1 = -1 ,\, z_0 = -1 \right.  \right]
     \Bigr)}^2
%    -  {\Bigl(
%       \alpha + {\mathbb{E}} \left[ x_0 \left| z_0 = -1 \right.  \right]
%        \Bigr)}^2,
\\
 & \, +  {\mathbb{P}} \left[ z_1 = +1 \left| z_0 = -1 \right. \right]
     \times
    {\Bigl(
    {\mathbb{E}} \left[ x_1 \left| z_1 = +1 ,\, z_0 = -1 \right.  \right]
     \Bigr)}^2
    -  {\Bigl(
       \alpha + {\mathbb{E}} \left[ x_0 \left| z_0 = -1 \right.  \right]
        \Bigr)}^2.
\end{align*}
The events corresponding to $z_0 = +1$, with $z_1 = -1$ and $z_1 = +1$
result in similar expressions. Thus: %, we have
\begin{align*}
{\mathbb{E}} \left[ \xbar{w}_0^2\left| z_0 = +1 \right.  \right]
 & = {\mathbb{P}} \left[ z_1 = -1 \left| z_0 = +1 \right. \right]
     \times
    {\Bigl(
    {\mathbb{E}} \left[ x_1 \left| z_1 = -1 ,\, z_0 = +1 \right.  \right]
    -  {\mathbb{E}} \left[ x_0 \left| z_0 = +1 \right.  \right] - \beta
     \Bigr)}^2
\\
 & \,  + {\mathbb{P}} \left[ z_1 = +1 \left| z_0 = +1 \right. \right]
     \times
    {\Bigl(
    {\mathbb{E}} \left[ x_1 \left| z_1 = +1 ,\, z_0 = +1 \right.  \right]
    -  {\mathbb{E}} \left[ x_0 \left| z_0 = +1 \right.  \right] - \beta
     \Bigr)}^2.
\\
 & = {\mathbb{P}} \left[ z_1 = +1 \left| z_0 = +1 \right. \right]
     \times
    {\Bigl(
    {\mathbb{E}} \left[ x_1 \left| z_1 = +1 ,\, z_0 = +1 \right.  \right]
     \Bigr)}^2
%    -  {\Bigl(
%       \beta + {\mathbb{E}} \left[ x_0 \left| z_0 = +1 \right.  \right]
%        \Bigr)}^2,
\\
 & \, + {\mathbb{P}} \left[ z_1 = -1 \left| z_0 = +1 \right. \right]
     \times
    {\Bigl(
    {\mathbb{E}} \left[ x_1 \left| z_1 = -1 ,\, z_0 = +1 \right.  \right]
     \Bigr)}^2
    -  {\Bigl(
       \beta + {\mathbb{E}} \left[ x_0 \left| z_0 = +1 \right.  \right]
        \Bigr)}^2.
\end{align*}
We also have:
\begin{align*}
\textrm{Var}_{t\left|t\right.}^{\text{err}}
 & = {\mathbb{E}} \left[ \left ( x_1 - {\widehat{x}}_{1\left|1\right.} \right )^2 \left| z_0 = -1
\right.  \right] \\
 & = {\mathbb{P}} \left[ z_1 = -1 \left| z_0 = -1 \right. \right]
     \times
\left\{
    {\mathbb{E}} \left[ x_1^2 \left| z_1 = -1 ,\, z_0 = -1
                          \right. \right]
  -
    {\Bigl(
    {\mathbb{E}} \left[ x_1 \left| z_1 = -1 ,\, z_0 = -1 \right.  \right]
%    -  {\mathbb{E}} \left[ x_0 \left| z_0 = -1 \right.  \right] - \beta
     \Bigr)}^2
\right\}
\\
 & \,  + {\mathbb{P}} \left[ z_1 = +1 \left| z_0 = -1 \right. \right]
     \times
\left\{
    {\mathbb{E}} \left[ x_1^2 \left| z_1 = +1 ,\, z_0 = -1
                          \right. \right]
  -
    {\Bigl(
    {\mathbb{E}} \left[ x_1 \left| z_1 = +1 ,\, z_0 = -1 \right.  \right]
%    -  {\mathbb{E}} \left[ x_0 \left| z_0 = -1 \right.  \right] - \beta
     \Bigr)}^2
\right\} .
\end{align*}
%{\large\bf{Appendix B. Reduction of a double integral}}
In Figure~\ref{graphsForExamplePredictionError}, ${\mathbb{E}} \left[
\xbar{w}_0^2\left| z_0 = -1 \right.  \right]$ and $\textrm{Var}_{t\left|t\right.}^{\text{err}}$ are plotted against
$\alpha$ to illustrate the presence of a dual effect.

Next, we consider how to compute the conditional expectation~\eqref{Eq:condX1} above.
Consider the scalar linear system:
\begin{align*}
x_{t+1} & = a\,x_t + u_t + w_t, \ t \in \left\{  0, 1, \ldots \right\},
\end{align*}
where $x_0 \sim {\mathcal{N}} \left( \mu_0, \sigma_0 \right),$ and the
noise process $w_t$ is IID with distribution~ ${\mathcal{N}} \left(
0, \sigma_w \right),$ and is independent of current and past controls and
states. Suppose that the control~$u_t$ is causally
computed on the basis of a sequence of quantized outputs:~$z_t$.
Consider the quantizer~$Q_t\left(\cdot\right)$:
%Consider a quantizer~$Q_t\left(\cdot\right)$ with interval
%code-cells for generating~$z_t$: %. Then one can say that~$Q_t\left(\cdot\right)$
%will have the form:
\begin{align*}
Q_t(x_t) & = i, \ {\text{if}} \ x_i\in\left(\theta_{i-1} , \theta_{i}\right),
\end{align*}
where the~$\theta_i$s are measurable with respect to the
partial sequence~$z_0^{t-1}.$ Under this setup, the conditional expectation of the state at time $1$ in~\eqref{Eq:condX1} can be expressed as
\begin{align*}
\widehat{x}_{1|1} &= {\mathbb{E}} \left[ x_1  \left| x_0 \in \left(\theta_{i-1} ,
\theta_{i}\right), x_1 \in \left(\varpi_{j-1} ,
 \varpi_{j}\right)  \right.  \right],
\end{align*}
where for notational clarity, we have denoted the quantization
levels at time~$0$ with the letter~$\theta,$ and the levels at time~$1$ with~$\eta.$ Computing $\widehat{x}_{1|1}$ requires
computation of the following double integral from~\eqref{Eq:condX1}:
\begin{align*}
D_1 &= \frac{1}{2 \pi \sigma_0 \sigma_w}
        \int_{\varpi_{j-1}}^{\varpi_{j}}{
       \int_{\theta_{i-1}}^{\theta_i}{ s \,
                e^{ - \frac{1}{2} \left(  \frac{r^2}{\sigma_0^2}
                - 2r \frac{\mu_0}{\sigma_0^2}
               + \frac{\mu_0^2}{\sigma_0^2}
               + \frac{(s-u_0)^2}{\sigma_w^2}
               + \frac{a^2r^2}{\sigma_w^2}
             - 2r \frac{a(s-u_0)}{\sigma_w^2}
         \right)}
          \ dr }
          \ ds }.
\end{align*}
Let $\bar{\sigma}$ and $\bar{\mu}$ be defined as follows:
\begin{align*}
\frac{1}{\bar{\sigma}^2} &= \frac{1}{\sigma^2_0} +
\frac{a^2}{\sigma_w^2}, \\
\bar{\mu_s} &= \bar{\sigma}^2 \left(\frac{\mu_0}{\sigma_0^2} +
\frac{a(s-u_0)}{\sigma_w^2} \right).
\end{align*}
Denote by $g(r)$
the standard Gaussian probability density $e^{-{r^2}/2} / {\sqrt{2
\pi}} $, and by $G(r)$, its CDF $\int_{- \infty}^r g(s) ds$.
Then, we can write:
\begin{align}
D_1 &= \frac{\bar{\sigma}}{\sqrt{2 \pi} \sigma_0 \sigma_w} \ e ^{-
\frac{\mu_0^2}{2 \sigma^2_0}}
\int_{\varpi_{j-1}}^{\varpi_{j}}  s \ e^{- \frac{(s-u_0)^2}{2
\sigma_w^2} + \frac{\bar{\sigma}^2}{2} \left(\frac{\mu_0}{\sigma_0^2} +
\frac{a(s-u_0)}{\sigma_w^2} \right)^2}
\int_{\theta_{i-1}}^{\theta_i} \frac{1}{\sqrt{2 \pi} \bar{\sigma}}
\ e ^{- \frac{1}{2 \bar{\sigma}^2} (r - \bar{\mu}_s)^2} dr \ ds, \notag \\
 &
= (a{\mu_0} + u_0)
   \cdot \mathbb{P} \left[ x_0 \in \left( \theta_{i-1}, \theta_i
\right),
                                   x_1 \in \left(\varpi_{j-1}, \varpi_{j}
\right)
                          \right] \notag \\
& \ \ \ + \sqrt{a^2\sigma_0^2 + \sigma_w^2} \cdot \int_{ \frac{\varpi_{j-1} - a{\mu_0} - u_0}
{\sqrt{a^2\sigma_0^2 + \sigma_w^2}}}^{
                        \frac{\varpi_{j} - a{\mu_0} -u_0 }
{\sqrt{a^2\sigma_0^2 + \sigma_w^2}}}
\ \tilde{s} \ g \left( \tilde{s} \right ) \
%\frac{1}{\sqrt{2\pi}} e^{- \frac{1}{2} \tilde{s}^2}
\left\{
G \left(\xbar{A}  - B \tilde{s} \right)
- G \left( \underline{A} - B \tilde{s} \right)
\right\} \ d{\tilde{s}}, \label{Eq:IntegralD1}
\end{align}
where $\xbar{A} = (\theta_{i} -  \mu_0) / \bar{\sigma}$, $\underline{A} = (\theta_{i-1} -  \mu_0) / \bar{\sigma} $ and $B = a \sigma_0 / \sigma_w $.
The first term above can be calculated using established routines for
calculating multivariable normal probabilities~(MATLAB command
{\texttt{mvncdf}}).
The second term can be explicitly calculated using the following indefinite
integral found in Owen's table of Normal integrals~\cite{owen1980tableOfNormalIntegrals}:
%Then, the indefinite integral
\begin{align*}
\int x \ g(x) \ G (A - B x) dx = &-\frac{B}{\sqrt{1+
B^2}} \ g \left(\frac{A}{\sqrt{1+ B^2}} \right) \ G \left(x
\sqrt{1+ B^2} - \frac{A B}{\sqrt{1+ B^2}} \right) \\
&- G (A - B x) \ g(x).
\end{align*}
Finally, the conditional expectation in~\eqref{Eq:conditionalSecondMonet}
requires evaluating an integral of the form:
\begin{align*}
D_2 & \triangleq \frac{1}{2 \pi \sigma_0 \sigma_w}
        \int_{\varpi_{j-1}}^{\varpi_{j}}{
       \int_{\theta_{i-1}}^{\theta_i}{ \ s^2 \,
                e^{ - \frac{1}{2} \left(  \frac{r^2}{\sigma_0^2}
                - 2r \frac{\mu_0}{\sigma_0^2}
               + \frac{\mu_0^2}{\sigma_0^2}
               + \frac{(s-u_0)^2}{\sigma_w^2}
               + \frac{a^2r^2}{\sigma_w^2}
             - 2r \frac{a(s-u_0)}{\sigma_w^2}
         \right)}
          \ dr }
          \ ds }, \\
 & =
 \widetilde{\sigma}^2 \int_{ \frac{\varpi_{j-1} - a{\mu_0} - u_0}
{\sqrt{a^2\sigma_0^2 + \sigma_w^2}}}^{
                        \frac{\varpi_{j} - a{\mu_0} -u_0 }
{\sqrt{a^2\sigma_0^2 + \sigma_w^2}}}
\ \tilde{s}^2 \ g \left( \tilde{s} \right ) \
%\frac{1}{\sqrt{2\pi}} e^{- \frac{1}{2} \tilde{s}^2}
\left\{
G \left(\xbar{A}  - B \tilde{s} \right)
- G \left( \underline{A} - B \tilde{s} \right)
\right\} \ d{\tilde{s}}
\\
 & +
2  \widetilde{\sigma} \left( u + a\mu_0\right)  \int_{ \frac{\varpi_{j-1} - a{\mu_0} - u_0}
{\sqrt{a^2\sigma_0^2 + \sigma_w^2}}}^{
                        \frac{\varpi_{j} - a{\mu_0} -u_0 }
{\sqrt{a^2\sigma_0^2 + \sigma_w^2}}}
\ \tilde{s} \ g \left( \tilde{s} \right ) \
%\frac{1}{\sqrt{2\pi}} e^{- \frac{1}{2} \tilde{s}^2}
\left\{
G \left(\xbar{A}  - B \tilde{s} \right)
- G \left( \underline{A} - B \tilde{s} \right)
\right\} \ d{\tilde{s}} \\
 & +
 {\left( u + a\mu_0\right)}^2  \int_{ \frac{\varpi_{j-1} - a{\mu_0} - u_0}
{\sqrt{a^2\sigma_0^2 + \sigma_w^2}}}^{
                        \frac{\varpi_{j} - a{\mu_0} -u_0 }
{\sqrt{a^2\sigma_0^2 + \sigma_w^2}}}
\ g \left( \tilde{s} \right ) \
%\frac{1}{\sqrt{2\pi}} e^{- \frac{1}{2} \tilde{s}^2}
\left\{
G \left(\xbar{A}  - B \tilde{s} \right)
- G \left( \underline{A} - B \tilde{s} \right)
\right\} \ d{\tilde{s}} ,
\end{align*}
where $\widetilde{\sigma} \triangleq \sqrt{a^2\sigma_0^2 + \sigma_w^2}$.
Let
$
 h(x) = x g(x) \left\{ G\left( \xbar{A} - B x \right) -
                         G\left( \underline{A} - B x \right) \right\}.
$
Then,
\begin{align*}
 {d\left[h(x)\right]} & =
      g(x) \left\{ G\left( \xbar{A} - B x \right) -
                   G\left( \underline{A} - B x \right) \right\} \, {dx} \\
& -
     x^2  g(x) \left\{ G\left( \xbar{A} - B x \right) -
                       G\left( \underline{A} - B x \right) \right\} \, {dx} \\
& -
   B  x  g(x) \left\{ g\left( \xbar{A} - B x \right) -
                        g\left( \underline{A} - B x \right) \right\} \, {dx}.
\end{align*}
Hence, the first term of $D_2$ is
\begin{align*}
D_3 & \triangleq
 \widetilde{\sigma}^2 \int_{ \frac{\varpi_{j-1} - a{\mu_0} - u_0}
{\sqrt{a^2\sigma_0^2 + \sigma_w^2}}}^{
                        \frac{\varpi_{j} - a{\mu_0} -u_0 }
{\sqrt{a^2\sigma_0^2 + \sigma_w^2}}}
\ \tilde{s}^2 \ g \left( \tilde{s} \right ) \
%\frac{1}{\sqrt{2\pi}} e^{- \frac{1}{2} \tilde{s}^2}
\left\{
G \left(\xbar{A}  - B \tilde{s} \right)
- G \left( \underline{A} - B \tilde{s} \right)
\right\} \ d{\tilde{s}}
\\
& =
\widetilde{\sigma}^2
{\left[h(x)\right]}^{\widetilde{\eta}_j}_{\widetilde{\eta}_{j-1}}
 + \widetilde{\sigma}^2 {\mathbb{P}} \left[ x_0 \in \left( \theta_{i-1}, \theta_{i} \right), x_1 \in \left( \varpi_{j-1} , \varpi_{j} \right) \right]
\\
&{} - \widetilde{\sigma}^2 B \int^{\widetilde{\eta}_j}_{\widetilde{\eta}_{j-1}}
   \ \ x \ g(x) \left\{ g\left( \xbar{A} - B x \right) -
                        g\left( \underline{A} - B x \right) \right\} \, {dx},
\end{align*}
where $\widetilde{\eta}_{l} = (\eta_l - a{\mu_0} -u_0 ) / \sqrt{a^2\sigma_0^2 + \sigma_w^2}$. To calculate the last integral, we can use the result found in Owen's table of Normal integrals~\cite{owen1980tableOfNormalIntegrals}:
\begin{align*}
%\int \ x {\frac{e^{-x^2/2}}{\sqrt{2\pi}}} {\frac{e^{-{(A-Bx)}^2/2}}{\sqrt{2\pi}}} \ dx \
\int \ x \ g(x) \ g(A-Bx) \ dx \
& =
- {\frac{1}{1+B^2}} \ g \left ( A/\sqrt{1+B^2} \right ) \ g \left( x \sqrt{1+B^2} - AB/\sqrt{1+B^2} \right )
%- {\frac{1}{2\pi}} {\frac{1}{1+B^2}} e^{-A^2/2(1+B^2)} \times
%e^{-{\left( x \sqrt{1 + B^2} - AB/\sqrt{1 + B^2} \right)}^2/2}
\\
&
+ {\frac{AB}{(1+B^2)^{3/2}}} \ g \left ( A/\sqrt{1+B^2} \right ) \ G \left( x \sqrt{1+B^2} - AB/\sqrt{1+B^2} \right ) .
%e^{-A^2/2(1+B^2)} \times
%G\left( x \sqrt{1 + B^2} - AB/\sqrt{1 + B^2} \right) .
%+ {\frac{1}{\sqrt{2\pi}}} {\frac{1}{1+B^2}} {\frac{AB}{\sqrt{1+B^2}}}
%e^{-A^2/2(1+B^2)} \times
%G\left( x \sqrt{1 + B^2} - AB/\sqrt{1 + B^2} \right) .
\end{align*}

By using the above expressions in~\eqref{Eq:condX1} and~\eqref{Eq:conditionalSecondMonet}, we can evaluate the terms ${\mathbb{E}} \left[
\xbar{w}_0^2\left| z_0 = -1 \right.  \right]$ and $\textrm{Var}_{t\left|t\right.}^{\text{err}}$ for different values of $\alpha$, and plot its dependence in Figure~\ref{graphsForExamplePredictionError}.

\section{Evaluating the cost-to-go $V_0$ in Example~\ref{Ex:CertaintyEquivalence_EncDes}} \label{App:ExplicitCalc}
In Example~\ref{Ex:CertaintyEquivalence_EncDes}, to compute the cost-to-go $V_0$, we must be able to compute the term $\mathbb{E} \left[ \widehat{x}_{1|1}^2 {}\big|{} z_0 \right]$. Using the expression for $\widehat{x}_{1|1}$ (see~\eqref{Eq:condX1}) in Appendix~\ref{App:ExplicitCalcEx3},
%The estimate $\widehat{x}_{1|1}$ is evaluated using the knowledge that $x_0$ lies in the cell $(\theta_{l-1},\theta_l)$, for $z_0 = l$, and $x_1$ lies in the cell $(\varpi_{j-1},\varpi_{j})$, for $z_1 = j$, respectively. The estimate can then be found as
%\begin{align*}
%\mathbb{E} \left[ x_1 \big| \left\{ z_i \right\}_0^1 \right] &= \int x_1 \mathbb{P}\left( x_1 \; \big| \; x_0 \in (\theta_{l-1},\theta_l), x_1 \in (\varpi_{j-1},\varpi_{j}) \right) dx_1 \\
%&= \int_{\varpi_{j-1}}^{\varpi_{j}} \int_{\theta_{l-1}}^{\theta_l} x_1 \frac{\mathbb{P}\left( x_1,x_0 \right)}{\mathbb{P}\left(x_0 \in (\theta_{l-1},\theta_l), x_1 \in (\varpi_{j-1},\varpi_{j}) \right)} dx_0 dx_1 \; .
%\end{align*}
%Then,
the desired quantity $\mathbb{E} \left[ \widehat{x}_{1|1}^2 {}\big|{} z_0 \right]$ can be written as
\begin{align*}
\mathbb{E} \left[ \widehat{x}_{1|1}^2 {}\big|{} z_0 \right] &= \sum_{j=1}^N \mathbb{P}\left( x_1 \in (\varpi_{j-1},\varpi_{j}) \; \big| \; x_0 \in (\theta_{i-1},\theta_i) \right) \cdot \left(\mathbb{E} \left[ x_1 {}\big|{} z_0 = l, z_1 = j \right]\right)^2 \\
&= \frac{1}{\mathbb{P}\left(x_0 \in \left(\theta_{i-1},\theta_i \right) \right)} \sum_{j=1}^N \frac{ D_1 }{ \mathbb{P}\left(x_0 \in (\theta_{i-1},\theta_i), x_1 \in (\varpi_{j-1},\varpi_{j}) \right)} \; .
%\frac{ \left( \int_{\varpi_{j-1}}^{\varpi_{j}} \int_{\theta_{l-1}}^{\theta_l} x_1 \mathbb{P}\left( x_1,x_0 \right) dx_0 dx_1 \right)^2}{ \mathbb{P}\left(x_0 \in (\theta_{l-1},\theta_l), x_1 \in (\varpi_{j-1},\varpi_{j}) \right)} \; .
\end{align*}
In the above expression, note that $D_1$ depends on the quantizer cell $(\varpi_{j-1},\varpi_{j})$ and has been evaluated in~\eqref{Eq:IntegralD1} in Appendix~\ref{App:ExplicitCalcEx3}. Also, the second term in~\eqref{Eq:IntegralD1} has been denoted by $\vartheta(\underline{r},\bar{r})$ in~\eqref{Eq:lambda} of Example~\ref{Ex:CertaintyEquivalence_EncDes}.
By setting $\mu_0 = 0$ as per Example~\ref{Ex:CertaintyEquivalence_EncDes} in the expression for $D_1$ above, the cost-to-go to be minimized can be rewritten as in~\eqref{Eq:CostToGoV0_Rewritten}.

%%%%%%%%%%%%%%%%%%%%%%%%%%%%%%%%%%%%%%%%%%%%%%%%%%%%%%%%%%%%%%%%%%%%%%%%%%%%%%%%%%%%%%

\end{document}